\documentclass{amsart}
\usepackage[USenglish,activeacute]{babel}
\usepackage[T1]{fontenc}
\usepackage{tikz-cd}
\usepackage[utf8]{inputenc}
\usepackage{fancyhdr}
\usepackage{yhmath}
\usepackage{mathrsfs}
\usepackage{graphicx}
\usepackage{wrapfig}
\usepackage{caption}
\usepackage{dsfont}
\usepackage[normalem]{ulem}  
\usepackage{enumitem}   
\usepackage{cutwin}
\usepackage{amsfonts}
\usepackage{mathtools}
\usepackage{subcaption}
\usepackage[colorlinks=true]{hyperref}
\usepackage{multirow}
\usepackage{amsmath}
\usepackage{amssymb}
\usepackage{accents}
\usepackage{amsthm}
\usepackage{textcomp}
\usepackage{adjustbox}
\newtheorem*{not*}{{ Notation}}
\newtheorem{defi}{ Definition}[subsection]
\newtheorem*{defi*}{ Definition}
\newtheorem{teo}[defi]{ Theorem}
\newtheorem*{teo*}{{ Theorem}}
\newtheorem{intteo}{Theorem}

\newtheorem{prop}[defi]{ Proposition}
\newtheorem*{prop*}{{ Proposition}}
\newtheorem{obs}[defi]{{Remark}}

\newtheorem*{Lemma*}{Lemma}
\newtheorem{Lemma}[defi]{Lemma}
\newtheorem*{coro*}{Corollary}
\newtheorem{coro}[defi]{Corollary}

\newcommand{\tarc}{\mbox{\large$\frown$}}
\newcommand{\arc}[2][-3ex]{{#2}{\kern #1{\raisebox{1.5ex}{\tarc}}}}

\newcommand{\Sp}{\operatorname{Sp}}

\newcommand{\Spf}{\operatorname{Spf}}

\newcommand{\Hom}{\operatorname{Hom}}

\newcommand{\Mod}{\operatorname{Mod}}
\newcommand{\D}{\mathcal{D}}
\newcommand{\OX}{\mathcal{O}}

\newcommand{\Spec}{\operatorname{Spec}}
\newcommand{\gr}{\operatorname{gr}}

\title{\texorpdfstring{Category $\OX$ for $p$-adic rational Cherednik algebras}{}}
\author{Fernando Peña Vázquez}
\email{fpvmath@gmail.com}
\address{Mathematisch-Naturwissenschaftliche Fakult\"at der Humboldt-Universit\"at zu Berlin, Rudower Chaussee 25, 12489 Berlin, Germany}
\begin{document}
\begin{abstract}
We introduce the concept of a triangular decomposition for  Banach and Fréchet-Stein algebras over $p$-adic fields, which allows us to define a category $\OX$ for a wide array of topological algebras. In particular, we apply this concept to $p$-adic rational Cherednik algebras, which allows us to obtain a $p$-adic analytic version of the category $\OX$ developed by Ginzburg, Guay, Opdam and Rouquier. Along the way, we study the global sections of $p$-adic Cherednik algebras on smooth Stein spaces, and determine their behavior with respect to the rigid analytic GAGA functor.
\end{abstract}
\maketitle

\tableofcontents
\section{Introduction}

This paper is a continuation of \cite{p-adicCheralg}, where we developed a theory of $p$-adic Cherednik algebras on smooth rigid analytic varieties with an action of a finite group. In this paper, we will start investigating the representation theory of such algebras, focusing on the rational case. That is, the case of a finite group $G$ acting linearly on a finite-dimensional vector space $\mathfrak{h}$. In particular, the main goal of this paper is showing that $p$-adic rational Cherednik algebras are Fréchet-Stein (\emph{cf.} \cite[Section 3]{schneider2002algebras}), and that they can be endowed with a \emph{triangular decomposition} as Fréchet-Stein algebras. As we will see below, any 
Fréchet-Stein algebra $\mathscr{R}$ with a triangular decomposition admits a category $\wideparen{\OX}$. This is a full abelian subcategory of $\mathcal{C}(\mathscr{R})$  (the category of co-admissible modules of $\mathscr{R}$)  which has a canonical structure as a highest weight category, and such that its irreducible objects are determined by the triangular decomposition. In the case of the rational Cherednik algebra $\mathcal{H}_c(\mathfrak{h},G)$, this will allow us to obtain a highest weight subcategory  of $\mathcal{C}(\mathcal{H}_c(\mathfrak{h},G))$, which contains all co-admissible Verma modules, and such that its irreducible objects are in natural bijection with the irreducible $K$-linear representations of $G$.\bigskip

Let us now give an overview of these notions: For the rest of the introduction, let us fix a $p$-adic field $K$. The concept of a triangular decomposition of a graded algebra was introduced by Ginzburg, Guay, Opdam and Rouquier in \cite{ginzburg2003category}.
\begin{defi*}
A graded $K$-algebra $R$ is admits a triangular decomposition when there are graded algebras $A,B,H\subset R$ satisfying that there is a decomposition:
\begin{equation*}
    R=A\otimes_KH\otimes_KB.
\end{equation*}
These algebras must also satisfy the following properties:
\begin{enumerate}[label=(\roman*)]
   \item $B\cdot  H=H\cdot B$ and $H\cdot A=A\cdot H$.
    \item $A=\bigoplus_{n\geq 0}A_n$, $B=\bigoplus_{n\leq 0}B_n$, $A_0=B_0=K$, and $H\subset R_0$. 
    \item For each $n\in\mathbb{Z}$, $A_n$ and $B_n$ are finite dimensional $K$-vector spaces.
    \item $H$ is a finite dimensional semi-simple split $K$-algebra.
    \item The grading on $R$ is inner. That is, there is some $\partial\in R_0$ such that:
    \begin{equation*}
        R_n=\{ x\in R \textnormal{ }\vert \textnormal{ } [\partial,x]=nx   \}.
    \end{equation*}
\end{enumerate}
We let $AH$ and $BH$ the subalgebras of $R$ given by  $A\otimes_KH$ and $B\otimes_KH$ respectively.    
\end{defi*}
As an example, notice that any Weyl algebra admits a triangular decomposition. To any such algebra, one can attach a category $\OX\subset \Mod(R)$, which is a highest weight category satisfying that its irreducible objects are in bijection with the irreducible representations of $H$. Any irreducible object in $\OX$ arises as the unique maximal quotient of a Verma module:
\begin{equation*}
    \Delta(W)=R\otimes_{BH}W =A\otimes_KW,
\end{equation*}
where $W\in\operatorname{Irr}(H)$. Furthermore, these objects are the standard objects in the highest weight category structure of $\OX$. Any object $M\in \OX$ decomposes as a direct sum of generalized eigenspaces for the action of $\partial$, and it can be shown that each eigenspace is a finite-dimensional $K$-vector space. The details of this construction can be found in the body of the paper. In particular in Section \ref{Section category O over p-adic fields}.\bigskip

Now that the board has been set, let us discuss the contents of the paper. As mentioned above, our primary goal is obtaining suitable versions of the triangular decomposition and the associated category $\OX$ for Banach and Fréchet-Stein algebras. Let us handle the Banach case first. Formulating a Banach version of the category $\OX$ poses a few technical problems, which are mostly analytic in nature. Namely, an important part of the theory of the algebraic category $\OX$ uses the fact that every $M\in \OX$ decomposes as a direct sum of finite-dimensional generalized eigenspaces. Hence, if one wishes to generalize the theory to the analytic setting, it is necessary to obtain a good theory of decomposition of Banach spaces into generalized eigenspaces of a bounded endomorphism.\\
A good analysis of the situation was obtained by C.T. Féaux de Lacroix in \cite{diagonFrechet modules}, where the following definition is given: Let $V$ be a Banach space with a bounded endomorphism $\partial:V\rightarrow V$. For any $\lambda \in K$, we let $V^{\lambda}$ be the weight space of weight $\lambda$  (\emph{cf.} Definition \ref{defi weight space decomp}). We say $V$ admits a weight space decomposition if there is some $I\subset K$ such that every $v\in V$ can be expressed as a convergent sum:
\begin{equation*}
    v=\sum_{\lambda \in I } v_{\lambda}, \textnormal{ where } v_{\lambda}\in V^{\lambda}.
\end{equation*}
This simple definition turns out to be rather successful, as the results in \cite{diagonFrechet modules} were employed  by T. Schmidt in \cite{Schmidt2010VermaMO}, and N. Dupré in \cite{Duprequantum}, where analytic versions of the category $\OX$ for the Arens-Michael envelope of a reductive Lie algebra (resp. quantum Arens-Michael envelope) were obtained. Unfortunately, the results in \cite{diagonFrechet modules} only deal with the case where the operator acts diagonally on each of the weight spaces, and this is not the situation we find ourselves in. To remedy this, we will dedicate Section \ref{Section decomposition into weight spaces} to discussing several types of decompositions of a Banach space into a sum of generalized eigenspaces, and analyzing their properties.\bigskip 

After this analytical detour, we are now ready to start discussing triangular decompositions. This is done in Section \ref{Section decompositions}, where we discuss the Banach and Fréchet-Stein cases. Let $\mathscr{R}$ be a two-sided noetherian Banach $K$-algebra. Roughly speaking, we say $\mathscr{R}$ admits a triangular decomposition if there is a dense graded subalgebra $R\subset \mathscr{R}$ which admits a triangular decomposition with respect to some graded subalgebras $A,B,H\subset R$ satisfying some additional technical conditions (\emph{cf.} Definition \ref{defi triangular decomposition of Banach algebras}). Let $\mathscr{A}$, and $\mathscr{B}$ be the closures in $\mathscr{R}$ of $A$ and $B$ respectively.
\begin{defi*}
We define the category $\widehat{\OX}$ as the full subcategory of $\Mod(\mathscr{R})$ given by the modules which are finite as $\mathscr{A}$-modules and admit a weight space decomposition with respect to the action of $\partial\in R$.
\end{defi*}
A priory, it is not clear that $\widehat{\OX}$ is the category we are looking for. For instance, our inability to
show that closed $\partial$-invariant subspaces admit weight space decompositions makes it hard to show that $\widehat{\OX}$ is an abelian category. We solve this by showing that elements $M\in \widehat{\OX}$ admit a Jordan-Hölder series:
\begin{equation*}
    0=M_0\subset M_1\subset \cdots \subset M_n=M,
\end{equation*}
such that the $M_i\in \widehat{\OX}$, and the  $M_{i+1}/M_i$ are simple in $\widehat{\OX}$. Additionally, we  show that every simple object is the unique simple quotient of an analytic Verma module:
\begin{equation*}
    \widehat{\Delta}(W):= \mathscr{A}\widehat{\otimes}_KW,
\end{equation*}
where $W\in \operatorname{Irr}(H)$. Using this, we will be able to show the following:
\begin{intteo}
The category $\widehat{\OX}$ satisfies the following properties:
\begin{enumerate}[label=(\roman*)]
    \item $\widehat{\OX}$ is a full abelian subcategory of $\Mod(\mathscr{R})$.
    \item $\widehat{\OX}$ is closed under subobjects and finite sums.
    \item There is an equivalence of abelian categories: $\mathscr{R}\otimes_R-:\OX\rightarrow \widehat{\OX}.$
\end{enumerate}
\end{intteo}
\begin{proof}
This is shown in Theorem \ref{teo category O for p-adic banach algebras with a triangular decomposition}.
\end{proof}
In particular, this theorem shows that $\widehat{\OX}$ is a highest weight category, and that its simple objects are in one to one correspondence with the irreducible representations of $H$. We will finish the section by giving different characterizations of $\widehat{\OX}$ in terms of the properties of weight space decompositions.\bigskip

Let 
    $\mathscr{R}=\varprojlim_{n\geq 0}\mathscr{R}(n)$ be a two-sided Fréchet-Stein algebra. Our notion of triangular decomposition of a Fréchet-Stein algebra is a natural extension of the notion for Banach algebras. Morally, we say $\mathscr{R}$ admits a triangular decomposition if there is a dense graded subalgebra $R\subset \mathscr{R}$ with a triangular decomposition with respect to graded subalgebras $A,B,H\subset R$ which induces a triangular decomposition of $\mathscr{R}(n)$ for each $n\geq 0$. Under mild technical assumptions (\emph{cf.} Section \ref{remark conditions for Frechet category O}), each of the $\mathscr{R}(n)$ admits a category  $\widehat{\OX}(n)$ satisfying that extension of scalars induces equivalences of highest weight categories $\widehat{\OX}(n+1)\rightarrow \widehat{\OX}(n)$ for each $n\geq 0$.
\begin{intteo}
We define the category $\wideparen{\OX}$ of $\mathscr{R}$ as the inverse limit:
\begin{equation*}
    \wideparen{\OX}:=\varprojlim\wideparen{\OX}(n).
\end{equation*}
The category $\wideparen{\OX}$ satisfies the following properties:
\begin{enumerate}[label=(\roman*)]
    \item $\wideparen{\OX}$ is a full abelian subcategory of $\mathcal{C}(\mathscr{R})$.
    \item $\wideparen{\OX}$ is contains all closed subobjects and finite direct sums.
    \item There is an equivalence of abelian categories: $\mathscr{R}\otimes_R-:\OX\rightarrow \wideparen{\OX}.$
\end{enumerate}
\end{intteo}
\begin{proof}
This is Proposition \ref{prop properties of the category O Frechet} in the body of the text.
\end{proof}
We end the section by characterizing $\wideparen{\OX}$ in terms of weight space decompositions.\bigskip

In the second part of the paper, we will apply the theory of triangular decompositions to $p$-adic rational Cherednik algebras. Let us give a brief remainder of these algebras: Let $\mathfrak{h}$ be a finite dimensional $K$-vector space, and $G\leq \operatorname{GL}(\mathfrak{h})$ be a finite group. We let $S(G)$ be the set of pseudo-reflections in $G$ (\emph{ie.} the family of elements $s\in G$ such that $\operatorname{Rank}(s-\operatorname{Id})=1$), and define a reflection function to be a $G$-invariant function $c:S(G)\rightarrow K$. We let $\operatorname{Ref}(\mathfrak{h},G)$ be the finite-dimensional vector space of all reflection functions. The family of rational Cherednik algebras associated to the action of $G$ on $\mathfrak{h}$, is the family of filtered $K$-algebras $H_{c}(\mathfrak{h},G)$,
parameterized by  reflection functions $c\in \operatorname{Ref}(\mathfrak{h},G)$, and defined as the quotient of the tensor algebra 
$G\ltimes T(\mathfrak{h}\oplus \mathfrak{h}^*)$ by the following relations:
\begin{equation}\label{equation relations rat cher alg}
   [v,w]=0, \textnormal{ } [x,y]=0, \textnormal{ } [v,x]-(v,x)+\sum_{g\in S(G)}c(g)(v,\alpha_g)(\alpha_g^{\vee},x)g=0,
\end{equation}
where $v,w\in \mathfrak{h}$,  $x,y\in \mathfrak{h}^*$, and the $\alpha_g\in \mathfrak{h}^*$ and $\alpha_g^{\vee}\in \mathfrak{h}$  are certain vectors determined by the hypersurface $\mathfrak{h}^g$. These algebras are the prime example of an algebra with a triangular decomposition, and their representation theory has been systematically studied in the last decade. 
\bigskip

We will mostly be interested in a $p$-adic version of this algebra. Let $\mathfrak{h}^{\operatorname{an}}$ be the rigid-analytification of $\mathfrak{h}$. As shown in \cite{p-adicCheralg}, our choice of reflection function  yields a sheaf of complete topological algebras $\mathcal{H}_{c,\mathfrak{h}^{\operatorname{an}},G}$ on the quotient $\mathfrak{h}^{\operatorname{an}}/G$. Such a sheaf is called a sheaf of $p$-adic Cherednik algebras. We define the $p$-adic rational Cherednik algebra as the following complete topological algebra:
\begin{equation*}
    \mathcal{H}_c(\mathfrak{h}^{\operatorname{an}},G):=\Gamma(\mathfrak{h}^{\operatorname{an}}/G,\mathcal{H}_{c,\mathfrak{h}^{\operatorname{an}},G}).
\end{equation*}
Our goal for the second part of the chapter is showing that $\mathcal{H}_c(\mathfrak{h}^{\operatorname{an}},G)$ is a Fréchet-Stein algebra with a triangular decomposition. Logically, the first step is showing that $\mathcal{H}_c(\mathfrak{h}^{\operatorname{an}},G)$ is indeed a Fréchet-Stein algebra. This is achieved in Section \ref{Section p-adic cher alg smooth stein}. We will actually work in greater generality, by working with smooth Stein spaces admitting an action of a finite group $G$. For the benefit of the reader, we include in Section \ref{Section group act on Stein spaces} a treatment of the basics of the theory of group actions on Stein spaces, putting special emphasis on the quotients of such actions. The main results of Section 
\ref{Section p-adic cher alg smooth stein} can be summarized into the following theorem:
\begin{intteo}
Let $X$ be a Stein space with a $G$-action and an étale map $X\rightarrow\mathbb{A}^r_K$. Let $\mathcal{H}_{t,c,\omega,X,G}$ be a sheaf of $p$-adic Cherednik algebras on $X$. Then the algebra:
\begin{equation*}
\mathcal{H}_{t,c,\omega}(X,G):=    \Gamma(X/G,\mathcal{H}_{t,c,\omega,X,G}),
\end{equation*}
is a Fréchet-Stein algebra and a nuclear Fréchet space. Furthermore, every co-admissible $\mathcal{H}_{t,c,\omega}(X,G)$-module is also a nuclear Fréchet space.
\end{intteo}
\begin{proof}
See Proposition \ref{prop sections on Stein spaces} for details.    
\end{proof}
The rest of the section is devoted to understanding the properties of the  map:
\begin{equation*}
  H_{c}(\mathfrak{h},G)\rightarrow  \mathcal{H}_{c}(\mathfrak{h}^{\operatorname{an}},G).
\end{equation*}
Again, we work in greater generality, by studying the properties of Cherednik algebras under the GAGA functor. This is done in Section \ref{Section GAGA}. In Section \ref{Section AM envelopes} we recall the notion of the Arens-Michael envelope of a topological algebra, and some of its basic properties. The main results of these sections can be condensed as follows:
\begin{intteo}
Let $X=\Spec(A)$ be an affine variety with an action of $G$. Let $H_{t,c,\omega,X,G}$ be a sheaf of Cherednik algebras on $X$. The canonical map of $K$-algebras:
\begin{equation*}
    H_{t,c,\omega}(X,G)\rightarrow \mathcal{H}_{t,c,\omega}(X^{\operatorname{an}},G),
\end{equation*}
is faithfully flat, and 
$\mathcal{H}_{t,c,\omega}(X^{\operatorname{an}},G)$ is the Arens-Michael envelope of $H_{t,c,\omega}(X,G)$.    
\end{intteo}
\begin{proof}
  This is a part of Theorem \ref{teo GAGA and Arens-Michaels envelopes}.  
\end{proof}
Finally, in Section \ref{Section category O for p-adic rat Cher}
we apply the previous machinery to the case of $p$-adic rational Cherednik algebras. In particular, we give a summary of the main properties of (algebraic) rational Cherednik algebras, and give an explicit description of their triangular decomposition. We conclude the paper  with Section \ref{Section p-adic rational Cher algebra}, where we construct an explicit Fréchet-Stein presentation of $\mathcal{H}_c(\mathfrak{h}^{\operatorname{an}},G)$, and show that it admits a triangular decomposition in the sense of Fréchet-Stein algebras. By the results above, it admits a category $\OX$, which we denote $\wideparen{\OX}_c$. We have the following theorem:
\begin{intteo}
 The rational Cherednik algebra admits a Fréchet-Stein presentation:
\begin{equation*}
    \mathcal{H}_c(\mathfrak{h}^{\operatorname{an}},G)=\varprojlim_{m\geq 0}\widehat{\mathcal{H}}(m)_{c,K},
\end{equation*}
which admits a triangular decomposition with respect to $H_c(\mathfrak{h},G)\subset \mathcal{H}_c(\mathfrak{h}^{\operatorname{an}},G)$.
Thus, there is a full abelian subcategory $\wideparen{\OX}_c\subset \mathcal{C}(\mathcal{H}_c(\mathfrak{h}^{\operatorname{an}},G))$  satisfying the following:
\begin{enumerate}[label=(\roman*)]
    \item $\wideparen{\OX}_c$ is closed under closed submodules and finite direct sums.
    \item $\wideparen{\OX}_c$ is a highest weight category, and   its irreducible objects are in one-to-one correspondence with the irreducible $K$-linear representations of $G$.
\end{enumerate}
The canonical map $H_c(\mathfrak{h},G)\rightarrow \mathcal{H}_c(\mathfrak{h}^{\operatorname{an}},G)$ satisfies the following properties:
\begin{enumerate}[label=(\roman*)]
    \item It is faithfully flat and has dense image.
    \item It makes $\mathcal{H}_c(\mathfrak{h}^{\operatorname{an}},G)$ the Arens-Michael envelope of $H_c(\mathfrak{h},G)$.
    \item It induces an equivalence of highest weight categories:
    \begin{equation*}
        \mathcal{H}_c(\mathfrak{h}^{\operatorname{an}},G)\otimes_{H_c(\mathfrak{h},G)}-:\OX_c\leftrightarrows \wideparen{\OX}_c.
    \end{equation*}
\end{enumerate}     
\end{intteo}
\begin{proof}
This is shown below in Theorem \ref{teo category O p-adic rat Cher}.
\end{proof}
The analytic nature of $\mathcal{H}_c(\mathfrak{h}^{\operatorname{an}},G)$ opens the way for $p$-adic analytic methods to be employed in the representation theory of rational Cherednik algebras defined over $p$-adic fields. For instance, the fact that $\mathcal{H}_c(\mathfrak{h}^{\operatorname{an}},G)$ is the Arens-Michael envelope of $H_c(\mathfrak{h},G)$ implies that every representation of $H_c(\mathfrak{h},G)$ on a Banach $K$-vector space lifts uniquely to a continuous representation of $\mathcal{H}_c(\mathfrak{h}^{\operatorname{an}},G)$.\bigskip

Similarly, we expect the theory of $p$-adic rational Cherednik algebras to have applications in modular representation theory. To wit, let $k$ be a finite field of characteristic $p$, $\mathcal{R}=W(k)$ be its ring of $p$-typical Witt vectors, and $K=\mathcal{R}[p^{-1}]$. Let $\mathfrak{h}_{\mathcal{R}}$ be a finite affine space over $\mathcal{R}$, and assume that the group  $G$ admits a $\mathcal{R}$-linear action on $\mathfrak{h}_{\mathcal{R}}$. In this situation, one can mimic the relations in equation $(\ref{equation relations rat cher alg})$, and define an \emph{integral} rational Cherednik algebra $H_c(\mathfrak{h}_{\mathcal{R}},G)$. This is a filtered associative $\mathcal{R}$-algebra which satisfies the following properties: First, the generic fiber $H_c(\mathfrak{h},G)=H_c(\mathfrak{h}_{\mathcal{R}},G)\otimes_{\mathcal{R}}K$ is a rational Cherednik algebra for the induced action of $G$ on $\mathfrak{h}$. Next, the quotient $H_c(\mathfrak{h}_k,G):=\mathcal{H}_c(\mathfrak{h}_{\mathcal{R}},G)/p$ is a Cherednik algebra for the action of $G$ on the special fiber $\mathfrak{h}_{k}$. The theory of Cherednik algebras in positive characteristic has been thoroughly studied in \cite{cherposchar1}, and \cite{cherposchar2}, and we believe it should be possible to study representations of the characteristic $p$ Cherednik algebra $H_c(\mathfrak{h}_k,G)$ via lifting them to a representation of the integral rational Cherednik algebra $H_c(\mathfrak{h}_{\mathcal{R}},G)$, and then to the associated $p$-adic rational Cherednik algebra $\mathcal{H}_c(\mathfrak{h},G)$. This process would combine ideas from the classical theory of lifting modular representations \cite{schneidermodular}, the theory of lifts for $\D$-modules in positive characteristic from \cite{liftingDmod}, and the theory of Weyl modules \cite{jantz}.\bigskip

Finally, let us mention that we also expect that this line of research can be extended to study rational Cherednik algebras for more general groups. In particular, our definition of $p$-adic rational Cherednik algebra is closely related with Ardakov-Wadsley's sheaf of completed $p$-adic differential operators $\wideparen{\D}$ 
(\emph{cf.} \cite{ardakov2019}). One of the most important features of these sheaves is the fact that there is a version of Beilinson-Bernstein localization in this setting, which links admissible locally analytic representations of a $p$-adic Lie group $G$ with $G$-equivariant co-admissible $\wideparen{\D}$-modules on (the rigid analytification of) the flag variety (\emph{cf.} \cite{ardakovequivariant}).\\
Although the definition of Cherednik algebras extensively uses the fact that the set of pseudo-reflections of $G$ is finite, from a deformation-theoretic perspective, rational Cherednik algebras classify all the  infinitesimal deformations of the skew group algebra $G\ltimes \D(\mathfrak{h})$. We believe that this perspective can be extended to fit more general $p$-adic Lie groups, and that a theory of $p$-adic Cherednik algebras in this context is possible. Such a theory would lead to new approaches to the theory of locally analytic representations of $p$-adic Lie groups, and to new algebraic and geometric invariants for such representations.
\subsection*{Related work}
This paper is part of a series of papers \cite{p-adicCheralg}, \cite{HochschildDmodules}, \cite{HochschildDmodules2} in which we start a study of the deformation theory of the skew group algebras of differential operators $G\ltimes \wideparen{\D}_X(X)$ on a smooth Stein space $X$ equipped with the action of a finite group $G$. In particular, in \cite{HochschildDmodules} we develop a formalism of Hochschild cohomology and homology for $\wideparen{\D}_X$-modules on smooth and separated rigid analytic spaces, and give explicit calculations of the Hochschild cohomology groups $\operatorname{HH}^{\bullet}(\wideparen{\D}_X)$ in terms of the de Rham cohomology of $X$. This is done in the setting of sheaves of complete Ind-Banach spaces and quasi-abelian categories, as developed in \cite{bode2021operations}. These results are further extended in \cite{HochschildDmodules2}, were study the case where $X$ is a smooth Stein space, and  classify the infinitesimal deformations of $G\ltimes\wideparen{\D}_X(X)$ in some situations.
\subsection*{Notation and Conventions}
 For the rest of the text, we fix a prime $p$, and we let $K$ be a complete discretely valued extension of $\mathbb{Q}_p$, we denote its valuation ring by $\mathcal{R}$, fix a uniformizer $\pi$, and denote its residue field by $k$. For a $\mathcal{R}$-module $\mathcal{M}$ we will often times
 write $\mathcal{M}_K=\mathcal{M}\otimes_{\mathcal{R}}K$. Throughout the text, we will freely make use of basic notions of $p$-adic functional analysis. We refer to P. Schneider's monograph \cite{schneider2013nonarchimedean} for definitions. 
\subsection*{Acknowledgments}
This paper was written as a part of a PhD thesis at the Humboldt-Universität zu Berlin under the supervision of Prof. Dr. Elmar Große-Klönne. I would like to thank Prof. Große-Klönne
for pointing me towards this exciting topic, and reading an early draft of the paper.

\bigskip
Funded by the Deutsche Forschungsgemeinschaft (DFG, German Research
Foundation) under Germany´s Excellence Strategy – The Berlin Mathematics
Research Center MATH+ (EXC-2046/1, project ID: 390685689).
\section{\texorpdfstring{Background in algebra and $p$-adic analysis}{}}
\subsection{\texorpdfstring{Category $\OX$ for algebras with a triangular decomposition}{}}\label{Section category O over p-adic fields}
Before moving to the realm of $p$-adic analysis, it is wise to do an introduction to the classical category $\OX$ for rational Cherednik algebras. We will in fact work in a more general setting, by studying the category $\OX$ attached to any algebra with a triangular decomposition. We will closely follow in the footsteps of \cite{ginzburg2003category}. 
\begin{defi}\label{defi triangular decomposition}
Let $R$ be a noetherian graded algebra. A triangular decomposition of $R$ consists of three graded subalgebras $A,B,H\subset R$ satisfying the following:
\begin{enumerate}[label=(\roman*)]
    \item $R=A\otimes_KH\otimes_KB$ as $K$-vector spaces.
    \item $B\cdot  H=H\cdot B$ and $H\cdot A=A\cdot H$.
    \item $A=\bigoplus_{n\geq 0}A_n$, $B=\bigoplus_{n\leq 0}B_n$, $A_0=B_0=K$, and $H\subset R_0$. 
    \item For each $n\in\mathbb{Z}$, $A_n$ and $B_n$ are finite dimensional $K$-vector spaces.
    \item $H$ is a finite dimensional semi-simple split $K$-algebra.
    \item The grading on $R$ is inner. That is, there is some $\partial\in R_0$ such that:
    \begin{equation*}
        R_n=\{ x\in R \textnormal{ }\vert \textnormal{ } [\partial,x]=nx   \}.
    \end{equation*}
\end{enumerate}
We will denote the data of a triangular decomposition by the tuple $(A,B,H,\partial)$.
\end{defi}
Examples of algebras with a triangular decomposition are given by the Weyl algebras, and more generally the rational Cherednik algebras (see Section \ref{Section p-adic rational Cher algebra} below). 
\begin{defi}\label{defi notions of triang decomp}
We introduce the following objects:
\begin{enumerate}[label=(\roman*)]
    \item $B^n=B_{-n}$ for all $n\geq 0$.
    \item We have $\partial=\partial'-\partial_0$, where $\partial'\in A\otimes_KH\otimes_K B_{<0}$ and $\partial_0\in Z(H)$.
    \item Let $\operatorname{Irr}(H)$ the set of isomorphism classes of irreducible $H$-modules. 
    \item For $W\in\operatorname{Irr}(H)$, we let $c(W)\in K$ be the scalar by which $\partial_0$ acts on $W$.
    \item We let $AH$ be  the subalgebra of $R$ given by  $A\otimes_KH$.
    \item We let $BH$ be  the subalgebra of $R$ given by  $B\otimes_KH$.
\end{enumerate}
\end{defi}
As mentioned above, algebras with a triangular decomposition have a rich representation theory, which resembles that of the universal enveloping algebra of a semi-simple Lie algebra. In particular, each algebra with a triangular decomposition admits a category $\OX$, which is a highest weight category whose simple objects are in bijection with the irreducible representations of $H$. We will devote the rest of this section to the construction of this category, and giving explicit descriptions of its relevant objects (\emph{ie.} standard objects,  partial order of the simple objects \emph{etc}.)
\begin{defi}
Let $M$ be a $B$-module. We say the action of $B$ on $M$ is locally nilpotent if for every $m\in M$ there is some $n\geq 0$ such that $B^{>n}m=0$.
\end{defi}
For every $n\geq 0$,  consider the following two-sided ideal of $BH$:
\begin{equation*}
    BH^{>n}:= B^{>n}\otimes_K H.
\end{equation*}
There is a canonical surjection $BH\rightarrow H$, given by quotienting by the ideal $BH^{>0}$. Hence, every $H$-module can be regarded as a $BH$-module via this epimorphism. In particular, each irreducible $H$-module gives rise to a (generalized) Verma module:
\begin{defi}\label{defi Verma modules}
Let $W\in\operatorname{Irr}(W)$. The generalized Verma module of level $n$ associated to $W$ is the following $R$-module:
\begin{equation*}
    \Delta(W)_n:=\operatorname{Ind}_{BH}^R(BH/BH^{>n}\otimes_H W)=A\otimes_K\left(B/B^{>n}\otimes_KW\right).
\end{equation*}
If $n=0$ we will write $\Delta(W):=\Delta(W)_0$ and call $\Delta(W)$ the Verma module of $W$. 
\end{defi}
As $H$ is split semi-simple, $\Mod(H)$ is also semi-simple. Hence, Verma modules yield an embedding $\Mod(H)\rightarrow \Mod(R)$. The last identity endows the generalized Verma modules  $\Delta(W)_n$ with a graded $R$-module structure by letting $W$ sit in degree zero. Furthermore, the action of $B$ on any $\Delta(W)_n$ is locally nilpotent. 
\begin{prop}[{\cite[Proposition 2.2]{ginzburg2003category}}]\label{prop description of Oln}
Let $\OX^{\operatorname{ln}}$ be the full subcategory of $\Mod(R)$ given by the locally nilpotent $B$-modules. The following are equivalent:
\begin{enumerate}[label=(\roman*)]
    \item $M\in \OX^{\operatorname{ln}}$.
    \item $M$ admits an ascending filtration whose quotients are quotients of $\Delta(W)$'s.
    \item $M$ is a quotient of a direct sum of $\Delta(W)_n's$.
\end{enumerate}
We call a filtration as in $(ii)$ a $\Delta$-filtration.
\end{prop}
It follows from this description that $\OX^{\operatorname{ln}}$ is a Serre subcategory of $\Mod(R)$ that contains all Verma modules. Unfortunately, this category is too big for our purposes. In particular, $\OX^{\operatorname{ln}}$ is closed under arbitrary direct sums, so is far from being an artinian category. For this reason, we introduce the following category:
\begin{defi}
Let $R$ be an algebra admitting a triangular decomposition. We define $\OX$ as the full subcategory of $\OX^{\operatorname{ln}}$ given by the finitely generated $R$-modules.
\end{defi}
As $R$ is noetherian, this is equivalent to the filtration in $(ii)$ being finite. Hence, $\OX$ is the smallest Serre subcategory of $\Mod(R)$ containing all Verma modules. To show that $\OX$ is highest weight, we need to study its simple objects. As it turns out, the best way of doing this is studying the action of $\partial$ on the objects in $\OX$. 
\begin{defi}\label{defi weight space decomp}
Let $M\in \Mod(R)$, $\lambda \in K$. A generalized weight space in $M$ is:
\begin{equation*}
    M^{\lambda}=\{ m\in M\textnormal{ }\vert \textnormal{ } (\partial-\lambda)^nm=0 \textnormal{ for } n>>0\}.
\end{equation*}
The scalar $\lambda\in K$ is called the weight of the generalized weight space $M_{\lambda}$. We say that $M$ admits a (generalized) weight space decomposition if we have:
\begin{equation*}
    M=\bigoplus_{\lambda\in K}M_{\lambda},
\end{equation*}
and we say the decomposition is of finite type if each of the $M_{\lambda}$ is finite dimensional. If $M$ admits a weight space decomposition, we set:
\begin{equation*}
    \Lambda(M)= \{ \lambda\in K \textnormal{ }\vert \textnormal{ } M^{\lambda}\neq 0 \}.
\end{equation*}
\end{defi}
The most important instance of a weight space decomposition is the following:
\begin{prop}\label{eigenspace structure for algebraic Verma modules}
For each $W\in \operatorname{Irr}(W)$, the Verma module $\Delta(W)$ admits a finite-type decomposition into generalized eigenspaces for $\partial$:
\begin{equation*}
    \Delta(W)=\bigoplus_{n\in \mathbb{Z}^{\geq 0}}\Delta(W)^{n+c(W)}.
\end{equation*}   
Furthermore, $\Delta(W)^{n+c(W)}=A_n\otimes_KW$, and  $\partial$ acts semi-simply on weight spaces.
\end{prop}
\begin{proof}
It suffices to show that $\partial$ acts on $A_n\otimes_KW$ by multiplication by $c(W)+n$. As $\partial$ is the grading element of $R$, for every $x\in A_n$ and every $v\in W$ we have:
\begin{equation*}
    \partial(x\otimes v)=(x\partial  +   nx)\otimes v= (c(W)+n)x\otimes v.
\end{equation*}
\end{proof}
It is clear that the modules admitting a finite-type weight space decomposition form a Serre subcategory of $\Mod(R)$. As Verma modules admit such a decomposition, it that every object in $\OX$ admits a finite-type decomposition into weight spaces. More generally, we have the following proposition:
\begin{prop}[{\cite[Lemma 2.22]{ginzburg2003category}}]\label{prop equiv conditions category O}
Let $M\in \OX^{\operatorname{ln}}$. The following are equivalent:
\begin{enumerate}[label=(\roman*)]
    \item $M\in \OX$.
    \item $M$ is a finitely generated $A$-module.
    \item $M$ admits a finite-type decomposition into generalized weight spaces for $\partial$.
\end{enumerate}
\end{prop}
Consider a short exact sequence of objects in $\OX$:
\begin{equation*}
    0\rightarrow V_1\rightarrow V_2\rightarrow V_3\rightarrow 0.
\end{equation*}
For every weight $\lambda\in K$ we have the identity $\operatorname{dim}(V_2^\lambda)=\operatorname{dim}(V_{1}^\lambda)+\operatorname{dim}(V_{3}^\lambda)$.
In particular, we have $\Lambda(V_2)=\Lambda(V_1)\cup \Lambda(V_3)$. Thus, the weight spaces of objects in $\OX$ are completely determined by the weight spaces of the Verma modules appearing in a $\Delta$-filtration. Therefore, we have:
\begin{prop}\label{prop weight spaces elements in OX}
Let $M\in\OX$. There exists a finite family $\{W_i \}_{i=1}^n\subset \operatorname{Irr}(H)$ such that the weight space decomposition of $M$ satisfies:
\begin{equation*}
    \Lambda(M)\subset \bigcup_{i=1}^nc(W_i)+\mathbb{Z}^{\geq 0}.
\end{equation*}
\end{prop}
\begin{proof}
    Follows at once from the fact that every $M\in\OX$ admits a $\Delta$-filtration, together with the description of the weight spaces given in Proposition \ref{eigenspace structure for algebraic Verma modules}.
\end{proof}
This  condition will be instrumental when dealing with the analytic version of $\OX$. As an upshot, notice that for any $M\in\OX$ and $W\in\operatorname{Irr}(H)$ we have:
\begin{equation*}
    \Hom_R(\Delta(W),M)=\Hom_H(W,M).
\end{equation*}
As $W$ is finite-dimensional, and the image of any $H$-linear homomorphism from $W$ must land in $M^{c(W)}$, it follows that $\Hom_R(\Delta(W),M)$ is finite-dimensional. Furthermore, as every element in $\OX$ admits a $\Delta$-filtration, we can conclude that $\Hom_{\OX}(M,N)$ is finite-dimensional for all $M,N\in \OX$.
\begin{prop}[{\cite[Proposition 3.29.]{etingof2010lecture}}]
Let $W\in\operatorname{Irr}(H)$. Then $\Delta(W)$ has a unique maximal submodule $M(W)$, and a unique simple quotient $L(W)$. Any simple object of $\OX$ arises in this way.
\end{prop}
The proof crucially uses the structure of weight spaces of Verma modules determined in Proposition \ref{eigenspace structure for algebraic Verma modules}. As any element in $\OX$ admits a $\Delta$-filtration, we may use this to show that every object in $\OX$ admits a Jordan-Hölder series:
\begin{prop}[{\cite[Lemma 2.14, Corollary 2.16.]{ginzburg2003category}}]\label{JH series in algebraic O}
Every object in $\OX$ admits a finite Jordan-Hölder series. That is, given $M\in \OX$ there is a finite filtration:
\begin{equation*}
    0=M_0\subset M_1\subset \cdots\subset M_n=M,
\end{equation*}
where $M_{i}/M_{i-1}=L(W_i)$ for $W_i\in \operatorname{Irr}(G)$. If $M$ is a quotient of some $\Delta(E)$, then $L(W_n)=L(E)$, and  $L(W_i)$ satisfies $c(E)-c(W_i)\in \mathbb{Z}^{>0}$ for $1\leq i \leq n-1$.
\end{prop}
We almost have all the ingredients needed to show that $\OX$ is highest weight. The only remaining part is showing the existence of projective covers of simple objects:
\begin{prop}[{\cite[Corollary 2.10.]{ginzburg2003category}}]
Let $W\in \operatorname{Irr}(H)$. Then  $L(W)$ has a projective cover
$P(W)$ in  $\OX$,  such that the quotients, $L(W_i)$, of its Jordan-Hölder series satisfy  $c(W_i)-c(E)\in \mathbb{Z}^{>0}$ for $1\leq i\leq n-1$, and $L(W_n)=L(W)$.
\end{prop}
We can condense the results thus far into the following theorem:
\begin{teo}
$\OX$ is a highest weight category. For every $W\in \operatorname{Irr}(H)$
its associated standard object is $\Delta(W)$, and  its simple module is $L(W)$. The partial order relation in $\operatorname{Irr}(H)$ is given by:
\begin{equation*}
    W<E \textnormal{ if and only if }  c(E)-c(W)\in \mathbb{Z}^{> 0}.
\end{equation*}
\end{teo}
Thus, we have constructed a highest weight category whose simple objects are in natural bijection with the irreducible representations of $H$. 
\subsection{\texorpdfstring{Weight space decomposition for $p$-adic Banach spaces}{}}\label{Section decomposition into weight spaces}
We start by recalling some concepts from $p$-adic functional analysis: Let $V$ be a $K$-vector space, and $\partial:V\rightarrow V$ be a $K$-linear endomorphism. Recall the notions of weight space from Definition \ref{defi weight space decomp}, and recall that we say that $V$ admits a decomposition into generalized eigenspaces (weight spaces) if there is an identity:
\begin{equation*}
    V=\bigoplus_{\lambda\in K}V^{\lambda}.
\end{equation*}
In our situation, this notion of decomposition is too restrictive. For instance, a direct sum $\bigoplus_{i\in I}V_i$ is complete if and only if there are only finitely many non-zero spaces in the sum, and each of them is complete.
Thus, we need to introduce an  adequate decomposition for Banach spaces:
\begin{defi}\label{defin basic concepts ws decomposition}
Let $V$ be a locally convex space and $\partial:V\rightarrow V$ be a continuous $K$-linear map. We define the set $\Lambda(V):= \{ \lambda\in K \textnormal{ }\vert \textnormal{ } V^{\lambda}\neq 0 \}$.
\end{defi}
The subspaces $V^{\lambda}\subset V$ need not be closed. However, if there is some $n\geq 0$ such that $(\partial -\lambda)^nv=0$ for all $v\in V^{\lambda}$, then $V^{\lambda}$ is closed. This happens, for example, if $V^{\lambda}$ is finite-dimensional. In a similar vein, for two different $\lambda,\mu\in K$, the fact that $\partial -\lambda$ acts locally nilpotently on $V^{\lambda}$ implies that $\partial$ acts locally finitely on $V^{\lambda}$ (\emph{ie.} for every $v\in V^{\lambda}$, the smallest $\partial$-invariant subspace of $V^{\lambda}$ containing $v$ is finite-dimensional). This in particular implies that $V^{\lambda}\cap V^{\mu}=0$. Hence, we define:
\begin{defi}\label{defi ws functor}
Let $V$ be a locally convex space and $\partial:V\rightarrow V$ be a continuous $K$-linear map. 
We define the following subspace of $V$:
\begin{equation*}
    V^{\operatorname{ws}}:=\bigoplus_{\lambda \in \Lambda(V)}V^{\lambda},
\end{equation*}
and regard it as a locally convex space with respect to the subspace topology.
\end{defi}
Notice that it is not clear if $\overline{V^{\lambda}}\cap \overline{V^{\mu}}=0$. Thus, in order to have a well-behaved theory, we need to impose some conditions on the weight spaces.
\begin{defi}
Let $V$ be a locally convex space and $\partial:V\rightarrow V$ be a continuous $K$-linear map. We say $V$ admits a weight space decomposition  for $\partial$ (with respect to $I$) if there is a family $I\subset \Lambda(V)$ such that:
\begin{itemize}
    \item For every $v\in V$ there is a family $(v_{\lambda})_{\lambda\in I}\in \prod_{\lambda\in I}V^{\lambda}$ satisfying that:
    \begin{equation*}
        v=\sum_{\lambda \in I}v_{\lambda}
    \end{equation*}
\end{itemize}
We consider the following special cases:
\begin{enumerate}[label=(\roman*)]
    \item The decomposition is of compact type if there is $\{\lambda_i\}_{i=1}^n\subset K$ such that:
    \begin{equation*}
        I\subset \bigcup_{i=1}^n \lambda_i+\mathbb{Z}.
    \end{equation*}
    \item The decomposition is of finite type if it is of  compact type and for each $\lambda \in I$ the generalized eigenspace $V^{\lambda}$ is finite dimensional.
    \item We say the decomposition is semi-simple if for each $\lambda\in I$ the action of $\partial$ on $V^{\lambda}$ is semi-simple (ie. it is given by multiplication by $\lambda$). 
\end{enumerate}
\end{defi}
Notice that it is a priory not clear (and in general false) that there is a unique family $I\subset \Lambda(V)$ such that $V$ admits a weight space decomposition with respect to $I$. Furthermore, given a decomposition, it could happen that that there is a non-trivial expression:
\begin{equation*}
    0=\sum_{\lambda \in I}v_{\lambda}.
\end{equation*}
However, this kind of pathological behavior will not be present in the cases of interest to us. A general theory of semi-simple weight space decompositions was established in the doctoral thesis of C.T. Féaux de Lacroix \cite{diagonFrechet modules}. We may condense the main results into the following theorem:
\begin{teo}[{\cite{diagonFrechet modules}}]\label{teo weight space decomposition ss case}
Let $V$ be a Banach space with a continuous endomorphism $\partial: V\rightarrow V$. Assume $V$ admits a semisimple weight space decomposition of compact type for $\partial$ with respect to some indexing set $I\subset K$. The following hold: 
\begin{enumerate}[label=(\roman*)]
    \item $I=\Lambda(V)$. In particular, if $\mu \not\in I$, then $V^{\mu}=0$.
    \item Let $W\subset V$ be a closed and $\partial$-invariant vector subspace. Then $W$ admits a semi-simple weight space decomposition of compact type for $\partial$. Furthermore, for all $\lambda \in \Lambda(W)$ we have a cartesian diagram:
    \begin{equation*}
\begin{tikzcd}
W \arrow[r]                     & V                     \\
W^{\lambda} \arrow[r] \arrow[u] & V^{\lambda} \arrow[u]
\end{tikzcd}
    \end{equation*}
    In particular, $W$ contains the weight components of its elements.
    \item Assume in addition that the decomposition is of finite type.  Then there is a natural injection-preserving bijection of sets:
    \begin{equation*}
      (-)^{\operatorname{ws}}: \left\{ 
				\begin{array}{c} 
					\textnormal{Closed }\partial \textnormal{-invariant}\\ 
					\textnormal{subspaces of } V				\end{array}
\right\}
\rightarrow 
\left\{ 
				\begin{array}{c} 
                \partial-
				\textnormal{invariant}\\ 
					\textnormal{subspaces of } V^{\operatorname{ws}}				\end{array}
\right\}
    \end{equation*}
with inverse given by taking the closure in $V$.
\end{enumerate}
For every closed $\partial$-invariant subspace $W\subset V$, $V/W$ admits a semi-simple weight space decomposition of compact type, with a $\partial$-equivariant short exact sequence:
\begin{equation*}
    0\rightarrow W^{\operatorname{ws}}\rightarrow V^{\operatorname{ws}}\rightarrow (V/W)^{\operatorname{ws}}\rightarrow 0.
\end{equation*}
\end{teo}
\begin{proof} 
This is all contained in chapter one of \cite{diagonFrechet modules}. See also Section 2 of \cite{Schmidt2010VermaMO}.
\end{proof}
As a consequence of $(ii)$, for each $v\in V$ the expression $v=\sum_{\lambda\in\Lambda(V)}v_{\lambda}$ is unique. 
\begin{coro}\label{coro surjectiveness of ws}
 Assume $V$ admits a semi-simple weight space decomposition of finite type for $\partial$, and let $W\subset V^{\operatorname{ws}}$ be dense and $\partial$-invariant. Then $W=V^{\operatorname{ws}}$.
\end{coro}
Unfortunately, in the case we are interested in, the action of $\partial$ on the weight spaces is not semi-simple. Hence, this theorem does not apply to our setting. However, we can circumvent this problem by restricting to spaces which admit an appropriate filtration. For the rest of this section, we fix a Banach space $V$ with a continuous endomorphism $\partial: V\rightarrow V$. We will assume that $V$ admits a weight space decomposition of compact type for $\partial$ with respect to a family $I\subset \Lambda(V)$.    
\begin{Lemma}\label{Lemma ws decomposition and filtrations}
Let $W\subset V$ be a closed and $\partial$-equivariant subspace satisfying that the action of $\partial$ on each weight space of $V/W$ is semi-simple. The following hold:
\begin{enumerate}[label=(\roman*)]
    \item $V/W$ admits a semi-simple weight space decomposition of compact type.
    \item $W$ admits a 
    weight space decomposition with respect to the following family:
    \begin{equation*}
        I_W:= \{ \lambda \in I \textnormal{ }\vert \textnormal{ } V^{\lambda}\cap W\neq 0  \}.
    \end{equation*}
   In particular, the decomposition is of compact type.
    \item If the decomposition is of finite type, then there is a $\partial$-equivariant short exact sequence:
    \begin{equation*}
        0\rightarrow \bigoplus_{\lambda\in I_W}W^{\lambda}\rightarrow \bigoplus_{\lambda\in I}V^{\lambda}\rightarrow (V/W)^{\operatorname{ws}}\rightarrow 0.
    \end{equation*}
\end{enumerate}
\end{Lemma}
\begin{proof}
As $W$ is $\partial$-invariant, the map $\varphi:V\rightarrow V/W$ is a $\partial$-equivariant surjection. Hence, for each $\lambda\in I$ we have $\varphi(V^{\lambda})\subset (V/W)^{\lambda}$. In particular, every $v\in V/W$ admits an expression $v=\sum_{\lambda\in I}v_{\lambda}$ with $v_{\lambda}\in (V/W)^{\lambda}$ for each $\lambda \in I$. By assumption, the action of $\partial$ on each weight space of $V/W$ is semisimple. As the decomposition of $V$ with respect to $\partial$ is of compact type, there is a finite family $\{\lambda_i  \}_{i=1}^n\subset K$ such that $I\subset \bigcup_{i=1}^n\lambda_i+\mathbb{Z}$. Hence, the action of $\partial$ on $V/W$ admits a  semi-simple weight space decomposition of compact type. For $(ii)$, choose $v\in W$, and assume $v=\sum_{\lambda\in I}v_{\lambda}$. It suffices to show that  $v_{\lambda}\in W$ for each $\lambda\in I$. As the action of $\partial$ on $V$ is continuous, we have:
\begin{equation*}
    0=\varphi(v)=\sum_{\lambda\in I}\varphi(v_{\lambda}),
\end{equation*}
and statement $(ii)$ in Theorem \ref{teo weight space decomposition ss case} implies that $\varphi(v_{\lambda})=0$ for all $\lambda\in I$. For the last part of the lemma, we just need to show that the map  $\bigoplus_{\lambda\in I}V_{\lambda}\rightarrow (V/W)^{\operatorname{ws}}$ is surjective. This follows from the fact that $\bigoplus_{\lambda\in I}V_{\lambda}\rightarrow V/W$ has dense image, together with Corollary \ref{coro surjectiveness of ws}.   
\end{proof}
\begin{prop}\label{prop ws decomposition and filtrations}
Assume there is a filtration by closed $\partial$-invariant subspaces:
\begin{equation*}
    0= V_{0}\subset \cdots \subset V_{n}=V,
\end{equation*}
satisfying that the action of $\partial$ on the weight spaces of $V_{i}/V_{i-1}$ is semi-simple for each $1\leq i \leq n$. Then the following holds for each $1\leq i \leq n$:
\begin{enumerate}[label=(\roman*)]
\item $V_{i}/V_{i-1}$ admits a semi-simple weight space decomposition of compact type.
    \item $V_i$ admits a weight space decomposition of compact type for the family:
    \begin{equation*}
        I_{V_i}:= \{ \lambda \in I \textnormal{ }\vert \textnormal{ } V^{\lambda}\cap V_{i}\neq 0  \}.
    \end{equation*}
    Furthermore, we have $V_{i}^{\lambda}=V_i\cap V^{\lambda}$ for all $\lambda \in I_{V_i}$.
    \item If the decomposition is of finite type, there is a $\partial$-equivariant filtration:
    \begin{equation*}
        0\subset \bigoplus_{\lambda\in I_{V_1}}V_{1}^{\lambda} \subset \cdots  \subset\bigoplus_{\lambda\in I_{V_n}}V_{n}^{\lambda},
    \end{equation*}
such that we have the following graded quotients for $1\leq i \leq n$:
\begin{equation*}
   \left(\bigoplus_{\lambda\in I_{V_i}}V_{i}^{\lambda}\right)/\left(\bigoplus_{\lambda\in I_{V_{i-1}}}V_{i-1}^{\lambda}\right)=(V_i/V_{i-1})^{\operatorname{ws}}. 
\end{equation*}  
\end{enumerate}
Thus, $I=\Lambda(V)$, and  $\sum_{\lambda\in\Lambda(V)}v_{\lambda}=0$ if and only if $v_{\lambda}=0$ for all $\lambda\in\Lambda(V)$.
\end{prop}
\begin{proof}
Claims $(i)$ to $(iii)$ follow from an inductive application of Lemma $\ref{Lemma ws decomposition and filtrations}$. Let $\mu \in K\setminus I$, and $v\in V^{\mu}$. There is a minimal $1\leq i \leq n$ such that $v\in V_i$. Let $\pi_i:V_i\rightarrow V_i/V_{i-1}$ be the quotient. Our minimality assumption implies that $\pi_i(v)\neq 0$. However, we have $\pi_i(v)\in (V_i/V_{i-1})^{\mu}$, and by statement $(iii)$ we have $\Lambda(V_i/V_{i-1})\subset I$. Thus, it follows that $(V_i/V_{i-1})^{\mu}=0$, and therefore $v=0$. Hence, we must have $I=\Lambda(V)$, as desired. Consider a trivial sum of weight vectors $0=\sum_{\lambda\in\Lambda(V)}v_{\lambda}$. We need to show $v_{\lambda}=0$ for all $\lambda\in\Lambda(V)$. For simplicity, we may assume $V=V_1$, the general case is analogous. By Theorem \ref{teo weight space decomposition ss case}, the identity $\sum_{\lambda\in\Lambda(V)}\pi_1(v_{\lambda})=0$ implies $\pi_1(v_{\lambda})=0$. Hence, $v_{\lambda}\in V_0$ for all $\lambda\in\Lambda(V)$. However, by assumption, $V_0$ also satisfies the conditions of Theorem \ref{teo weight space decomposition ss case}, so it follows that $v_{\lambda}=0$ for all $\lambda\in\Lambda(V)$, as we wanted to show.
\end{proof}
Notice that this proposition is slightly weaker than the semi-simple version shown in Theorem \ref{teo weight space decomposition ss case}. In particular, it does not allow us to deduce the existence of weight space decompositions for closed $\partial$-invariant subspaces of $V$. However, we will see in Section \ref{Section banach algebras with triangular decomp} that this weak version is enough for our purposes.
\section{\texorpdfstring{Category $\wideparen{\OX}$ for Fréchet-Stein algebras with a triangular decomposition}{}}\label{Section decompositions}
\subsection{\texorpdfstring{Banach algebras with a triangular decomposition}{}}\label{Section banach algebras with triangular decomp}
After this analytic detour, we will apply these results to the representation theory of Banach algebras. 
\begin{defi}\label{defi triangular decomposition of Banach algebras}
Let $\mathscr{R}$ be a noetherian Banach $K$-algebra. A triangular decomposition of $\mathscr{R}$ consists of the data $(R,A,B,H,\partial)$, satisfying the following conditions:
\begin{enumerate}[label=(\roman*)]
    \item $R\subset \mathscr{R}$ is a dense graded subalgebra.
    \item $A,B,H\subset R$ are graded subalgebras, $\partial\in R$ is an element.
    \item $R$ admits a triangular decomposition $(A,B,H,\partial)$.
\end{enumerate}
Let $\mathscr{A},\mathscr{B}$ be the completions of $A$ and $B$ in $\mathscr{R}$, and denote the commutator of $\partial$ by  $\operatorname{ad}(\partial)=[\partial,-]$. This data must satisfy the following requirements:
\begin{enumerate}[label=(\roman*)]
    \item There is a decomposition of Banach spaces $\mathscr{R}=\mathscr{A}\widehat{\otimes}_KH\widehat{\otimes}_K\mathscr{B}.$
    \item $\mathscr{A}$ has a semi-simple weight space decomposition for $\operatorname{ad}(\partial)$, and $\mathscr{A}^{\operatorname{ws}}=A$.
    \item $\mathscr{B}$ has a semi-simple weight space decomposition for $\operatorname{ad}(\partial)$, and  $\mathscr{B}^{\operatorname{ws}}=B$.
\end{enumerate}
\end{defi}
The goal of this section is showing that every noetherian Banach $\mathscr{R}$-algebra with a triangular decomposition (plus some mild technical conditions) admits a category $\widehat{\OX}$, with properties analogous to the algebraic version studied above. For the rest of this section we fix a noetherian Banach algebra $\mathscr{R}$ with a triangular decomposition $(R,A,B,H,\partial)$.  Let us start by analyzing the properties of $\mathscr{R}$:
\begin{prop}\label{prop definition AH BH}
The following hold:
\begin{enumerate}[label=(\roman*)]
    \item $\mathscr{A}$ and $\mathscr{B}$ are Banach $K$-algebras.
    \item $\mathscr{A}\cdot H=H\cdot \mathscr{A}$, and $\mathscr{B}\cdot H=H\cdot \mathscr{B}$.
\end{enumerate}
As before, we let $\mathscr{A}H$ and $\mathscr{B}H$ be the algebras $\mathscr{A}\widehat{\otimes}_KH$ and $\mathscr{B}\widehat{\otimes}_KH$.
\end{prop}
\begin{proof}
By definition we have $\mathscr{A}=\widehat{A}$, and $\mathscr{B}=\widehat{B}$, and $\mathscr{R}=\mathscr{A}\widehat{\otimes}_KH\widehat{\otimes}_K\mathscr{B}$. Thus, all the statements follow from their algebraic counterparts (\emph{cf}. Definition \ref{defi notions of triang decomp}). 
\end{proof}
Notice that we did not require that $\mathscr{R}$ admits a weight space decomposition for $\operatorname{ad}(\partial)$. As it turns out, this does not follow directly from the definitions, but we can make some additional assumptions on $\mathscr{A}$ and $\mathscr{B}$ which assure that this will hold.  Regard $\mathscr{C}=\mathscr{A}\widehat{\otimes}_K\mathscr{B}$ as a Banach space. We have:
\begin{Lemma}\label{Lemma R admits a ws decomposition}
Choose any $a\otimes b \in \mathscr{C}$. Then we have:
\begin{equation*}
    \operatorname{ad}(\partial)(a\otimes b)=\operatorname{ad}(\partial)a\otimes b + a \otimes \operatorname{ad}(\partial)b.
\end{equation*}
In particular, the subspace  $\mathscr{C}\subset \mathscr{R}$ is closed under the action of $\operatorname{ad}(\partial)$.
\end{Lemma}
\begin{proof}
The fact that $\operatorname{ad}(\partial)(a\otimes b)=\operatorname{ad}(\partial)a\otimes b + a\otimes \operatorname{ad}(\partial)b$ is clear, and implies $\operatorname{ad}(\partial)(\mathscr{A}\otimes_K\mathscr{B})\subset \mathscr{C}$. As $\mathscr{A}\otimes_K\mathscr{B}$ is dense in $\mathscr{C}$, the result follows.
\end{proof}
\begin{prop}\label{prop ws decomposition of C}
Assume that $\mathscr{A}$ and $\mathscr{B}$ satisfy the following condition:
\begin{itemize}
    \item Let $v=\sum_{\lambda\in I} v_{\lambda}$ with $\vert v \vert \leq \epsilon$. Then $\vert v_{\lambda}\vert \leq \epsilon$ for all $\lambda \in I$.
\end{itemize}
Then $\mathscr{C}$ admits a semi-simple weight space decomposition for $\operatorname{ad}(\partial)$ with $\Lambda(\mathscr{C})\subset\mathbb{Z}$. Furthermore, for any $n\in \mathbb{Z}$ we have:
\begin{equation*}
    \mathscr{C}^n=\widehat{\bigoplus}_{m\geq n}A_m\otimes B^{m-n}.
\end{equation*}
\end{prop}
\begin{proof}
First, by Lemma \ref{Lemma R admits a ws decomposition}, given $a\in A_n$, and $b\in B^m$ we have $a\otimes b \in \mathscr{C}^{n-m}$ for all $n,m\geq 0$. Let $v\in \mathscr{C}$. By construction of the completed tensor product of Banach spaces, there is a family $(a_i\otimes b_i)_{i\geq 0}\in \mathscr{A}\otimes_K\mathscr{B}$ satisfying the identities:
\begin{equation}\label{equation C admits a decomposition}
    v=\sum_{i\geq 0}a_i\otimes_Kb_i, \textnormal{  } \varinjlim_{i\geq 0}a_i\otimes_Kb_i=0.
\end{equation}
By definition of triangular decomposition, for each $i\geq 0$ there are unique identities:
\begin{equation*}
    a_i=\sum_{n\geq 0} a_i^n, \quad b_i=\sum_{m\geq 0} b_i^m,
\end{equation*}
such that $a_i^n\in A_n$, and $b_i^m\in B^m$ for all $n,m\geq 0$. As these are infinite sums:
\begin{equation}\label{equation 3 decomposition of C into ws}
    \varinjlim_n a_i^n=\varinjlim_n b_i^m=0.
\end{equation}
Hence, we may rearrange the terms in equation $(\ref{equation C admits a decomposition})$ to obtain:
\begin{equation*}
    v=\sum_{r\in \mathbb{Z}}\left(\sum_{i\geq 0}\sum_{r=n-m}a_i^n\otimes b_i^m  \right).
\end{equation*}
We are done if we are able to show that for each $r\in \mathbb{Z}$ the sum:
\begin{equation}\label{equation 2 C admits a ws decomposition}
    \sum_{i\geq 0}\sum_{r=n-m}a_i^n\otimes b_i^m,
\end{equation}
converges. It suffices to show that the principal term of $(\ref{equation 2 C admits a ws decomposition})$ converges to zero. As $\varinjlim_{i\geq 0}a_i\otimes_Kb_i=0$, for each $s\geq 0$ there is some $j\geq 0$ such that for all $i\geq j$ we have $\vert a_i\otimes_Kb_i \vert \leq \vert \pi \vert^{2s}$. In particular, we may assume that $\operatorname{max}(\vert a_i\vert, \vert b_i\vert)\leq \vert \pi \vert^s$ for all $i\geq j$. Thus, our assumptions on $\mathscr{A}$ and $\mathscr{B}$ imply that:
\begin{equation*}
    \operatorname{max}\{\vert a^n_i\vert, \vert b^m_i\vert\}_{i\geq j}^{n,m\geq 0}\leq \vert \pi \vert^s.
\end{equation*}
As there are finitely many $0\leq i \leq j$, it follows by $(\ref{equation 3 decomposition of C into ws})$ that there is $\overline{n}\geq 0$ such that:
\begin{equation*}
    \operatorname{max}\{\vert a^n_i\vert, \vert b^m_i\vert\}_{0\leq i\leq j}^{n,m\geq \overline{n}}\leq \vert \pi \vert^s.
\end{equation*}
Hence, the principal term of $(\ref{equation 2 C admits a ws decomposition})$ converges to zero, as we wanted to show.
\end{proof}
\begin{coro}
In the conditions of the previous proposition, $\mathscr{R}$ admits a semi-simple weight space decomposition for $\operatorname{ad}(\partial)$. Furthermore, we have $R\subset \mathscr{R}^{\operatorname{ws}}$.
\end{coro}
\begin{proof}
By definition, we have $\mathscr{R}=\mathscr{A}\widehat{\otimes}_KH\widehat{\otimes}_K\mathscr{B}$. As $H$ is finite-dimensional and $H\cdot\mathscr{B}=\mathscr{B}\cdot H$, this yields an identity:
\begin{equation*}
    \mathscr{R}=\mathscr{C} \otimes_KH.
\end{equation*}
By Proposition \ref{prop ws decomposition of C}, $\mathscr{C}$ admits a semi-simple weight space decomposition for $\operatorname{ad}(\partial)$, and $H\subset \mathscr{R}^{0}$. Thus,  $\mathscr{R}$ admits a semi-simple weight space decomposition, as wanted. The fact that $R\subset \mathscr{R}^{\operatorname{ws}}$ follows from the identity $R=A\otimes_KH\otimes_KC$.
\end{proof}
\begin{obs}
The weight space decomposition of $\mathscr{R}$ is only of finite type if there is some $m\in\mathbb{Z}^{\geq 0}$ such that $A_n=0$ for all $n\geq m$ or $B^n=0$ for all $n\geq m$.
\end{obs}
The rest of this section is devoted to studying the representation theory of $\mathscr{R}$. For convenience, we will use the following notation:
\begin{obs}
In this section $\Mod(\mathscr{R})$ will denote the category of finitely generated $\mathscr{R}$-modules. All weight space decompositions will be with respect to $\partial$, so we omit it from the notation. Additionally, we will assume that $\mathscr{A}$ and $\mathscr{B}$ are noetherian Banach $K$-algebras.
\end{obs}
As $\mathscr{R}$ is noetherian and Banach, every $M\in \Mod(\mathscr{R})$  carries a canonical Banach topology, which is induced by any surjection $\bigoplus_{i=1}^r\mathscr{R}\rightarrow M$. Given $M\in \Mod(\mathscr{R})$, we write $M^{\operatorname{ws}}$
for the direct sum of the generalized weight spaces with respect to the action of $\partial$ (\emph{cf.} Definition \ref{defi ws functor}). The goal of this section is defining a full abelian subcategory of $\Mod(\mathscr{R})$ which is analogous to the category $\OX$ for $R$ defined in Section \ref{Section category O over p-adic fields}. We will start by introducing the analytic analogs of Verma modules:
\begin{Lemma}
The following hold:
\begin{enumerate}[label=(\roman*)]
    \item $\mathscr{B}H$ is a noetherian Banach $K$-algebra.
    \item $\mathscr{B}H$ admits a semi-simple weight space decomposition of finite type.
    \item For each $\lambda \in K$ we have $\mathscr{B}H^{\lambda}=B^{\lambda}\otimes H$. In particular, we have:
    \begin{equation*}
        \mathscr{B}H^{\operatorname{ws}}=BH, \quad \Lambda(\mathscr{B}H)=\mathbb{Z}^{\leq 0}.
    \end{equation*}
\end{enumerate}
\end{Lemma}
\begin{proof}
 As $\mathscr{B}$ is noetherian and $H$ is finite-dimensional, it follows that $\mathscr{B}H$ is finite over $\mathscr{B}$. Hence, it is also a noetherian Banach $K$-algebra. Choose $b\in B^n$, and $v\in H$, then $b\otimes v\in B^n\otimes H$, and we have the following identity in $R$:
 \begin{equation*}
     \operatorname{ad}(\partial)(b\otimes v)=(\partial b-b\partial)\otimes v +b\otimes (\partial v-v\partial)=-n(b\otimes v).
 \end{equation*}
Thus, $BH$ is stable under the action of $\operatorname{ad}(\partial)$, and we have $BH^{-n}=B^n\otimes_K H$. As the action of $\operatorname{ad}(\partial)$ is continuous, $\mathscr{B}H$ is also stable under the action of $\operatorname{ad}(\partial)$. Furthermore, any element $\alpha\in \mathscr{B}H$ admits the following expression:
\begin{equation*}
    \alpha= \sum_{i=1}^mb_i\otimes v_i=\sum_{n\geq 0}\left(\sum_{i=1}^mb_i^n\otimes v_i\right),
\end{equation*}
where $b_i^n\in B^n$ for all $1\leq i \leq m$ and all $n\geq 0$. Thus, 
$\mathscr{B}H$ admits a semi-simple weight space decomposition of finite type for $\operatorname{ad}(\partial)$, with weight spaces:
\begin{equation*}
    \mathscr{B}H^{\operatorname{ws}}=\bigoplus_{n\leq 0}\mathscr{B}H^n=\bigoplus_{n\geq 0}B^n\otimes_KH,
\end{equation*}
and this is precisely what we wanted to show.
\end{proof}
\begin{Lemma}\label{Lemma short exact sequence quotients}
For $n\geq 0$, let $\mathscr{B}H^{>n}$ be the two-sided ideal in $\mathscr{B}H$ generated by:
\begin{equation*}
    BH^{>n}:=B^{>n}\otimes_KH.
\end{equation*}
There is a short exact sequence:
\begin{equation*}
  0\rightarrow \mathscr{B}H^{>n}\rightarrow \mathscr{B}H \rightarrow BH/BH^{>n}\rightarrow 0.  
\end{equation*}
\end{Lemma}
\begin{proof}
As $\mathscr{B}H$ is noetherian and Banach, it follows that $\mathscr{B}H^{>n}$ is closed in $\mathscr{B}H$. Furthermore, $BH^{>n}$ is dense in $\mathscr{B}H^{>n}$. By the description of the weight spaces of $\mathscr{B}H$ given in the previous lemma we have the following identities:
\begin{align*}  \mathscr{B}H^{>n}=&\widehat{\bigoplus}_{m< -n}\mathscr{B}H^m, \\\mathscr{B}H/\mathscr{B}H^{>n}=\bigoplus_{-n \leq m\leq 0}\mathscr{B}H^m=&\bigoplus_{-n \leq m\leq 0}BH^m=BH/BH^{>n}.
\end{align*}
So the result holds.
\end{proof}
These lemmas allow us to define analytic Verma modules:
\begin{defi}
Let $W\in \operatorname{Irr}(H)$. The analytic Verma module of $W$ is the following $\mathscr{R}$-module:
\begin{equation*}
    \widehat{\Delta}(W):= \mathscr{R}\widehat{\otimes}_{\mathscr{B}H}\left(\mathscr{B}H/\mathscr{B}H^{>0}\widehat{\otimes}_{H} W \right)=\mathscr{A}\widehat{\otimes}_KW.
\end{equation*}
\end{defi}
And one can readily check that the following holds:
\begin{prop}\label{prop ws analytic verma}
For $W\in \operatorname{Irr}(H)$, the module $\widehat{\Delta}(W)$ admits a semi-simple weight space decomposition of finite type. Moreover, we have $\widehat{\Delta}(W)^{\operatorname{ws}}=\Delta(W)$.
\end{prop}
Next, we would like to analyze the structure of the quotients of analytic Verma modules. In particular, we aim to study their simple quotients. As we have an explicit description of their weight spaces, our strategy boils down to using Theorem \ref{teo weight space decomposition ss case} to reduce the problem to the algebraic situation, where matters are simpler.
\begin{Lemma}\label{Lemma analytification of Verma}
Let $W\in\operatorname{Irr}(H)$. Then we have an identity of $\mathscr{R}$-modules:
\begin{equation*}
    \widehat{\Delta}(W)=\mathscr{R}\otimes_{R}\Delta(W).
\end{equation*}
\end{Lemma}
\begin{proof}
By definition of Verma modules, we have the following identities:
\begin{align*}
\mathscr{R}\otimes_{R}\Delta(W)=\mathscr{R}\otimes_{R}\left(R\otimes_{BH}\left(BH/BH^{>0}\otimes_{H} W \right)  \right)\\
= \mathscr{R}\widehat{\otimes}_{\mathscr{B}H}\left(\mathscr{B}H\otimes_{BH}\left(BH/BH^{>0}\otimes_{H} W \right)\right)\\
=\mathscr{R}\widehat{\otimes}_{\mathscr{B}H}\left(\mathscr{B}H/\mathscr{B}H^{>0}\widehat{\otimes}_{H} W \right)\\
=\widehat{\Delta}(W)&.
\end{align*}
The second identity holds because $W$ is finite, so $\mathscr{B}H\otimes_{BH}\left(BH/BH^>0\otimes_{H} W \right)$ is a finite $\mathscr{B}H$-module. As $\mathscr{R}$ is  noetherian and Banach, the completed tensor product agrees with the uncompleted tensor product on finite modules.
\end{proof}
\begin{prop}\label{prop simple modules in OX}
Let $W\in \operatorname{Irr}(H)$. Then  $\widehat{\Delta}(W)$ has a unique maximal submodule, and hence a unique simple quotient $\widehat{L}(W)$. This simple quotient satisfies:
\begin{equation*}
    \widehat{L}(W)=\mathscr{R}\otimes_{R}L(W),
\end{equation*}
where $L(W)$ is the unique simple quotient of $\Delta(W)$. 
\end{prop}
\begin{proof}
Let  $\widehat{\Delta}(W)$ be an analytic Verma module. By Proposition \ref{prop ws analytic verma} we know:
\begin{equation*}
    \Lambda(\widehat{\Delta}(W))= c(w)+\mathbb{Z}^{\geq 0}, \textnormal{  } \operatorname{dim}\left(\widehat{\Delta}(W)^{c(W)}\right)=1.
\end{equation*}
As every $v\in \widehat{\Delta}(W)^{c(W)}$ generates $\widehat{\Delta}(W)$
as a $\mathscr{R}$-module, and every $\mathscr{R}$-submodule of  $\widehat{\Delta}(W)$ admits a finite type decomposition into weight spaces for $\partial$ (\emph{cf.} Theorem \ref{teo weight space decomposition ss case} and Proposition \ref{prop ws analytic verma}), it follows that the sum of all proper submodules of  $\widehat{\Delta}(W)$ has trivial intersection with $\widehat{\Delta}(W)^{c(W)}$, and therefore is a maximal proper submodule. Denote this maximal module by $\widehat{M}(W)$, and let $M(W)$ be the maximal submodule of $\Delta(W)$. Then we have the following commutative diagram:
\begin{equation*}
\begin{tikzcd}
            & \mathscr{R}\otimes_RM(W) \arrow[r] \arrow[d] & \mathscr{R}\otimes_R\Delta(W) \arrow[r] \arrow[d] & \mathscr{R}\otimes_RL(W) \arrow[r] \arrow[d] & 0 \\
0 \arrow[r] & \widehat{M}(W) \arrow[r]                     & \widehat{\Delta}(W) \arrow[r]                     & \widehat{L}(W) \arrow[r]                     & 0
\end{tikzcd}
\end{equation*}
where the middle map is an isomorphism by Lemma \ref{Lemma analytification of Verma}. As $\widehat{L}(W)$ is simple, $\mathscr{R}\otimes_RL(W)\rightarrow \widehat{L}(W)$ is surjective. It is an isomorphism if and only if the map:
\begin{equation*}
    \varphi:\mathscr{R}\otimes_{R}M(W)\rightarrow \widehat{M}(W),
\end{equation*}
is a surjection. As all modules are finite, it suffices to show that it has dense image. By construction, we have $\widehat{M}(W)^{\operatorname{ws}}\subset \widehat{\Delta}(W)^{\operatorname{ws}}= \Delta(W)$. As $M(W)$ is the maximal submodule of $\Delta(W)$, it follows that $\widehat{M}(W)^{\operatorname{ws}}\subset M(W)$. Hence, we have:
\begin{equation*}
    \mathscr{R}\otimes_{R}\widehat{M}(W)^{\operatorname{ws}}\rightarrow \mathscr{R}\otimes_{R}M(W)\rightarrow \widehat{M}(W).
\end{equation*}
As $\widehat{M}(W)^{\operatorname{ws}}$ is dense in $\widehat{M}(W)$, $\varphi$ has dense image, and therefore is surjective.
\end{proof}
\begin{coro}\label{coro ws of simple modules}
Let $W\in \operatorname{Irr}(H)$. We have an identification: 
\begin{equation*}
    \widehat{L}(W)^{\operatorname{ws}}=L(W).
\end{equation*}
\end{coro}
\begin{proof}
By the previous proposition, we have maps:
\begin{equation*}
    L(W)\rightarrow\mathscr{R}\otimes_{R}L(W)\rightarrow \widehat{L}(W),
\end{equation*}
where the second map is an isomorphism, and the first has dense image. As $L(W)$ is simple, it follows that $L(W)\rightarrow \mathscr{R}\otimes_{R}L(W)$ is injective, and has image contained in $(\mathscr{R}\otimes_{R}L(W))^{\operatorname{ws}}$. Hence, Corollary \ref{coro surjectiveness of ws} implies that $\widehat{L}(W)^{\operatorname{ws}}=L(W)$.
\end{proof}

Now that we have a good understanding of the structure of analytic Verma modules, the next step is studying how they map to objects in $\Mod(\mathscr{R})$.
\begin{defi}[{\cite[Definition 3.23.]{etingof2010lecture}}]
Let $M\in \Mod(\mathscr{R})$. An element $v\in M$ is singular if $B^{>0}v=0$. We let $M^{\operatorname{Sing}}\subset M$ be the space of singular vectors in $M$.
\end{defi}
The singular vectors play the same role in this theory as the highest weight vectors play in the category $\OX$ of a semi-simple Lie algebra. 
\begin{prop}\label{prop maps from Verma}
Let $W\in\operatorname{Irr}(H)$. For any $M\in \Mod(\mathscr{R})$ we have:
\begin{equation*}
   \Hom_{\mathscr{R}}(\widehat{\Delta}(W),M)=\{ v\in M^{\operatorname{Sing}} \textnormal{ }\vert\textnormal{ } H\cdot v\cong W \textnormal{ or } H\cdot v=0 \}. 
\end{equation*}
\end{prop}
\begin{proof}
The proof is purely formal, as we have:
\begin{align*}
    \Hom_{\mathscr{R}}(\widehat{\Delta}(W),M)=&\Hom_{\mathscr{B}H}(\mathscr{B}H/\mathscr{B}H^{>0}\widehat{\otimes}_{H} W,M)\\=&
    \Hom_{H}(W,\Hom_{\mathscr{B}H}(\mathscr{B}H/\mathscr{B}H^{>0},M))\\=&\Hom_{H}(W,\Hom_{BH}(BH/BH^{>0},M)). 
\end{align*}
Next, we want to show the following identity:
\begin{equation*}
    \Hom_{BH}(BH/BH^{>0},M)=M^{\operatorname{Sing}}.
\end{equation*}
As $B^{>0}\subset BH^{>0}$, any morphism $BH/BH^{>0}\rightarrow M$ is uniquely determined by some $v\in M^{\operatorname{Sing}}$. We need to show that any $v\in M^{\operatorname{Sing}}$ determines a map $BH/BH^{>0}\rightarrow M$. Equivalently, we need to show $BH^{>0}v=0$. Choose some $b\otimes h\in BH^{>0}$. By definition of triangular decomposition, $B\cdot H =H\cdot B$. Hence, we have:
\begin{equation*}
    b\otimes h=\sum_{i=1}^nh_i\otimes b_i,
\end{equation*}
for some $h_i\in H$ and $b_i\in B$. The fact that $R$ is a graded algebra and that $H\subset R_0$ implies that we can assume $b_i\in B^{>0}$ for all $1\leq i \leq n$. Thus, $bh\cdot v=0$, and the identity holds. Hence, we arrive at the following identity:
\begin{equation*}
   \Hom_{\mathscr{R}}(\widehat{\Delta}(W),M)= \Hom_{H}(W, M^{\operatorname{Sing}}).
\end{equation*}
As $W$ is an irreducible $H$-module, every $H$-linear map $W\rightarrow M^{\operatorname{Sing}}$ is either zero or an isomorphism, and is determined by the action of $H$ on $v\in M^{\operatorname{Sing}}$.
\end{proof}
\begin{coro}
Let $M\in \Mod(\mathscr{R})$. Then $M^{\operatorname{Sing}}$  is an $H$-module.   
\end{coro}
We are now finally ready to introduce the analytic version of the category $\OX$:
\begin{defi}
Let $\widehat{\OX}$ be the full subcategory of $\Mod(\mathscr{R})$ given by modules $M$  satisfying the following conditions:
\begin{enumerate}[label=(\roman*)]
    \item $M$ admits a  weight space decomposition.
    \item $M$ is a finite $\mathscr{A}$-module.
\end{enumerate}
\end{defi}
A priory, it is not clear that $\widehat{\OX}$ is a highest weight category. In particular, it is not even clear that it is an abelian category. This stems from the fact that Theorem \ref{teo weight space decomposition ss case} requires the decompositions to be semi-simple, and we are working in a more general setting. However, the fact that objects in $\widehat{\OX}$ are finite $\mathscr{A}$-modules imposes severe restrictions on the weight spaces: 
\begin{prop}\label{prop generating system by eigenvectors}
Let $M\in\widehat{\OX}$. The following hold:
\begin{enumerate}[label=(\roman*)]
    \item $M$ is finitely generated by generalized eigenvectors over $\mathscr{A}$.
    \item There is a finite family  of weights $\{\lambda_i\}_{i=1}^n\subset \Lambda(M)$  such that $M$ admits a weight space decomposition with respect to $I=\bigcup_{i=1}^n\lambda_i+\mathbb{Z}^{\geq 0}$. 
\end{enumerate}
\end{prop}
\begin{proof}
By definition of $\widehat{\OX}$, the set of all generalized eigenvectors of $\partial$ is a dense $K$-vector subspace of $M$. As $\mathscr{A}$ is noetherian and Banach, and $M$ is finitely generated, it follows that the $\mathscr{A}$-module generated by generalized eigenvectors of $\partial$ is dense and closed in $M$. Thus, it agrees with $M$, so we can choose a finite family of generalized eigenvectors $x_1,\cdots,x_n$ which generate $M$ as an $\mathscr{A}$-module. Assume $x_i\in M^{\lambda_i}$ for all $1\leq i\leq n$. We claim that $\{\lambda_i\}_{i=1}^n$ satisfies the properties in $(ii)$.  Indeed, every $x\in M$ admits an expression $x=\sum_{i=1}^na_ix_i$, with $a_i\in \mathscr{A}$. By assumption, the action of $\operatorname{ad}(\partial)$ on $\mathscr{A}$ admits a semisimple weight space decomposition of compact type. Hence, we have $a_i=\sum_{m\geq 0}a_i^m$, where $a_i^m\in A_m$ for each $m\geq 0$. In particular, for each $1\leq i \leq n$ and each $m\geq 0$ we have $a_i^mx_i\in M^{\lambda_i+m}$. As the family $\{\lambda_i \}_{i=1}^n$ is finite, we have the following identity:
\begin{equation}\label{equation ws decomposition objects of OX}
 x=\sum_{i=1}^na_ix_i=\sum_{i=1}^n\sum_{m\geq 0}a_i^mx_i=\sum_{\lambda \in I}  \left(\sum _{i=1}^na_i^{\lambda_i-\lambda }x_i\right),
\end{equation}
where $a_i^{\lambda_i-\lambda }=0$ if $\lambda_i-\lambda\not \in \mathbb{Z}^{\geq 0}$. Thus, $\sum _{i=1}^na_i^{\lambda_i-\lambda }x_i\in M^{\lambda}$ for each $\lambda \in I$.
\end{proof}
The explicit structure of the decomposition found in Proposition \ref{prop generating system by eigenvectors} is instrumental for our understanding of $\widehat{\OX}$. In particular, the idea is using this decomposition to show that every object in $\widehat{\OX}$ admits a filtration by objects in $\widehat{\OX}$ such that its graded quotients are quotients of analytic Verma modules. 
\begin{Lemma}\label{Lemma start of the filtration by quotients of Verma}
 The category $\widehat{\OX}$ is closed under quotients by $\mathscr{R}$-submodules.
\end{Lemma}
\begin{proof}
Let $M\in \widehat{\OX}$ and let $N\subset M$ be a $\mathscr{R}$-submodule. As $M$ is a finite $\mathscr{A}$-module, it is clear that $M/N$ is a finite $\mathscr{A}$-module. The proof that $M/N$ admits a weight space decomposition is analogous to the proof of Lemma \ref{Lemma ws decomposition and filtrations}.
\end{proof}
\begin{Lemma}\label{Lemma filtrations by analytic Verma 1}
For $M\in \widehat{\OX}$ and $\lambda\in K$ consider the following $K$-vector space:
\begin{equation*}
    M_{\operatorname{ss}}^{\lambda}:=\{m\in M^{\lambda} \textnormal{ }\vert \textnormal{ } \partial m=\lambda m    \}
\end{equation*}
Then $M_{\operatorname{ss}}^{\lambda}$ is $\partial$-invariant. Furthermore, if $M^{\lambda}\neq 0$, then $M_{\operatorname{ss}}^{\lambda}\neq 0$.
\end{Lemma}
\begin{proof}
Assume $M^{\lambda}\neq 0$, and let $v\in M^{\lambda}$. By definition, there is some $n\geq 0$ such that $(\partial-\lambda)^nv=0$. In particular, the $K$-vector space $V=\sum_{n\geq 0}(\partial-\lambda)^nv$ has finite dimension. By construction, $V$ is invariant under the action of $(\partial-\lambda)$, and it is clearly invariant under the action of $\lambda$. Hence, it is a $\partial$-invariant finite dimensional space such that $(\partial -\lambda)^n_{\vert V}=0$. But then it follows from the existence of the Jordan canonical form that there is some $w\in V$ such that $\partial w=\lambda w$.
\end{proof}
\begin{Lemma}\label{Lemma filtrations by analytic Verma 2}
Let $M\in \widehat{\OX}$. Choose a finite family $\{\lambda_i\}_{i=1}^n\subset \Lambda(M)$, and let $I=\bigcup_{i=1}^n\lambda_i+ \mathbb{Z}^{\geq 0}$. Consider the following $K$-vector space:
\begin{equation*}
    M^I_{\operatorname{ss}}:=\{ m\in M \textnormal{ }\vert \textnormal{ } m=\sum_{\lambda \in I} m_{\lambda}, \textnormal{ where } m_{\lambda}\in M^{\lambda}_{\operatorname{ss}} \textnormal{ for all }\lambda \in I\}.
\end{equation*}
Then $M^I_{\operatorname{ss}}$ is a $\partial$-invariant $\mathscr{A}$-submodule of $M$. In particular, it is a Banach space.
\end{Lemma}
\begin{proof}
It is clear that $M^I_{\operatorname{ss}}$ is a $K$-vector space. We need to see it is closed under the action of $\mathscr{A}$. First, notice that for $a\in A_n$ and $v\in M^{\lambda}_{\operatorname{ss}}$ we have $\partial (av)=(n+\lambda)av$. Hence, we have $A_n \cdot M^{\lambda}_{\operatorname{ss}}\subset M^{\lambda+n}_{\operatorname{ss}}$. Next, let $x\in M^I_{\operatorname{ss}}$, and $a\in \mathscr{A}$. We have:
\begin{equation*}
    x=\sum_{\lambda\in I}m_{\lambda}, \textnormal{ } a=\sum_{n\geq 0}a_n.
\end{equation*}
Thus, proceeding in the same way as in equation $(\ref{equation ws decomposition objects of OX})$ we get:
\begin{equation*}
    ax=\sum_{\lambda \in I}\left(\sum_{n+ \gamma = \lambda}a_n m_{\gamma}\right).
\end{equation*}
As the family $\{\lambda_i\}_{i=1}^n$ is finite, for each $\lambda \in I$ there are only finitely many $\gamma \in I$ such that $\lambda -\gamma \in \mathbb{Z}^{\geq 0}$. Thus, each of the terms $\sum_{n+ \gamma = \lambda}a_n m_{\gamma}$ is a well-defined element of $M_{\operatorname{ss}}^{\lambda}$, so $ax\in M^I_{\operatorname{ss}}$. The fact that  $M^I_{\operatorname{ss}}$ is $\partial$-invariant is straightforward.
\end{proof}
We use this to obtain a description of $\widehat{\OX}$ in terms of analytic Verma modules:
\begin{prop}\label{prop filtrations by quot of verma}
Let $M\in \widehat{\OX}$. Then $M$ admits a finite filtration by $\mathscr{R}$-modules:
\begin{equation*}
    0=M_0\subset M_1\subset \cdots\subset M_n=M,
\end{equation*}
where $M_{i}/M_{i-1}$ is a quotient of $\widehat{\Delta}(W_i)$ for $W_i\in \operatorname{Irr}(H)$. We call this a $\widehat{\Delta}$-filtration. 
\end{prop}
\begin{proof}
As $\mathscr{R}$ is noetherian and $M$ is finite, every ascending chain of submodules in $M$ is stationary. Furthermore, by Lemma \ref{Lemma start of the filtration by quotients of Verma}, $\widehat{\OX}$ is closed under quotients by $\mathscr{R}$-submodules. Thus, it suffices to show that for every $M\in \widehat{\OX}$ there is some $W\in\operatorname{Irr}(H)$ with a non-zero map $\widehat{\Delta}(W)\rightarrow M$. By Proposition \ref{prop maps from Verma}, this amounts to showing that $M^{\operatorname{Sing}}\neq 0$. Choose a countable family $\{\lambda_i\}_{i=1}^{\infty}\subset \Lambda(M)$. For each $n\geq 1$, let $I_n=\bigcup_{i=1}^n\lambda_i+ \mathbb{Z}^{\geq 0}$. By Lemma \ref{Lemma filtrations by analytic Verma 2}, we have an increasing sequence of $\partial$-invariant    $\mathscr{A}$-submodules of $M$:
\begin{equation*}
   0 \subset M_{\operatorname{ss}}^{I_1}\subset \cdots \subset M_{\operatorname{ss}}^{I_n}\subset \cdots
\end{equation*}
As $\mathscr{A}$ is noetherian and $M$ is finite, there is $m\geq 1$ that  stabilizes the sequence. Let $M_{\operatorname{ss}}:=M_{\operatorname{ss}}^{I_m}$. By Lemma \ref{Lemma filtrations by analytic Verma 2},  $M_{\operatorname{ss}}$ is a Banach space with a semi-simple weight space decomposition of compact type. Thus, by Theorem \ref{teo weight space decomposition ss case} the weight spaces of $M_{\operatorname{ss}}$ are uniquely determined by the decomposition. In other words, we have:
\begin{equation}\label{equation determination of set of weights}
    \Lambda(M_{\operatorname{ss}})\subset \bigcup_{i=1}^m\lambda_i+\mathbb{Z}^{\geq 0}.
\end{equation}
By Lemma \ref{Lemma filtrations by analytic Verma 1}, for any $\lambda \in \Lambda(M)$ we have $M_{\operatorname{ss}}^{\lambda}\neq 0$. But then, by maximality of $M^{\operatorname{ss}}$ and equation $(\ref{equation determination of set of weights})$, it follows that $\Lambda(M)\subset \bigcup_{i=1}^m\lambda_i+\mathbb{Z}^{\geq 0}$. Choose any $\lambda_i$ such that $\lambda_i-\lambda_j \not\in \mathbb{Z}^{> 0}$ for all $1\leq j \leq n$. For every $n\geq 0$, and $b\in B^n$, we have $[\partial,b]=-nb$. Thus, our choice of $\lambda_i$ implies that $M^{\lambda_i}\subset M^{\operatorname{Sing}}$.
\end{proof}
This proposition allows us to describe the weight spaces of objects in $\widehat{\OX}$ in terms of the weight spaces of analytic Verma modules, as well as showing uniqueness of the weight space decomposition:
\begin{coro}\label{coro analytic properties of OX}
Let $M\in \widehat{\OX}$. The following hold:
\begin{enumerate}[label=(\roman*)]
    \item There is a finite family $\{ W_i\}_{i=1}^n\subset \operatorname{Irr}(H)$ satisfying that:
    \begin{equation*}
        \Lambda(M)\subset \bigcup_{i=1}^nc(W_i)+\mathbb{Z}^{\geq 0}.
    \end{equation*}
    Furthermore, for every $\lambda \in \Lambda(M)$, the space $M^{\lambda}$ is finite-dimensional.
    \item Every module appearing in a $\widehat{\Delta}$-filtration of $M$ is an object in $\widehat{\OX}$.
\end{enumerate}
\end{coro}
Now that the analytic structure of the objects of  $\widehat{\OX}$ has been clarified, it is time to start studying the algebraic structure of $\widehat{\OX}$. In order to do this, we will upgrade the   $\widehat{\Delta}$-filtrations obtained above to finite Jordan-Hölder series:
\begin{prop}\label{prop JH series in analytic category O}
Every object in $\widehat{\OX}$ admits a finite Jordan-Hölder series. That is, given $M\in \widehat{\OX}$ there is a finite filtration by objects in $\widehat{\OX}$:
\begin{equation*}
    0=M_0\subset M_1\subset \cdots\subset M_n=M,
\end{equation*}
satisfying that $M_{i}/M_{i-1}=\widehat{L}(W_i)$ for some $W_i\in \operatorname{Irr}(H)$. 
\end{prop}
\begin{proof}
By Proposition  \ref{prop filtrations by quot of verma}, it suffices to show that every analytic Verma module $\widehat{\Delta}(W)$  admits a Jordan-Hölder series. As $\mathscr{R}$ is noetherian, it suffices to show that every quotient of an analytic Verma module admits a simple submodule. Let $M$ be a quotient of $\widehat{\Delta}(W)$. If $M$ does not admit a simple submodule, then there is a descending sequence $\{M_n\}_{n\geq 0}$
satisfying that $M_{n+1}\subsetneq M_n$. In particular, by part $(v)$ in Theorem \ref{teo weight space decomposition ss case} we get a decreasing family $\{M_n^{\operatorname{ws}}\}_{n\geq 0}$ of subobjects of $M^{\operatorname{ws}}$. By claim $(iii)$ in  \ref{teo weight space decomposition ss case}, the fact that $\widehat{\Delta}(W)\rightarrow M$ is a surjection implies that $\Delta(W)\rightarrow M^{\operatorname{ws}}$ is also a surjection. Hence, we have $M\in \OX$. As $\OX$ is an artinian category, we have $M_n^{\operatorname{ws}}=M_{n+1}^{\operatorname{ws}}$ for sufficiently high $n$.
Thus, a second application of Theorem \ref{teo weight space decomposition ss case} shows that  $M_n=M_{n+1}$ for sufficiently high $n$, contradicting the fact that $M$ does not have a simple submodule.
\end{proof}
This proposition allows us to deduce a plethora of conclusions:
\begin{coro}\label{coro analytic OX is abelian}
The category $\widehat{\OX}$ is a full abelian subcategory of $\Mod(\mathscr{R})$. It is also closed under $\mathscr{R}$-submodules and finite direct sums.
\end{coro}
\begin{proof}
We saw in Lemma \ref{Lemma start of the filtration by quotients of Verma} that $\widehat{\OX}$ is closed under taking quotients by $\mathscr{R}$-submodules, and it is clear that $\widehat{\OX}$ is closed under finite direct sums. Hence, we only need to prove that for every $M\in \widehat{\OX}$, and every $\mathscr{R}$-submodule $N\subset M$, we have $N\in \widehat{\OX}$. By the previous proposition $M$ has a Jordan-Holder series:
\begin{equation*}
   0\subset M_1\subset \cdots\subset M_n=M, 
\end{equation*}
which satisfies that $M_i\in \widehat{\OX}$ for every $1\leq i \leq n$. We proceed by induction on the length of the composition series for $M$, the base case being that $M$ is a simple $\mathscr{R}$-module. There are two possible cases: If $M/N$ is a simple $\mathscr{R}$-module, then by uniqueness of the factors in a composition series there is some $W\in \operatorname{Irr}(H)$ such that $M/N\cong \widehat{L}(W)$. In this situation, as $M$ admits a weight space decomposition of compact type, and  $\widehat{L}(W)$ admits a semisimple weight space decomposition of finite type. Hence, Lemma \ref{Lemma ws decomposition and filtrations} implies that $N$ also admits a weight space decomposition. As $\mathscr{A}$ is noetherian, $N$ is also a finite $\mathscr{A}$-module, so we have $N\in \widehat{\OX}$. If $M/N$ is not simple, then there is an $\mathscr{R}$-submodule $T\subset M$ satisfying that $M/T$ is simple, $N\subset T$, and $T$ has a composition series of length $n-1$. The arguments above show that $T\in \widehat{\OX}$, so we may apply the induction hypothesis again to show $N\in \widehat{\OX}$.
\end{proof}
\begin{coro}\label{coro ws is well-defined}
There is an exact functor of abelian categories:
\begin{equation*}
    (-)^{\operatorname{ws}}:\widehat{\OX}\rightarrow \OX.
\end{equation*}
\end{coro}
\begin{proof}
We need to show that for any $M\in \widehat{\OX}$, we have  $M^{\operatorname{ws}}\in\OX$. We proceed by induction on the length of a composition series of $M$, the base case is Corollary \ref{coro ws of simple modules}. Let $N\subset M$ such that $M/N$ is simple. We have a left exact sequence:
\begin{equation}\label{equation definition of ws functor}
    0\rightarrow N^{\operatorname{ws}}\rightarrow M^{\operatorname{ws}}\rightarrow (M/N)^{\operatorname{ws}}.
\end{equation}
By Proposition \ref{prop JH series in analytic category O}, there is  $W\in\operatorname{Irr}(H)$ such that $M/N=\widehat{L}(W)$. In particular, $(M/N)^{\operatorname{ws}}=L(W)$ is a simple $R$-module. Hence, either $(\ref{equation definition of ws functor})$ is a short exact sequence, or $N^{\operatorname{ws}}=M^{\operatorname{ws}}$. If the sequence is short exact, then the induction hypothesis, together with the fact that $\OX$ is a Serre subcategory of $\Mod(R)$, imply that $M^{\operatorname{ws}}\in \OX$. If its not, then  $N^{\operatorname{ws}}=M^{\operatorname{ws}}$. As $M^{\operatorname{ws}}$ is dense in $M$, this implies $N=M$.
Next, we need to show $(-)^{\operatorname{ws}}$ is exact. As it is clearly left exact, we need to show that given an epimorphism $M\rightarrow N$ in $\widehat{\OX}$, the corresponding map $M^{\operatorname{ws}}\rightarrow N^{\operatorname{ws}}$ is also an epimorphism. Again, we may proceed by induction of the length of a Jordan-Hölder series of $N$. The base case was already shown above. If $N$ is not simple, there is some $\Tilde{N}\subset N$ such that $N/\Tilde{N}=\widehat{L}(W)$ for some $W\in \operatorname{Irr}(H)$. In this case, we have the following commutative diagram with exact rows:
\begin{equation*}
\begin{tikzcd}
0 \arrow[r] & \Tilde{M} \arrow[r] \arrow[d] & M \arrow[r] \arrow[d] & \widehat{L}(W) \arrow[d, "="] \arrow[r] & 0 \\
0 \arrow[r] & \Tilde{N} \arrow[r]           & N \arrow[r]           & \widehat{L}(W) \arrow[r]                & 0
\end{tikzcd}
\end{equation*}
as the middle map is a surjection and the rightmost map is an isomorphism, $\Tilde{M}\rightarrow \Tilde{N}$ is a surjection. Thus, by the induction hypothesis $\Tilde{M}^{\operatorname{ws}}\rightarrow \Tilde{N}^{\operatorname{ws}}$ is also a surjection. Applying $(-)^{\operatorname{ws}}$ to the previous diagram  implies that $M^{\operatorname{ws}}\rightarrow N^{\operatorname{ws}}$.
\end{proof}
\begin{coro}
Let $M\in\widehat{\OX}$, then $M^{\operatorname{Sing}}$ is a finite dimensional $K$-vector space.
\end{coro}
\begin{proof}
Assume that $M=\widehat{L}(W)$ for $W\in \operatorname{Irr}(H)$. By Proposition \ref{prop simple modules in OX}, for any $E\in \operatorname{Irr}(H)$ we have $\Hom_{\widehat{\OX}}(\widehat{\Delta}(E),\widehat{L}(W))=0$ if $W\neq E$. Hence, by Proposition \ref{prop maps from Verma} we have $\widehat{L}(W)^{\operatorname{Sing}}\subset \widehat{L}(W)^{c(W)}$, and this is finite-dimensional by $(i)$ in Corollary \ref{coro analytic properties of OX}. The general case follows from the existence of Jordan-Hölder series, together with the fact that $(-)^{\operatorname{Sing}}:\widehat{\OX}\rightarrow \operatorname{Vect}_K$  is a left exact functor.
\end{proof}

With all these tools at our disposal, we can start working towards the main theorem of the section. Namely, we want to show that the functor $(-)^{\operatorname{ws}}:\widehat{\OX}\rightarrow \OX$ constructed in Corollary \ref{coro ws is well-defined} is an equivalence of abelian categories. In particular, this will show that $\widehat{\OX}$ is a highest weight category. We will achieve this by showing that the extension of scalars along the map $R\rightarrow \mathscr{R}$ is an inverse to this functor.
\begin{Lemma}\label{Lemma properties generalized Verma}
The following hold:
\begin{enumerate}[label=(\roman*)]
    \item For every $W\in\operatorname{Irr}(H)$ and every $n\geq 0$ we have $\Delta(W)_n\in \OX$.
    \item $M\in \Mod(R)$ is an object in $\OX$ if and only if it admits a surjection:
\begin{equation*}
     \bigoplus_{i=1}^r\Delta(W_i)_{n_i}\rightarrow M.
\end{equation*}
\end{enumerate}
\end{Lemma}
\begin{proof}
This is a consequence of Proposition \ref{prop description of Oln} and the definition of $\OX$.
\end{proof}
As a consequence, every module $M\in \OX$ admits a two-term resolution:
\begin{equation}\label{equation resolution by generalized Verma}
   \bigoplus_{j=1}^s\Delta(W_j)_{m_j}\rightarrow \bigoplus_{i=1}^r\Delta(W_i)_{n_i}\rightarrow M.
\end{equation}
In practice, this means that we can use generalized Verma modules to study the behavior of right exact functors. As expected, there are analytic versions of these objects. In particular, Lemma \ref{Lemma short exact sequence quotients} allows us to make the following definition:
\begin{defi}
Let $W\in \operatorname{Irr}(H)$. The generalized analytic Verma module of level $n$ of $W$ is the following $\mathscr{R}$-module:
\begin{equation*}
    \widehat{\Delta}(W)_n= \mathscr{R}\widehat{\otimes}_{\mathscr{B}H}\left(\mathscr{B}H/\mathscr{B}H^{>n}\widehat{\otimes}_{H} W \right)= \mathscr{A}\widehat{\otimes}_K(B/B^{>n}\otimes_KW).
\end{equation*}    
\end{defi}
Where the last identity uses that $\mathscr{B}H/\mathscr{B}H^{>n}$ is a finite $K$-vector space.
\begin{prop}\label{prop basic properties generalized analytic Verma}
Let $W\in \operatorname{Irr}(H)$ and $n\geq 0$, $\widehat{\Delta}(W)_n$ satisfies the following:
\begin{enumerate}[label=(\roman*)]
    \item $\widehat{\Delta}(W)_n\in \widehat{\OX}$.
    \item $\widehat{\Delta}(W)^{\operatorname{ws}}_n=\Delta(W)_n$.
    \item There is a canonical isomorphism $\widehat{\Delta}(W)_n=\mathscr{R}\otimes_R\Delta(W)_n$.
\end{enumerate}
\end{prop}
\begin{proof}
As $B/B^{>n}\otimes_KW$ is a finite $K$-vector space, it is clear that $\widehat{\Delta}(W)_n$  is a finite $\mathscr{A}$-module. Furthermore, the explicit description of $\widehat{\Delta}(W)_n$ yields:
\begin{equation*}
    \widehat{\Delta}(W)^{\operatorname{ws}}_n=(\mathscr{A}\widehat{\otimes}_K(B/B^{>n}\otimes_KW))^{\operatorname{ws}}=A\otimes_K(B/B^{>n}\otimes_KW)=\Delta(W)_n,
\end{equation*}
and this shows $(i)$ and $(ii)$. Claim $(iii)$ follows as in the proof of Lemma \ref{Lemma analytification of Verma}.
\end{proof}
We are now finally ready to show the main theorem of the section:
\begin{Lemma}
 Extension of scalars along $R\rightarrow \mathscr{R}$ induces a right exact functor:
 \begin{equation*}
     \mathscr{R}\otimes_R-:\OX\rightarrow \widehat{\OX}.
 \end{equation*}
\end{Lemma}
\begin{proof}
Let $M\in \OX$. By Lemma \ref{Lemma properties generalized Verma} we have a surjection $\bigoplus_{i=1}^n\Delta(W_i)_{n_i}\rightarrow M$. By Proposition \ref{prop basic properties generalized analytic Verma} this induces a surjection $\bigoplus_{i=1}^n\widehat{\Delta}(W_i)_{n_i}\rightarrow \mathscr{R}\otimes_RM$. The category $\widehat{\OX}$ is closed under direct sums by Corollary \ref{coro analytic OX is abelian} and by $\mathscr{R}$-submodules by Lemma \ref{Lemma start of the filtration by quotients of Verma}. Thus, $M$ is an object of $\widehat{\OX}$, as we wanted to show.
\end{proof}
\begin{teo}\label{teo category O for p-adic banach algebras with a triangular decomposition}
There is an equivalence of abelian categories:
\begin{equation*}
 \mathscr{R}\otimes_{R}-:     \OX\leftrightarrows \widehat{\OX}:(-)^{\operatorname{ws}},
\end{equation*}
which makes $\widehat{\OX}$ a highest weight category.
\end{teo}
\begin{proof}
For every $M\in \OX$ we denote  $\widehat{M}:=\mathscr{R}\otimes_{R}M$ for simplicity. We want to show that $M=(\widehat{M})^{\operatorname{ws}}$. By Lemma \ref{Lemma properties generalized Verma} we have a right exact sequence:
\begin{equation*}
    \bigoplus_{j=1}^s\Delta(W_j)_{m_j}\rightarrow \bigoplus_{i=1}^r\Delta(W_i)_{n_i}\rightarrow M.
\end{equation*}
As $\widehat{(-)}$ and $(-)^{\operatorname{ws}}$ are right exact, we get a commutative diagram with exact rows:
\begin{equation*}
\begin{tikzcd}
\bigoplus_{j=1}^s\Delta(W_j)_{n_j} \arrow[r] \arrow[d] & \bigoplus_{i=1}^r\Delta(W_i)_{n_i} \arrow[r] \arrow[d] & M \arrow[r] \arrow[d]                       & 0 \\
\bigoplus_{j=1}^s\widehat{\Delta}(W_j)^{\operatorname{ws}}_{n_j} \arrow[r]           & \bigoplus_{i=1}^r\widehat{\Delta}(W_i)^{\operatorname{ws}}_{n_i} \arrow[r]           & (\widehat{M})^{\operatorname{ws}} \arrow[r] & 0
\end{tikzcd}
\end{equation*}
By Proposition \ref{prop basic properties generalized analytic Verma}, the first two vertical maps are isomorphisms. Thus $M\rightarrow (\widehat{M})^{\operatorname{ws}}$ is also an isomorphism. Next, choose $M\in \widehat{\OX}$. We need to show that $M=\widehat{(M^{\operatorname{ws}})}$. We proceed by induction on the length of a Jordan-Hölder series for $M$, the base case being Proposition \ref{prop simple modules in OX} and Corollary \ref{coro ws of simple modules}. Let $N\subset M$ such that $M/N$ is simple. We have a commutative diagram:
\begin{equation*}
\begin{tikzcd}
            & \mathscr{R}\otimes_R N^{\operatorname{ws}} \arrow[d] \arrow[r] & \mathscr{R}\otimes_R M^{\operatorname{ws}} \arrow[d] \arrow[r] & \mathscr{R}\otimes_R(M/N)^{\operatorname{ws}} \arrow[d] \arrow[r] & 0 \\
0 \arrow[r] & N \arrow[r]                                                    & M \arrow[r]                                                    & M/N \arrow[r]                                                     & 0
\end{tikzcd} 
\end{equation*}
By the induction hypothesis, the left and right vertical maps are isomorphisms, so the middle map is also an isomorphism. 
\end{proof}
\begin{obs}
 Notice that we did not require the map $R\rightarrow \mathscr{R}$ to be flat for the preceding theorem. However, in the case of rational Cherednik algebras, this map is always flat (cf. Theorem \ref{teo faithfully flat map to analytification}).   
\end{obs}
As a consequence of Theorem \ref{teo category O for p-adic banach algebras with a triangular decomposition} the structural elements of a highest weight category (\emph{ie}. standard objects, simple objects, projective covers, \emph{etc}) of $\widehat{\OX}$ are obtained from $\OX$ via extension of scalars. Notice that even if $\widehat{\OX}$ has enough projectives, these objects are not projective in $\Mod(\mathscr{R})$. In particular, the canonical map between the derived categories:
\begin{equation*}
    \operatorname{D}(\widehat{\OX})\rightarrow \operatorname{D}(\Mod(\mathscr{R})),
\end{equation*}
is not necessarily fully faithful. Furthermore, it is not even clear that for objects $M,N\in \widehat{\OX}$ we have:
\begin{equation}\label{equation analytic OX not a Serre subcat}
    \operatorname{Ext}^1_{\widehat{\OX}}(M,N)=\operatorname{Ext}^1_{\mathscr{R}}(M,N).
\end{equation}
This is in stark contrast with the algebraic situation, where the fact that $\OX$ is a Serre subcategory of $\Mod(R)$ implies that equation $(\ref{equation analytic OX not a Serre subcat})$ always holds. The crux of the matter is that given a $\partial$-equivariant short exact sequence of Banach spaces:
\begin{equation*}
    0\rightarrow V_1\rightarrow V_2\rightarrow V_3\rightarrow 0,
\end{equation*}
such that $V_1$ and $V_3$ admit weight space decompositions, it is in general not possible to assure that $V_2$ will also admit a weight space decomposition.\\
We finish this section by giving a couple of alternative definitions of $\widehat{\OX}:$
\begin{prop}\label{prop equivalent definitions of O}
Let $M\in \Mod(\mathscr{R})$. The following are equivalent:
\begin{enumerate}[label=(\roman*)]
    \item $M\in \widehat{\OX}$.
    \item There is an epimorphism $\bigoplus_{i=1}^r\widehat{\Delta}(W_i)_{n_i}\rightarrow M$.
    \item There is $\{W_i \}_{i=1}^n\subset \operatorname{Irr}(H)$ such that: $\Lambda(M)\subset \bigcup_{i=1}^n\lambda_i+\mathbb{Z}^{\geq 0}$, each $M^{\lambda}$ is finite dimensional, and  $M$ admits a weight space decomposition.
    \item  The action of $B$ on $M^{\operatorname{ws}}$ is locally nilpotent, and $M^{\operatorname{ws}}$ is dense in $M$.
\end{enumerate}
\end{prop}
\begin{proof}
It follows from Theorem \ref{teo category O for p-adic banach algebras with a triangular decomposition} that $(i)$ is equivalent to $(ii)$, and that it implies $(iii)$ and $(iv)$. Furthermore, $(iii)$ implies $(iv)$ by choosing $I=\Lambda(M)$, and noticing that $M^{\lambda-n}=0$ for big enough $n\in\mathbb{Z}^{\geq 0}$. Hence, assume that $M$ satisfies $(iv)$, we need to see  $M\in \widehat{\OX}$. By construction, $M^{\operatorname{ws}}$ is a $R$-module such that the action of $B$  on $M^{\operatorname{ws}}$ is locally nilpotent. By Proposition \ref{prop equiv conditions category O} we have a surjection:
\begin{equation*}
    \bigoplus_{i\in I}\Delta(W_i)_{n_i}\rightarrow M^{\operatorname{ws}}.
\end{equation*}
Hence, we get a $\mathscr{R}$-linear map $\bigoplus_{i\in I}\widehat{\Delta}(W_i)_{n_i}\rightarrow M$. As the image of this map contains $M^{\operatorname{ws}}$, it is dense. As $M$ is a finite $\mathscr{R}$-module, it follows that the map is surjective. Furthermore, by finiteness, there is a finite set $J\subset I$ such that the map 
$\bigoplus_{i\in J}\widehat{\Delta}(W_i)_{n_i}\rightarrow M$ is still surjective. Thus, $M$ satisfies $(ii)$.
\end{proof}
\begin{obs}
Unlike in the algebraic situation, the condition that $M$ is a finitely generated $\mathscr{R}$-module is necessary. In particular, there are Banach $\mathscr{R}$-modules which admit a finite type decomposition into weight spaces satisfying $(iii)$ which are not finitely generated $\mathscr{R}$-modules.
\end{obs}
\subsection{Fréchet-Stein algebras with a triangular decomposition}\label{category O for Fréchet-Stein algebras}
We will now extend the previous theory to the context of Fréchet-Stein algebras. 
\begin{defi}
Let $\mathscr{R}=\varprojlim_n\mathscr{R}(n)$ be a Fréchet-Stein algebra. A triangular decomposition of $\mathscr{R}$ consists of a tuple $(R, A,B,\partial)$ satisfying the following:
\begin{enumerate}[label=(\roman*)]
    \item $R\subset \mathscr{R}$ is a dense graded subalgebra, and $\partial \in R$.
    \item The map $R\rightarrow \mathscr{R}(n)$ is injective for each $n\geq 0$.
    \item The tuple $(R,A,B,\partial)$ is a triangular decomposition of $\mathscr{R}(n)$ for each $n\geq 0$.
\end{enumerate}
\end{defi}
Morally, a triangular decomposition of a Fréchet-Stein algebra is a family of decompositions which behave well with respect to the transition maps. For the rest of this section, we fix a  Fréchet-Stein algebra $\mathscr{R}=\varprojlim_n\mathscr{R}(n)$, with a triangular decomposition 
$(R,A,B,\partial)$. As before, for each $n\geq 0$ we let $\mathscr{A}(n)$, and $\mathscr{B}(n)$ be the closures in $\mathscr{R}$ of $A$ and $B$ respectively. Let us begin the section by showing some properties of triangular decompositions:
\begin{prop}\label{prop decomposition as tensor products of Fréchet spaces}
The transition maps in  $\mathscr{R}=\varprojlim_n\mathscr{R}_n$ induce inverse limits:
\begin{equation*}
    \mathscr{A}=\varprojlim_n\mathscr{A}(n), \textnormal{ }\mathscr{B}=\varprojlim_n \mathscr{B}(n),
\end{equation*}
and there is an isomorphism of Fréchet spaces $\mathscr{R}=\mathscr{A}\widehat{\otimes}_KH\widehat{\otimes}_K\mathscr{B}$.
\end{prop}
\begin{proof}
First, notice that the fact that $R\subset \mathscr{R}(n)$ for each $n\geq 0$ implies that the transition maps $\mathscr{R}(n+1)\rightarrow \mathscr{R}(n)$ restrict to an isomorphism on $R$. Thus, they also restrict to isomorphisms on $A$ and  $B$. Hence, we get our desired inverse systems
$\mathscr{A}=\varprojlim_n\mathscr{A}(n),$ and $\mathscr{B}=\varprojlim_n \mathscr{B}(n)$. Thus, we have:
\begin{equation*}
    \mathscr{A}\widehat{\otimes}_KH\widehat{\otimes}_K\mathscr{B}=\varprojlim_{n,m}\mathscr{A}(n)\widehat{\otimes}_KH\widehat{\otimes}_K \mathscr{B}(m)=\varprojlim_{n}\mathscr{A}(n)\widehat{\otimes}_KH\widehat{\otimes}_K \mathscr{B}(n)=\varprojlim_n\mathscr{R}(n),
\end{equation*}
where the last identity follows by definition of a triangular decomposition.
\end{proof}
\begin{coro}
The following hold:
\begin{enumerate}[label=(\roman*)]
    \item $\mathscr{A}$ and $\mathscr{B}$ are Fréchet $K$-algebras.
    \item $\mathscr{A}\cdot H=H\cdot \mathscr{A}$, and $\mathscr{B}\cdot H=H\cdot \mathscr{B}$.
\end{enumerate}
We let $\mathscr{A}H$ and $\mathscr{B}H$ be the algebras $\mathscr{A}\widehat{\otimes}_KH$ and $\mathscr{B}\widehat{\otimes}_KH$.
\end{coro}
\begin{proof}
In view of the previous proposition, it follows from the corresponding fact in the Banach setting. Namely, Proposition \ref{prop definition AH BH}.
\end{proof}
Hence, the fact that the decomposition $\mathscr{R}(n)=\mathscr{A}(n)\widehat{\otimes}_KH\widehat{\otimes}_K\mathscr{B}(n)$ holds  at the level of Banach spaces for each $n\geq 0$ implies that this decomposition holds at the level of Fréchet spaces. This will be a recurring theme in the theory of triangular decompositions of Fréchet-Stein algebras. Before showing the following proposition, we mention that there is an analog of Theorem \ref{teo weight space decomposition ss case} for Fréchet spaces. Namely, this is contained in  \cite[Section 2]{Schmidt2010VermaMO}. Hence, we will use the results of Theorem \ref{teo weight space decomposition ss case} in this generality without further mention.
\begin{prop}\label{prop ws decompositions for graded subalgebras Frechet case}
The following hold:
\begin{enumerate}[label=(\roman*)]
    \item $\mathscr{A}$ has a semi-simple weight space decomposition for $\operatorname{ad}(\partial)$, and $\mathscr{A}^{\operatorname{ws}}=A$.
    \item $\mathscr{B}$ has a semi-simple weight space decomposition for $\operatorname{ad}(\partial)$, and $\mathscr{B}^{\operatorname{ws}}=B$.
    \item If each $\mathscr{R}(n)$ has a a semi-simple weight space decomposition for $\operatorname{ad}(\partial)$, then so does $\mathscr{R}$.
\end{enumerate}    
\end{prop}
\begin{proof}
We will only show $(i)$, the rest are analogous. For each $n\geq 0$ the space $\mathscr{A}(n)$ admits a semi-simple weight space decomposition of finite type, and $\mathscr{A}(n)^{\operatorname{ws}}=A$. In particular, each $v\in \mathscr{A}(m)$ admits a unique expression:
\begin{equation*}
    v=\sum_{n\geq 0}v_n, \textnormal{ where }v_n\in A_n.
\end{equation*}
Hence, we may regard $\mathscr{A}=\varprojlim\mathscr{A}(m)$ as the space of all infinite sums $\sum_{n\geq 0}v_n$, satisfying $v_n\in A_n$ for all $n\geq 0$, and such that the sum $\sum_{n\geq 0}v_n$ converges in $\mathscr{A}(m)$ for all $m\geq 0$. In particular, every $v\in \mathscr{A}$ admits an expression as an infinite sum of weight vectors, as we wanted to show.
\end{proof}
We will now study the representation theory of $\mathscr{R}$. As in the Banach case, we need to impose some finiteness conditions in order to obtain a good theory.
\begin{obs}\label{remark conditions for Frechet category O}
For the rest of the section, we will assume that the inverse limits:
\begin{equation*}
    \mathscr{A}=\varprojlim_{n\geq 0}\mathscr{A}(n), \textnormal{ } \mathscr{B}=\varprojlim_{n\geq 0}\mathscr{B}(n)
\end{equation*}
are Fréchet-Stein presentations for $\mathscr{A}$ and $\mathscr{B}$ respectively.
\end{obs}
By Theorem \ref{teo category O for p-adic banach algebras with a triangular decomposition}, there are highest weight categories $\widehat{\OX}(n)\subset \Mod(\mathscr{R}(n))$, together with pairs of mutually inverse equivalences of categories:
\begin{equation}\label{equation equiv of categories Ocm}
    \mathscr{R}(n)\otimes_{R}-:\OX\leftrightarrows \widehat{\OX}(n):(-)^{\operatorname{ws}}.
\end{equation}
Let us now study the behavior of the $\widehat{\OX}(n)$ under the transition maps:
\begin{Lemma}\label{Lemma transition maps category O}
For $n\geq 0$ let $\rho(n):\mathscr{R}(n+1)\rightarrow \mathscr{R}(n)$. The extension of scalars:
\begin{equation*}
    \rho(n)^*:=\mathscr{R}(n)\otimes_{\mathscr{R}(n+1)}-:\widehat{\OX}(n+1)\rightarrow \widehat{\OX}(n),
\end{equation*}
is an equivalence of categories. Furthermore, we have: $(-)^{\operatorname{ws}}=(-)^{\operatorname{ws}}\circ\rho(n)^*$.
\end{Lemma}
\begin{proof}
For each $n\geq 0$, consider the map $i(n):R\rightarrow \mathscr{R}(n)$. By definition of a triangular decomposition we have $i(n)=\rho(n)\circ i(n+1)$. In particular, $\rho(n)^*$ restricts to a functor $\rho(n)^*:\widehat{\OX}(n+1)\rightarrow \widehat{\OX}(n)$. As $i(n)^*$ and $i(n+1)^*$ are equivalences, it follows that $\rho(n)^*$ is an equivalence. Furthermore, as $i(n)^*$ and $i(n+1)^*$ have inverse $(-)^{\operatorname{ws}}$, it follows that we have $(-)^{\operatorname{ws}}=(-)^{\operatorname{ws}}\circ\rho(n)^*$, as wanted.
\end{proof}
We now define the following category of modules:
\begin{defi}
We define the category $\wideparen{\OX}$ of $\mathscr{R}$ as the following inverse limit:
\begin{equation*}
    \wideparen{\OX}=\varprojlim \widehat{\OX}(n).
\end{equation*}
\end{defi}
For $n,m\geq 0$  and $W\in \operatorname{Irr}(H)$, we let:
\begin{equation*}
    \widehat{\Delta}(W)_m^n:= \mathscr{A}(n)\otimes_K\left(B/B^{>m}\otimes_KH \right),
\end{equation*}
be the analytic generalized Verma module of level $m$ of $\mathscr{R}(n)$. Then we have:
\begin{equation*}
    \wideparen{\Delta}(W)_m:= \varprojlim \widehat{\Delta}(W)_m^n= \mathscr{A}\otimes_K\left(B/B^{m>0}\otimes_K W\right),
\end{equation*}
where the last identity follows because $B/B^{m>0}\otimes_K W$ is finite-dimensional. 
\begin{defi}
$\wideparen{\Delta}(W)_m$ is the co-admissible Verma module of level $m$ of $W$.
\end{defi}
Our thorough study of the $\widehat{\OX}(n)$ allows us to show multiple properties of $\wideparen{\OX}$:
\begin{prop}\label{prop properties of the category O Frechet}
The category $\wideparen{\OX}$ satisfies the following properties:
\begin{enumerate}[label=(\roman*)]
    \item $\wideparen{\OX}$ is a full abelian subcategory of $\mathcal{C}(\mathscr{R})$.
    \item Let $M\in\wideparen{\OX}$. Then $M$ is a Fréchet space which admits a finite-type weight space decomposition with respect to the action of the grading element $\partial \in \mathscr{R}$. 
    \item There is a pair of mutually inverse equivalences of abelian categories:
    \begin{equation*}
   \wideparen{(-)}:=\mathscr{R}\otimes_{R}-:\OX\leftrightarrows \wideparen{\OX}:(-)^{\operatorname{ws}}.
    \end{equation*}
    \item Every $M\in \wideparen{\OX}$ is a co-admissible $\mathscr{A}$-module.
\end{enumerate}
\end{prop}
\begin{proof}
By definition we have $\mathcal{C}(\mathscr{R})=\varprojlim \Mod(\mathscr{R}(n))$, where the transition maps are given by extension of scalars along the maps $\rho(n):\mathscr{R}(n+1)\rightarrow \mathscr{R}(n)$. As $\widehat{\OX}(n)$ is a full abelian subcategory of $\Mod(\mathscr{R}(n))$ for each $n\geq 0$, it follows that $\wideparen{\OX}$ is a full abelian subcategory of $\mathcal{C}(\mathscr{R})$. Thus, statement $(i)$  holds. Choose $M=\varprojlim M(n)\in\wideparen{\OX}$. By Lemma \ref{Lemma transition maps category O}, we have the following identities:
\begin{equation*}
    (M(n))^{\operatorname{ws}}=(M(n+1))^{\operatorname{ws}} \textnormal{ for any } n\geq 0.
\end{equation*}
Hence, we can define $M^{\operatorname{ws}}:=(M(n))^{\operatorname{ws}}$ for any $n\geq 0$. Furthermore, by Theorem \ref{teo category O for p-adic banach algebras with a triangular decomposition} we have $M^{\operatorname{ws}}\in \OX$, and $M(n)$ admits a weight space decomposition of finite type for each $n\geq 0$. Thus, it follows just as in the proof of Proposition \ref{prop ws decompositions for graded subalgebras Frechet case} that $M$ admits a finite-type weight space decomposition with respect to $\partial$, with weight spaces given by $M^{\operatorname{ws}}$. Thus, statement $(ii)$ holds, and the weight space functor $(-)^{\operatorname{ws}}:\wideparen{\OX}\rightarrow \OX$ is well-defined. Let $M\in \OX$, then $M$ is a finitely presented $R$-module, so   $\mathscr{R}\otimes_{R}M$ is a finitely presented $\mathscr{R}$-module. In particular, it is a co-admissible module. Furthermore, we have:
\begin{equation*}
    \wideparen{M}=\mathscr{R}\otimes_{R}M=\varprojlim \mathscr{R}(n)\otimes_{R}M,
\end{equation*}
which shows that $\mathscr{R}\otimes_{R}M\in \wideparen{\OX}$. Let us now show that this functor is an equivalence, with quasi-inverse $(-)^{\operatorname{ws}}$. Choose some $M\in\OX$, then by the proof of $(ii)$, we have:
\begin{equation}\label{equation 1 properties of O}
    \wideparen{M}^{\operatorname{ws}}=(\mathscr{R}(n)\otimes_{R}M)^{\operatorname{ws}}=M, \textnormal{ for any } n\geq 0,
\end{equation}
On the other hand, by Lemma \ref{Lemma transition maps category O}, every $\wideparen{M}\in\wideparen{\OX}$ satisfies:
\begin{equation*}
    \wideparen{M}=\varprojlim \mathscr{R}(n)\otimes_{R}M=\mathscr{R}\otimes_{R}M,
\end{equation*}
for a unique $M\in\OX$. Thus, $\wideparen{(-)}$ is essentially surjective, and hence an equivalence. The arguments above show that $(-)^{\operatorname{ws}}$ is a quasi-inverse.\\
 For the last statement, choose $m\geq 0$ and $W\in\operatorname{Irr}(G)$. Then we have:
\begin{equation*}
    \wideparen{\Delta}(W)_m:=\mathscr{A}\otimes_K(B/B^{m>0}\otimes_K W).
\end{equation*}
Thus, every  $\wideparen{\Delta}_m(W)$ is a finite-free $\mathscr{A}$-module, and therefore is co-admissible. By Lemma \ref{Lemma properties generalized Verma} for any $M\in\OX$ there is a right exact sequence:
\begin{equation*}
    \bigoplus_{j=1}^r\Delta(W_j)_{m_j}\rightarrow \bigoplus_{i=1}^s\Delta(W_i)_{n_i}\rightarrow M\rightarrow 0.
\end{equation*}
Applying $\mathscr{R}\otimes_{R}-$ to this sequence yields the right exact sequence:
\begin{equation*}
    \bigoplus_{j=1}^r\wideparen{\Delta}(W_j)_{m_j}\rightarrow \bigoplus_{i=1}^s\wideparen{\Delta}(W_i)_{n_i}\rightarrow \mathscr{R}\otimes_{R}M\rightarrow 0.
\end{equation*}
As the first two terms are co-admissible $\mathscr{A}$-modules, it follows by \cite[Corollary 3.4]{schneider2002algebras} that the $\mathscr{A}$-module $\mathscr{R}\otimes_{R}M$ is also co-admissible, as wanted. 
\end{proof}
Hence, $\wideparen{\OX}$ is a highest weight category, with standard objects given by the co-admissible Verma modules:
\begin{equation*}
    \wideparen{\Delta}(W):=\mathscr{R}\otimes_{R}\Delta(W)=\mathscr{A}\otimes_KW.
\end{equation*}
 The simple objects are given by the following quotients of Verma modules:
\begin{equation*}
    \wideparen{L}(W):=\mathscr{R}\otimes_{R}L(W).
\end{equation*}
Let us now analyze the properties of co-admissible Verma modules in more detail:
\begin{Lemma}\label{Lemma closed under subobjects}
The category $\wideparen{\OX}$ is closed under taking closed $\mathscr{R}$-submodules.
\end{Lemma}
\begin{proof}
Let $M\in \wideparen{\OX}$ and let $N\subset M$ be closed. Then $N$ is co-admissible, so we have:
\begin{equation*}
N=\varprojlim \mathscr{R}(n)\otimes_{\mathscr{R}}N, \textnormal{  }    \mathscr{R}(n)\otimes_{\mathscr{R}}N\subset \mathscr{R}(n)\otimes_{\mathscr{R}}M,
\end{equation*}
and the result follows by Corollary \ref{coro analytic OX is abelian}.
\end{proof}
\begin{prop}
Let $W\in\operatorname{Irr}(G)$, then $\wideparen{\Delta}(W)$ has a unique maximal closed submodule, with quotient $\wideparen{L}(W)$. In particular, $\wideparen{L}(W)$ is a simple object in $\mathcal{C}(\mathscr{R})$.
\end{prop}
\begin{proof}
Let $N\subset \wideparen{\Delta}(W)$ be a closed submodule. By the previous lemma, it is an object in $\wideparen{\OX}$. Hence, it has a weight space decomposition with respect to $\partial$, and we have $N^{\operatorname{ws}}\subset \Delta(W)$. The Verma module $\Delta(W)$ has a maximal proper submodule $M(W)$. Hence, we have $N^{\operatorname{ws}}\subset M(W)$, which implies that $N\subset \wideparen{M}(W)$.The equivalence of categories of Proposition \ref{prop properties of the category O Frechet} shows that $\wideparen{\Delta}(W)/\wideparen{M}(W)=\wideparen{L}(W)$.
\end{proof}
\begin{obs}
 Notice, however, that unlike in the algebraic or Banach situations, it is not clear that $\wideparen{L}(W)$ is a simple module in $\Mod(\mathscr{R})$. Indeed, as $\mathscr{R}$ is not noetherian, it could happen that $\wideparen{L}(W)$ admits non-closed submodules. Nevertheless, as the closure of a submodule is a submodule, the above proposition shows that any non-zero $\mathscr{R}$-submodule of $\wideparen{L}(W)$ is dense in $\wideparen{L}(W)$.   
\end{obs}
 Our next goal is giving  different characterizations of the elements of $\wideparen{\OX}$, similar to what we did in Proposition \ref{prop equivalent definitions of O}. In particular, we aim to characterize objects in $\wideparen{\OX}$ in terms of their decompositions into weight spaces.
 \begin{defi}
Let $\wideparen{\OX}^{\operatorname{ws}}\subset \mathcal{C}(\mathscr{R})$ be the full subcategory of modules which admit a weight space decomposition of finite type. 
\end{defi}

\begin{teo}
Let $M\in \mathcal{C}(\mathscr{R})$, the following are equivalent:
\begin{enumerate}[label=(\roman*)]
    \item $M\in \wideparen{\OX}$.
    \item $M\in \wideparen{\OX}^{\operatorname{ws}}$, and there $\{\lambda_i\}_{i=1}^n\subset \Lambda(M)$  satisfying that:
    \begin{equation*}
        \Lambda(M)\subset \bigcup_{i=1}^n\lambda_i+\mathbb{Z}^{\geq 0}.
    \end{equation*}
    \item $M\in \wideparen{\OX}^{\operatorname{ws}}$, and the action of $B$ on $M^{\operatorname{ws}}$ is locally nilpotent.
\end{enumerate}
Furthermore, if the equivalent conditions above hold, then $M\in\mathcal{C}(\mathscr{A})$.
\end{teo}
\begin{proof}
The fact that $(i)$ implies $(ii)$ was shown in Proposition \ref{prop properties of the category O Frechet}.  The fact that $(ii)$ implies $(iii)$ is clear. Hence, we only need to show that $(iii)$ implies $(i)$. Let $M\in \mathcal{C}(\mathscr{R})$ and assume $M$ satisfies $(iii)$. In this situation $M^{\operatorname{ws}}$ is an $R$-module which admits a weight space decomposition of finite type (in the algebraic sense) and satisfies that the action of $B$ on $M^{\operatorname{ws}}$ is locally nilpotent. Thus, by Proposition \ref{prop equiv conditions category O} we have  $M^{\operatorname{ws}}\in \OX$. Thus, we obtain a map $\varphi:\wideparen{M}\rightarrow M$, which has dense image. As both $\wideparen{M}$ and $M$ are co-admissible $\mathscr{R}$-modules, the image of $\varphi$ is closed. Hence, $\varphi$ is a surjection. But then it follows by Lemma \ref{Lemma closed under subobjects} that $M\in \wideparen{\OX}$.
\end{proof}
Thus, we could have chosen any of the conditions above to define $\wideparen{\OX}$. We remark that $(ii)$ in the previous proposition is analogous to \cite[Definition 3.6.2.]{Schmidt2010VermaMO}, where the case of the category $\OX$ for the Arens-Michael envelope of the universal algebra of a reductive Lie algebra is treated. This is also similar to \cite{Duprequantum}, which treats the theory for quantum Lie algebras.\bigskip

Ideally, we would like to describe $\wideparen{\OX}$ as the category of co-admissible $\mathscr{R}$-modules which admit a weight space decomposition and are co-admissible as $\mathscr{A}$-modules. Unfortunately, we have not been able to prove such a statement (and we doubt its veracity). We can, however, show a slightly weaker version of this statement:
\begin{prop}
Assume $M\in \mathcal{C}(\mathscr{R})$ admits a weight space decomposition, and a Fréchet-Stein presentation:
\begin{equation*}
    M=\varprojlim_n M(n),
\end{equation*}
satisfying that each $M(n)$ is also a finite $\mathscr{A}(n)$-module. Then $M$ is an object in $\wideparen{\OX}$.
\end{prop}
\begin{proof}
It suffices to show that $M(n)\in \widehat{\OX}(n)$ for each $n\geq 0$. By assumption, $M(n)$ is a finite $\mathscr{R}(n)$-module and a finite $\mathscr{A}(n)$-module. As $M$ admits a weight space decomposition, it follows that $M^{\operatorname{ws}}$ is dense in $M$. Thus, its image in $M(n)$ is also dense. As $M(n)$ is a finite $\mathscr{A}(n)$-module,
we can use the arguments in the proof of statement $(i)$ in Proposition \ref{prop generating system by eigenvectors} to show that there is a finite family of generalized eigenvectors in $M(n)$ which generate $M(n)$ as a  $\mathscr{A}(n)$-module. In this situation, the fact that $\mathscr{A}(n)$ admits a weight space decomposition for $\operatorname{ad}(\partial)$ allows us to use the proof of statement $(ii)$ in Proposition \ref{prop generating system by eigenvectors} to show that $M(n)$ admits a weight space decomposition, as wanted.
\end{proof}
\section{\texorpdfstring{$p$-adic Cherednik algebras and the GAGA functor}{}}\label{Section GAGA for cherednik algebras}
Now that we have introduced our main technical tools, our aim is to slowly drift into the world of representation theory. In particular, we aim at studying different representation-theoretical aspects of $p$-adic Cherednik algebras. For the benefit of the reader, we devote Section \ref{section intro padicCHer} to give a quick introduction to the theory of $p$-adic Cherednik algebras. In order to keep the background material contained, we will not give any proof, and rely on our companion paper \cite{p-adicCheralg} as our source for the basics on the theory of $p$-adic Cherednik algebras on rigid spaces. 
\subsection{\texorpdfstring{$p$-adic Cherednik algebras}{}}\label{section intro padicCHer}
We start this section by giving a quick summary of the contents of \cite{p-adicCheralg}. For the rest of the section, we fix a finite group $G$, and let $X=\Sp(A)$ be a smooth affinoid space with a $G$-action. Let us start by making the following definitions:
\begin{defi}
We define the following objects:
\begin{enumerate}[label=(\roman*)]
    \item A reflection hypersurface in $X$ is a pair $(Y,g)$ given by an element $g\in G$ and a connected  hypersurface $Y\subset X^g$. We let $S(X,G)$ be the set of all reflection hypersurfaces. Notice that $G$ has an action on $S(X,G)$ given by:
    \begin{equation*}
        h(Y,g)=(h(Y),hgh^{-1}, \textnormal{ for }(Y,g)\in S(X,G), \textnormal{ and }h\in G.
    \end{equation*}
    \item A reflection function is a $G$-equivariant map:
    \begin{equation*}
        c:S(X,G)\rightarrow K.
    \end{equation*}
    We let $\operatorname{Ref}(X,G)$ be the $K$-vector space of all reflection functions.
    \item Let $X^{\operatorname{reg}}$ be the locus of points in $X$ which have trivial stabilizer with respect to the action of $G$. The space $X^{\operatorname{reg}}$ is a $G$-invariant admissible open of $X$.
\end{enumerate}
\end{defi}
The $p$-adic Cherednik algebras of the action of $G$ on $X$ are a family of Fréchet-Stein algebras parameterized by a unit $t\in K^*$, a reflection function $c\in \operatorname{Ref}(X,G)$, and a $G$-invariant cohomology class $[\omega]\in \operatorname{H}_{\operatorname{dR}}^2(X)^G$.  For the rest of the section, fix $t,c,$ and $\omega$ as above. Additionally, we will assume that $X$ admits an étale map $X\rightarrow \mathbb{A}^r_K$, and that each reflection hypersurface $(Y,g)\in S(X,G)$ satisfies that $Y$ is the zero locus of a unique rigid function $f_{(Y,g)}\in \OX_X(X)$. We point out that, in virtue of \cite[Proposition 2.1.23]{p-adicCheralg}, $X$ admits a finite cover by $G$-invariant affinoid subdomains satisfying these conditions. In order to keep the introduction short, we will assume that the reader is familiar with the theory of sheaves of (infinite order) twisted differential operators, and of the actions of finite groups on them. All the necessary material on this topic can be found in \cite[Sections 2.3,2.4]{p-adicCheralg}.\bigskip

We define the Cherednik algebra $H_{t,c,\omega}(X,G)$ as the subalgebra of $G\ltimes \D_{\omega/t}(X^{\operatorname{reg}})$ generated by $G$, $\OX_X(X)$, and the (standard) Dunkl-Opdam operators: 
\begin{equation*}
    D_{v}=tv+\sum_{(Y,g)\in S(X,G)}\frac{2c(Y,g)}{1-\lambda_{Y,g}}\frac{v(f_{(Y,g)})}{f_{(Y,g)}}(g-1),
\end{equation*}
where $v\in \mathcal{T}_{X/K}(X)$  is a vector field on $X$, and the operators $\lambda_{Y,g}\in K\setminus \{1\}$ are the eigenvalues of $g$ on the conormal bundle of $Y$  (\emph{cf.}  \cite[Proposition 2.1.14]{p-adicCheralg}).
\begin{teo}\label{teo defi Cher alg}
Let $\mathcal{H}_{t,c,\omega}(X,G)$
be the closure of  $H_{t,c,\omega}(X,G)$ in $G\ltimes \wideparen{\D}_{\omega/t}(X^{\operatorname{reg}})$. We will call $\mathcal{H}_{t,c,\omega}(X,G)$
the $p$-adic Cherednik algebra of the action of  $G$ on $X$ (relative to the choice of $t,c,\omega)$. It satisfies the following properties:
\begin{enumerate}[label=(\roman*)]
    \item $\mathcal{H}_{t,c,\omega}(X,G)$ is a Fréchet-Stein algebra.
    \item For big enough $G$-invariant  $U\subset X^{\operatorname{reg}}$, $\mathcal{H}_{t,c,\omega}(X,G)$ agrees with the closure of $H_{t,c,\omega}(X,G)$ in $G\ltimes \wideparen{\D}_{\omega/t}(U)$.
\end{enumerate}
\end{teo}
\begin{proof}
This is shown in \cite[Theorem A]{p-adicCheralg}.
\end{proof}
The $p$-adic Cherednik algebras algebras can be arranged canonically into a sheaf on $X/G=\Sp(A^G)$. This is defined on  affinoid subdomains  $Y \subset X/G$ by the rule:
\begin{equation*}
    \mathcal{H}_{t,c,\omega,X,G}(Y,G):= \mathcal{H}_{t,c,\omega}(Y\times_{X/G}X,G).    
\end{equation*}
We can analogously define a sheaf of Cherednik algebras $H_{t,c,\omega,X,G}$ on $X/G$. The restriction maps of $\mathcal{H}_{t,c,\omega,X,G}$ are obtained by continuously extending those of $H_{t,c,\omega,X,G}$, and these maps are, in turn, obtained from restricting those of $G\ltimes \D_{\omega/t}$.
\begin{teo}
The sheaf of $p$-adic Cherednik algebras $\mathcal{H}_{t,c,\omega,X,G}$  is the unique sheaf on $X/G$ satisfying the following properties:
\begin{enumerate}[label=(\roman*)]
    \item Let $Y\rightarrow X/G$ be an affinoid subdomain. Then we have:
\begin{equation*}
    \mathcal{H}_{t,c,\omega,X,G}(Y,G):= \mathcal{H}_{t,c,\omega}(Y\times_{X/G}X,G).    
\end{equation*}
\item Let $Z\rightarrow Y\rightarrow X/G$ be affinoid open immersions. The restriction map:
\begin{equation*}
    \mathcal{H}_{t,c,\omega,X,G}(Y,G)\rightarrow \mathcal{H}_{t,c,\omega,X,G}(Z,G),
\end{equation*}
is a continuous morphism of Fréchet-Stein $K$-algebras.
\item There is a monomorphism of sheaves of $K$-algebras $H_{t,c,\omega,X,G}\rightarrow \mathcal{H}_{t,c,\omega,X,G}$.
\end{enumerate}
Furthermore, $\mathcal{H}_{t,c,\omega,X,G}$ has trivial higher sheaf cohomology.
\end{teo}
\begin{proof}
This is shown in \cite[Theorem B]{p-adicCheralg}.
\end{proof}
This theorem can be used to obtain sheaves of $p$-adic Cherednik algebras on more general settings. Indeed, if $X$ is a smooth rigid space with an action of $G$ satisfying some mild technical requirements (see Theorem \ref{teo existence of quotients of rigid spaces by finite groups}) then one can define a family of $p$-adic Cherednik algebras $\mathcal{H}_{t,c,\omega,X,G}$ on  $X/G$. Furthermore, this sheaf can even be extended to the small étale site  $X/G_{\textnormal{ét}}$. See \cite[Section 4.5]{p-adicCheralg} for details.\bigskip

Finally, let us comments on the Fréchet-Stein structure of the $\mathcal{H}_{t,c,\omega}(X,G)$. This construction is based upon the notion of a Cherednik lattice, which we now define. Let $\mathcal{A}_{\omega/t}(X)$ be the Atiyah algebra on $X$ associated to the algebra of twisted differential operators $\D_{\omega/t}(X)$ (\emph{cf}. \cite[Section 2.3]{p-adicCheralg}). By construction, this is a $(K,\OX_X(X))$-Lie algebra, and we have a short exact sequence of $\OX_X(X)$-modules:
\begin{equation*}
   0\rightarrow \OX_X(X)\rightarrow \mathcal{A}_{\omega/t}(X) \xrightarrow[]{\rho}\mathcal{T}_X(X)\rightarrow 0,
\end{equation*}
where $\rho:\mathcal{A}_{\omega/t}(X)\rightarrow \mathcal{T}_X(X)$ is the anchor map of $\mathcal{A}_{\omega/t}(X)$. By assumption, the tangent sheaf $\mathcal{T}_X(X)$ is a free $\OX_X(X)$-module, and $\omega$ is a $G$-invariant closed $2$-form on $X$.  Hence, by \cite[Lemma 2.4.5]{p-adicCheralg}, $\mathcal{A}_{\omega/t}(X)$
admits a $G$-invariant $(R,A^{\circ})$-Lie lattice which contains $A^{\circ}$. Let $\mathscr{A}_{\omega/t}$ be any such a lattice. In order to fix notation, we will now introduce some objects. First, we define the following $A^{\circ}$-module:
\begin{equation*}
    \mathscr{T}_{\mathscr{A}_{\omega/t}}:=\rho(\mathscr{A}_{\omega/t})\subset \mathcal{T}_{X/K}(X).
\end{equation*}
Next, let $\mathcal{I}(\mathscr{A}_{\omega/t})$ be the two-sided ideal in $U(\mathscr{A}_{\omega/t})$ generated by:
\begin{equation*}
    1_{A^{\circ}}-1_{\widehat{U}(\mathscr{A}_{\omega/t})}.
\end{equation*}
Then we define the following $\mathcal{R}$-algebra:
 \begin{equation*}
\D(\mathscr{A}_{\omega/t})=U(\mathscr{A}_{\omega/t})/\mathcal{I}(\mathscr{A}_{\omega/t}).
 \end{equation*}
We are now ready to define the Cherednik lattices of the action of $G$ on $X$:
\begin{defi}
A Cherednik lattice of $\mathcal{A}_{\omega/t}(X)$ is a  $G$-invariant finite-free $(\mathcal{R},A^{\circ})$-Lie lattice $\mathscr{A}_{\omega/t}\subset \mathcal{A}_{\omega/t}(X)$ such that there is a  $G$-invariant affinoid subdomain $\Sp(B)\subset X$ containing the Shilov boundary of $X$ and satisfying the following:
 \begin{enumerate}[label=(\roman*)]
       \item $B^{\circ}$ is $\mathscr{A}_{\omega/t}$-stable, and we let $\mathscr{B}_{\omega/t}=B^{\circ}\otimes_{A^{\circ}}\mathscr{A}_{\omega/t}$.
       \item  For $v\in \mathscr{T}_{\mathscr{A}_{\omega/t}}$ there is a Dunkl-Opdam operator $D_{v/t}\in G\ltimes\D(\mathscr{B}_{\omega/t})$.
   \end{enumerate}
\end{defi}
In this situation, we can define the following algebras:
\begin{defi}
    We define the following algebras:
\begin{enumerate}[label=(\roman*)]
\item $H_{t,c,\omega}(X,G)_{\mathscr{A}_{\omega/t}} :=G\ltimes \mathcal{D}(\mathscr{B}_{\omega/t})\cap H_{t,c,\omega}(X,G)$.
\item $\mathcal{H}_{t,c,\omega}(X,G)_{\mathscr{A}_{\omega/t}}:=\widehat{H}_{t,c,\omega}(X,G)_{\mathscr{A}_{\omega/t}}\otimes_{\mathcal{R}}K$.  
\end{enumerate}
\end{defi}
A priory, these algebras also depend on $U$. However, it can be shown that they are in fact independent of $U$. The definition of a Cherednik lattice implies that if $\mathscr{A}_{\omega/t}$ is a Cherednik lattice, then $\pi^n\mathscr{A}_{\omega/t}$ is also a Cherednik lattice for each $n$. 
\begin{prop}
There is an identity of topological $K$-algebras:
\begin{equation*}
    \mathcal{H}_{t,c,\omega}(X,G)=\varprojlim_n\mathcal{H}_{t,c,\omega}(X,G)_{\pi^n\mathscr{A}_{\omega/t}}.
\end{equation*}
Furthermore, this identity is a Fréchet-Stein presentation of $\mathcal{H}_{t,c,\omega}(X,G)$.
\end{prop}
\begin{proof}
This was shown in \cite[Section 4.3]{p-adicCheralg}.
\end{proof}
This presentation will be a basic tool in all our developments regarding $p$-adic Cherednik algebras, as it will allow us to perform explicit calculations.
\subsection{Finite group actions on Stein spaces}\label{Section group act on Stein spaces}
The next goal is giving an introduction to the theory of $p$-adic Stein spaces. In particular, we will recall some geometric properties, and provide a geometric analysis of the action of a finite group on a Stein space. For the rest of the section, we fix a finite group $G$.
\begin{defi}\label{defi Stein spaces}
We define the following objects:
\begin{enumerate}[label=(\roman*)]
    \item Let $V\subset U$ be an open immersion of affinoid spaces. We say $V$ is relatively compact in $U$, and write $V \Subset U$, if there is a set of affinoid generators $f_1,\cdots, f_n\in \OX_U(U)$ such that:
    \begin{equation*}
        V\subset \{ x\in U \textnormal{ } \vert \textnormal{ } \vert f_i(x)\vert < 1, \textnormal{ for }1\leq i\leq n\}.
    \end{equation*}
    \item A rigid space $X$ is called a quasi-Stein space if it admits an admissible affinoid cover $(U_i)_{i\in\mathbb{N}}$ such that $U_{i}\subset U_{i+1}$ is a Weierstrass subdomain. 
    \item We say $X$ is a Stein space if, in addition, we have $U_i\Subset U_{i+1}$ for all $i\geq 0$. We call $(U_i)_{i\in\mathbb{N}}$ a Stein covering of $X$, and write $X=\varinjlim_iU_i$.
    \end{enumerate}
\end{defi}
Notice that any Stein space is partially proper. In particular, it is separated. Let $X$ be a Stein space. It follows from the definition that any Stein covering $X=\varinjlim_nX_n$ satisfies that for each $n\geq 1$ there are affinoid generators $f_1,\cdots,f_{r_n}\in \OX_X(X_n)$, and some $\epsilon \in \sqrt{K^*}$, such that  $0< \epsilon <1$ and such that the following holds:
\begin{equation}\label{equation special Stein covering}
    X_{n-1}\subseteq X_n\left(\epsilon^{-1} f_1,\cdots, \epsilon^{-1} f_{r_n}\right).
\end{equation}
If the identity holds, we recover the definition of Stein space given by R. Kiehl in \cite[Definition 2.3]{kiehl1967theorem}. In most of the proofs we will encounter, it will in fact suffice to assume that the identity above does hold. This is because our interest in Stein spaces stems from the fact that the global sections of the sheaves we will be interested in are nuclear Fréchet spaces. Let us recall the definition:
\begin{defi}[{\cite[Definition 1.1]{kiehl analysis}}]
 A bounded map of Banach spaces $f:V\rightarrow W$ is strictly completely continuous (\emph{scc} ) if there is a sequence of morphisms:
 \begin{equation*}
     f_n:V\rightarrow W,
 \end{equation*}
such that $f=\varinjlim_{n\geq 0} f_n$ in operator norm and  $f_n(V)$ is finite for all $n\geq 0$.
\end{defi}
\begin{defi}[{\cite[Definition 5.3]{bode2021operations}}]
Let $V$ be a Fréchet space. We say $V$ is a nuclear Fréchet space if there is a presentation $V=\varprojlim_n V_n$ satisfying the following:
\begin{enumerate}[label=(\roman*)]
    \item The image of $V$ in $V_n$ is dense for each $n\geq 0$.
    \item For each $n\geq 1$ there is a Banach $K$-vector space $F_n$ and a strict surjection:
    \begin{equation*}
        F_n\rightarrow V_n,
    \end{equation*}
    such that the composition $F_n\rightarrow V_n\rightarrow V_{n-1}$ is $\emph{scc}$.
\end{enumerate}
\end{defi}
This notion of nuclear Fréchet space agrees with the one from \cite[Chapter IV]{schneider2013nonarchimedean}. Given an \emph{scc} map $f:V\rightarrow W$, for any Banach space $U$ and any continuous map $g:W\rightarrow U$ the composition $g\circ f:V\rightarrow U$ is still \emph{scc}. For this reason, in most of the proofs below, it will suffice to assume that equality holds in (\ref{equation special Stein covering}).
 We can now start analyzing the properties of an action of a finite group $G$ on a Stein space $X$. 
 \begin{teo}[{\cite{hansen2016quotients}}]\label{teo existence of quotients of rigid spaces by finite groups}
 Let $X$ be a rigid analytic space with a right action by a finite group $G$. Assume that this action satisfies the following property: 
 \begin{equation*}
     \textnormal{(G-Aff)}:= \textnormal{There is an admissible cover of } X \textnormal{ by } G\textnormal{-stable affinoid spaces.}
 \end{equation*}
Let $\operatorname{Rig}_K$ be the category of rigid spaces over $K$. The following hold:
 \begin{enumerate}[label=(\roman*)]
     \item The action of $G$ on $X$ admits a categorical quotient $X/G$ in $\operatorname{Rig}_K$.
     \item The canonical projection:
     \begin{equation*}
         \pi: X\rightarrow X/G,
     \end{equation*}
     is finite. Furthermore, $X$ is affinoid if and only if $X/G$ is affinoid.
     \item If $X=\Sp(A)$, then $X/G=\Sp(A^G)$.
     \item The preimage along the projection $\pi:X\rightarrow X/G$ induces a bijection:
     \begin{equation*}
 \pi^{-1}:\left\{ 
				\begin{array}{c} 
					\textnormal{admissible open }\\ 
					\textnormal{subspaces of } X/G
				\end{array}
\right\} \cong \left\{
				\begin{array}{c}
				 G-\textnormal{stable admissible open}\\
				
				  \textnormal{subspaces of } X
				\end{array}
\right\},
\end{equation*}
     satisfying that for every admissible open $U\subset X/G$ we have $U=\pi^{-1}(U)/G$.
 \end{enumerate}
 If the condition above holds, we say $X$ is a $G$-variety.
 \end{teo}
The first step is showing that the action of any finite group $G$ on $X$ satisfies $(\textnormal{G-Aff})$, and therefore that the quotient space $X/G$ is a rigid space: 
\begin{Lemma}\label{Lemma stein satisfies gaff}
Let $X$ be a Stein space with a  $G$-action. Then $X$ satisfies $(\operatorname{G-Aff})$.   
\end{Lemma}
\begin{proof}
Let $X=\varinjlim X_n$ be a quasi-Stein covering, and choose a point $x\in X$. As $G$ is finite there is some $X_n$ such that the orbit of $x$ under $G$ is contained in $X_n$. Thus, $Y_n=\cap_{g\in G}X_n$ is non-empty, and  $X=\varinjlim_n Y_n$ is an admissible cover of $X$.
\end{proof}
\begin{prop}\label{prop quotients of Stein are Stein}
Let $X=\varinjlim_nX_n$ be a Stein $G$-variety. The following hold:
\begin{enumerate}[label=(\roman*)]
\item There is a Stein covering $X=\varinjlim_n\Tilde{X}_n$ where each $\Tilde{X}_n$ is $G$-invariant.
    \item $X/G$ is a Stein space.   
\end{enumerate}
\end{prop}
\begin{proof}
Let $X=\varinjlim_n X_n$ be a Stein covering of $X$, and define $\Tilde{X}_n=\cap_{g\in G}g(X_n)$ for each $n\geq 0$. By Lemma \ref{Lemma stein satisfies gaff}, for each point $x\in X$ there is a $G$-invariant affinoid open subspace $U\subset X$  such that $x\in U$. As $U$ is affinoid, it is quasi-compact. Hence, it is contained in some $X_n$ for sufficiently high $n$. Thus, we may assume that the $\Tilde{X}_n$ are non-empty for all $n\geq 0$. Furthermore, $X=\cup_{n\geq 0}\Tilde{X}_n$ is an admissible affinoid cover. We just need to show that $X=\varinjlim_n \Tilde{X}_n$ is a Stein covering of $X$. For each $g\in G$ we have $g(X_n)\Subset g(X_{n+1})$. Hence, it follows by \cite[Lemma 6.3.7]{bosch2014lectures} that $\Tilde{X}_n\Subset \Tilde{X}_{n+1}$. Similarly, one shows that $\Tilde{X}_n$ is a Weierstrass subdomain of $\Tilde{X}_{n+1}$, thus showing that  $X=\varinjlim_n \Tilde{X}_n$ is a Stein covering. By Theorem \ref{teo existence of quotients of rigid spaces by finite groups}, we have an admissible affinoid covering $X/G=\varinjlim_n \Tilde{X}_n/G$. Fix some $n\geq 0$. By \cite[Proposition 2.5.2]{berkovich2012spectral}, and \cite[Corollary 2.5.5]{berkovich2012spectral} there are affinoid generators $f_1,\cdots f_{r_n}\in \OX_{\Tilde{X}_n/G}(\Tilde{X}_n/G)$ satisfying that the composition $\Tilde{X}_{n-1}\rightarrow \Tilde{X}_{n}\rightarrow \Tilde{X}_{n}/G$, has image contained in the following admissible open subspace:
\begin{equation*}
   \{ x\in \Tilde{X}_{n}/G \textnormal{ } \vert \textnormal{ } \vert f_i(x)\vert < 1, \textnormal{ for }1\leq i\leq n_r\}. 
\end{equation*}
As $\Tilde{X}_{n-1}\rightarrow \Tilde{X}_{n-1}/G$ is surjective,
it follows that $\Tilde{X}_{n-1}/G\Subset \Tilde{X}_{n}/G$.
\end{proof}
As a consequence of these results, it follows that every Stein $G$-variety admits a finite covering by Stein $G$-varieties satisfying good properties: 
\begin{coro}\label{coro Stein G-invariant coverings}
Let $X$ be a smooth Stein $G$-variety . There is an admissible cover $\{U_i\}_{i=1}^n$ of $X$ such that every $U_i$ satisfies the following:
\begin{enumerate}[label=(\roman*)]
    \item $U_i$ is a $G$-invariant Stein open subspace of $X$ with an étale map $U_i\rightarrow \mathbb{A}^r_K$.
    \item Any reflection hypersurface on $U$ is the vanishing locus of a rigid function.
\end{enumerate}
Moreover, for $1\leq i,j\leq n$, the space $U_{ij}=U_i\cap U_j$ satisfies the conditions above.
\end{coro}
\subsection{\texorpdfstring{$p$-adic Cherednik algebras on smooth Stein spaces}{}}\label{Section p-adic cher alg smooth stein}
The goal now is describing the sections of a sheaf of $p$-adic Cherednik algebras on a smooth Stein $G$-variety $X=\varinjlim_nX_n$, and studying its category of co-admissible modules.
\begin{Lemma}\label{Lemma nuclear sections on Frechet spaces}
Let $f:V_1\rightarrow V_2$, $g:W_1\rightarrow W_2$ be \emph{scc} morphisms of Banach $K$-vector spaces. Then the map $f\widehat{\otimes}_Kg:V_1\widehat{\otimes}_K W_1\rightarrow V_2\widehat{\otimes}_K W_2$ is also \emph{scc}.
\end{Lemma}
\begin{prop}\label{prop sections on Stein spaces}
Let $X=\varinjlim X_n$ be a Stein $G$-variety with an étale map $X\rightarrow \mathbb{A}^r_K$. Choose a sheaf of $p$-adic Cherednik algebras $\mathcal{H}_{t,c,\omega,X,G}$, and define: 
\begin{equation*}
    \mathcal{H}_{t,c,\omega}(X,G):=\Gamma(X/G,\mathcal{H}_{t,c,\omega,X,G}).
\end{equation*}
Then  $\mathcal{H}_{t,c,\omega}(X,G)$ is a Fréchet-Stein algebra and a nuclear Fréchet $K$-space.
\end{prop}
\begin{proof}
Let $\mathcal{A}_{\omega}$ be the Atiyah algebra on $X$ associated to $\omega$. We will assume for simplicity that $t=1$. Let $X_n=\Sp(A_n)$ and $A:=\OX_X(X)$. As we have an étale map $X\rightarrow \mathbb{A}^r_K$, there are rigid functions $f_1,\cdots,f_r\in A$  such that:
\begin{equation*}
    \Omega^1_{X/K}(X)=\bigoplus_{i=1}^rdf_iA.
\end{equation*}
Let $v_1,\cdots,v_r\in \mathcal{T}_{X/K}(X)$ be the dual basis, so that for each $n\geq 0$ we have $\mathcal{T}_{X/K}(X_n)=\bigoplus_{i=1}^rv_iA_n$. For each $n\geq 0$, we define the following $(\mathcal{R}, A^{\circ}_n)$-lattice:
\begin{equation*}
    \mathscr{A}_{\omega,n}=\bigoplus_{i=1}^rv_iA^{\circ}_n\subset \mathcal{A}_{\omega}(X_n).
\end{equation*}
By construction, we have 
$\mathscr{A}_{\omega,n}=A^{\circ}_n\otimes_{A^{\circ}_{n+1}}\mathscr{A}_{\omega,n+1}$, and by \cite[Lemma 4.3.7]{p-adicCheralg}, there is a strictly increasing sequence of positive integers $\{m(n) \}_{n\geq 0}$
such that each $\pi^{m(n)}\mathscr{A}_{\omega,n}$ is a Cherednik lattice of $X_n$. Thus, we have Fréchet-Stein presentations: 
\begin{equation*}
    \mathcal{H}_{1,c,\omega}(X_n,G)=\varprojlim_{m\geq m(n)} \mathcal{H}_{1,c,\omega}(X_n,G)_{\pi^m\mathscr{A}_{\omega,n}}.
\end{equation*}
As the sequence $\{m(n) \}_{n\geq 0}$ is strictly increasing, we have:
\begin{equation*}
    \mathcal{H}_{1,c,\omega}(X,G)=\varprojlim_{n\geq 0}\mathcal{H}_{1,c,\omega}(X_n,G)=\varprojlim_{n\geq 0}\mathcal{H}_{1,c,\omega}(X_n,G)_{\pi^{m(n)}\mathscr{A}_{\omega,n}}.
\end{equation*}
We may assume for simplicity that the Stein covering $X=\varinjlim_n X_n$ satisfies that the identity holds in (\ref{equation special Stein covering}). The general case is analogous. By definition of a Stein covering, the maps $X_{n}\rightarrow X_{n+1}$ are inclusions of Weierstrass subdomains. Hence, the associated maps $A_{n+1}\rightarrow A_n$ have dense image. Furthermore, as $X$ is Stein, for all $n\geq 0$ there is some $\epsilon_n\in \sqrt{\vert K^*\vert}$, with $0< \epsilon_n<1$ and  such that we have the following pushout diagram:
\begin{equation*}
\begin{tikzcd}
A_{n+1} \arrow[d] & {K\langle t_1,\cdots, t_{i_n}\rangle} \arrow[l] \arrow[d] \\
A_n               & {K\langle \epsilon_n^{-1}t_1,\cdots, \epsilon_n^{-1}t_{i_n}\rangle} \arrow[l]                            
\end{tikzcd}
\end{equation*}
where the horizontal maps are surjective and the vertical maps are dense. The rightmost vertical map is \emph{scc} by the proof of \cite[Lemma 3.1.3]{HochschildDmodules2}. By \cite[Corollary 4.3.3]{p-adicCheralg}, for each $n\geq 0$, and $m\geq m(n)$ we have an identity of Banach $A_n$-modules:
\begin{equation*}
    \mathcal{H}_{1,c,\omega}(X_n,G)_{\pi^{m}\mathscr{A}_{\omega,n}}=\bigoplus_{g\in G}A_{n}\widehat{\otimes}_K K\langle \pi^mt_1,\cdots, \pi^m t_r\rangle.
\end{equation*}
Thus, for each $n\geq 0$ we have the following commutative diagram:
\begin{equation*}
\begin{tikzcd}
{\mathcal{H}_{1,c,\omega}(X_{n+1},G)_{\pi^{m(n+1)}\mathscr{A}_{\omega,n+1}}} \arrow[d] & {\bigoplus_{g\in G}K\langle t_1,\cdots, \pi^{m(n+1)}t_r\rangle}\arrow[l] \arrow[d] \\
{\mathcal{H}_{1,c,\omega}(X_n,G)_{\pi^{m(n)}\mathscr{A}_{\omega,n}}}                   & \bigoplus_{g\in G}K\langle \epsilon_n^{-1}t_1,\cdots,  \pi^{m(n)}t_r\rangle \arrow[l]                                                                                          
\end{tikzcd}
\end{equation*}
where the horizontal maps are surjective and the rightmost vertical map is \emph{scc} with dense image by Lemma \ref{Lemma nuclear sections on Frechet spaces}. Hence, $\mathcal{H}_{1,c,\omega}(X,G)$ is a nuclear Fréchet $K$-space.
The fact that the following inverse limit is a Fréchet-Stein presentation:
\begin{equation*}
  \mathcal{H}_{1,c,\omega}(X,G)=\varprojlim_{n\geq 0}\mathcal{H}_{1,c,\omega}(X_n,G)_{\pi^{m(n)}\mathscr{A}_{\omega,n}},  
\end{equation*}
 follows from \cite[Theorem 4.3.15]{p-adicCheralg}, and \cite[Corollary 4.5.5]{p-adicCheralg}.
\end{proof}
\begin{coro}
Let $\mathcal{M}\in \mathcal{C}(\mathcal{H}_{t,c,\omega}(X,G))$. Then $\mathcal{M}$ is a nuclear Fréchet space.
\end{coro}
\begin{proof}
Follows from Proposition \ref{prop sections on Stein spaces} using the arguments in \cite[Proposition 5.5]{bode2021operations}.
\end{proof}
\subsection{\texorpdfstring{The GAGA Functor}{}}\label{Section GAGA}
Now that we have a solid theory of sheaves of $p$-adic Cherednik algebras on smooth $p$-adic Stein spaces, we will put the focus on studying the behavior of these sheaves under the rigid-analytic GAGA functor. Given a locally of finite type $K$-scheme $X$, we denote its rigid analytification by $X^{\operatorname{an}}$. An overview of the basic properties of this functor may be found in  \cite[5.4]{bosch2014lectures}. We remark that if $X$ is affine, then $X^{\operatorname{an}}$ is always a Stein space. Hence, we can apply the machinery developed in the previous section to this setting.
\begin{Lemma}\label{GAGA and quotients}
Let $X=\Spec(A)$ be an affine variety with an action of $G$. Then the induced action of $G$ on  $X^{\operatorname{an}}$ satisfies $\textnormal{(G-Aff)}$. 
\end{Lemma}
\begin{proof}
As the GAGA functor sends a cover of $X$ in the Zariski topology to an admissible cover of $X$, our assumption that  $X$ has a $G$-stable affine cover allows us to reduce to the case where 
$X=\Spec(A)$ is affine. As $X$ is affine, $X^{\operatorname{an}}$ is a Stein space. Hence, the result follows by Lemma \ref{Lemma stein satisfies gaff}.
\end{proof}
Thus, in the situation of the previous lemma, we have a canonical map:
\begin{equation*}
    X^{\operatorname{an}}/G\rightarrow (X/G)^{\operatorname{an}}.
\end{equation*}
As expected, this map is always an isomorphism of rigid spaces:
\begin{Lemma}
 Let $X=\Spec(A)$ be an affine variety with an action of $G$. There is a canonical morphism of rigid spaces:
 \begin{equation*}
     \Phi:X^{\operatorname{an}}/G\rightarrow (X/G)^{\operatorname{an}}.
 \end{equation*}
which satisfies the following properties:
\begin{enumerate}[label=(\roman*)]
    \item It is an isomorphism.
    \item It induces a $K$-linear isomorphism: $\operatorname{H}_{\operatorname{dR}}^2(X)^G\cong\operatorname{H}_{\operatorname{dR}}^2(X^{\operatorname{an}})^G$.
    \item  It induces a $G$-equivariant bijection: $S(X,G) \cong S(X^{\operatorname{an}},G)$.
\end{enumerate}
\end{Lemma}
\begin{proof}
The first claim is a routine calculation using the construction of the GAGA functor. Claim $(ii)$ follows by \cite[Example 1.8.(b)]{grosse2004rham}, and $(iii)$ is a consequence of the GAGA functor preserving dimensions and connected components.   
\end{proof}
As a consequence of this lemma, we may identify $X^{\operatorname{an}}/G$ with $(X/G)^{\operatorname{an}}$, and the GAGA functor induces a morphism of ringed sites:
\begin{equation*}
    (X^{\operatorname{an}}/G,\OX_{X^{\operatorname{an}}/G})\rightarrow (X/G,\OX_{X/G}).
\end{equation*}
By \cite[Section 2.13]{etingof2004cherednik}, sheaves of Cherednik algebras on $X$ are quasi-coherent modules on $X/G$. Hence, for each sheaf of Cherednik algebras $H_{t,c,\omega,X,G}$ on $X$, pullback along this morphism of ringed sites induces a sheaf of Cherednik algebras on $X^{\operatorname{an}}/G$.
Furthermore, by \cite[Section 2.13]{etingof2004cherednik} and \cite[Theorem B]{p-adicCheralg}, the isomorphism classes of Cherednik algebras on $X/G$ and $X^{\operatorname{an}}/G$ are parameterized by the same space. Fix a $K$-variety $X=\Spec(A)$ with an action of $G$, and  an étale map $X\rightarrow \mathbb{A}^r_K$. Choose a sheaf of Cherednik algebras $H_{t,c,\omega,X,G}$ on $X$, and let $H_{t,c,\omega,X^{\operatorname{an}},G}$ be the corresponding sheaf of Cherednik algebras on $X^{\operatorname{an}}$.
\begin{Lemma}\label{Lemma construction of the map to GAGA extension}
For every $G$-invariant open affinoid subspace $U=\Sp(B)\subset X^{\operatorname{an}}$ there is a canonical map of filtered $K$-algebras:
\begin{equation*}
    i_U:H_{t,c,\omega}(X,G)\rightarrow H_{t,c,\omega}(U,G),
\end{equation*}
which satisfies the following properties:
\begin{enumerate}[label=(\roman*)]
    \item $i_U$ induces an identification of $B$-modules: $H_{t,c,\omega}(U,G)=B\otimes_AH_{t,c,\omega}(X,G).$
    \item $i_U:H_{t,c,\omega}(X,G)\rightarrow H_{t,c,\omega}(U,G)$ is two-sided flat.
    \item The map $H_{t,c,\omega}(X,G)\rightarrow \mathcal{H}_{t,c,\omega}(U,G)$ is two-sided flat.
\end{enumerate}
If $\OX_{X^{\operatorname{an}}}(X^{\operatorname{an}})\rightarrow \OX_{X^{\operatorname{an}}}(U)$ is dense, then so are the maps on $(iii)$ and $(iv)$.
\end{Lemma}
\begin{proof}

Consider the following map of filtered $K$-algebras:
\begin{equation*}
    i_U:H_{t,c,\omega}(X,G)\rightarrow G\ltimes \D_{\omega}(X^{\operatorname{reg}})\rightarrow G\ltimes \D_{\omega}(U^{\operatorname{reg}}).
\end{equation*}
We need to show that $i_U$ sends Dunkl-Opdam operators to Dunkl-Opdam operators. We may assume that the $(Y,g)\in S(X,G)$ satisfy that there is  $f_{(Y,g)}\in A$ such that $Y=\mathbb{V}(f_{(Y,g)})$. Hence, for $v\in \mathcal{T}_{X/K}(X)$ there is a standard Dunkl-Opdam operator:
\begin{equation*}
    D^s_v:=tv+\sum_{(Y,g)\in S(X,G)}\frac{2c(Y,g)}{1-\lambda_{Y,g}}\frac{v(f_{(Y,g)})}{f_{(Y,g)}}(g-1).
\end{equation*}
Let $D_v$ be a Dunkl-Opdam operator for $v$. By \cite[Lemma 3.1.7]{p-adicCheralg}, we have $D_v-D^s_v\in G\ltimes A$. Hence, it suffices to show that $i_U(D^s_v)$ is a Dunkl-Opdam operator of $v\in \mathcal{T}_{X^{\operatorname{an}}/K}(U)$, which is clear. Now $(i)$ is a consequence of \cite[Theorem 2.17]{etingof2004cherednik}, and \cite[Theorem 3.3.1]{p-adicCheralg}. Hence, we have an isomorphism of $B$-modules:
\begin{equation*}
    H_{t,c,\omega}(U,G)=B\otimes_AH_{t,c,\omega}(X,G).
\end{equation*}
Thus, claim $(iii)$ holds if $A\rightarrow B$, and this is true by the properties of the GAGA functor. Statement $(iv)$ follows at once from $(iii)$ and  the fact that the map $H_{t,c,\omega}(U,G)\rightarrow \mathcal{H}_{t,c,\omega}(U,G)$ is always two-sided flat (\emph{cf}. \cite[Proposition 4.3.16]{p-adicCheralg}).
\end{proof}
\begin{teo}\label{teo faithfully flat map to analytification}
Let $X=\Spec(A)$ be a smooth $G$-variety. Let $H_{t,c,\omega,X,G}$ be a sheaf of Cherednik algebras on $X/G$. There is a canonical morphism of $K$-algebras:
\begin{equation*}
    H_{t,c,\omega}(X,G)\rightarrow \mathcal{H}_{t,c,\omega}(X^{\operatorname{an}},G),
\end{equation*}
which is faithfully flat and has dense image.
\end{teo}
\begin{proof}
 We may assume without loss that there is an étale map $X\rightarrow \mathbb{A}^r_K$. Furthermore, as étale maps are preserved by the GAGA functor  (\emph{cf}. \cite[Theorem 5.2.1]{Conrad1999}), there is also an étale map $X^{\operatorname{an}}\rightarrow \mathbb{A}^r_K$. As $X^{\operatorname{an}}$
is a Stein space, we may use Proposition \ref{prop quotients of Stein are Stein} to obtain a Stein covering $X^{\operatorname{an}}=\varinjlim_n X_n$ such that each of the $X_n$ is a $G$-invariant affinoid subdomain. Let $X_n=\Sp(A_n)$ for all $n\geq 0$. Following the strategy of the proof of Proposition \ref{prop sections on Stein spaces}, we can choose a sequence $\{m(n) \}_{n\geq 0}$  and a family $(\mathscr{A}_{\omega,n})_{n\geq 0}$
such that we have a Fréchet-Stein presentation:
\begin{equation}\label{equation flatness of Cher alg and GAGA}
    \mathcal{H}_{t,c,\omega}(X^{\operatorname{an}},G)=\varprojlim_n\mathcal{H}_{t,c,\omega}(X_n,G)=\varprojlim_{n\geq 0}\mathcal{H}_{t,c,\omega}(X_n,G)_{\pi^{m(n)}\mathscr{A}_{\omega,n}}.
\end{equation}
By Lemma \ref{Lemma construction of the map to GAGA extension}, for each $n\geq 0$ we have maps:
\begin{equation*}
    j_n:H_{t,c,\omega}(X,G)\xrightarrow[]{i_{X_n}}H_{t,c,\omega}(X_n,G)\rightarrow \mathcal{H}_{t,c,\omega}(X_n,G).
\end{equation*}
As these maps commute with the transition maps in $(\ref{equation flatness of Cher alg and GAGA})$, we obtain a map:
\begin{equation*}
    j:H_{t,c,\omega}(X,G)\rightarrow \mathcal{H}_{t,c,\omega}(X^{\operatorname{an}},G).
\end{equation*}
By part $(iii)$
in Lemma \ref{Lemma construction of the map to GAGA extension}, and the fact that $\mathcal{H}_{t,c,\omega}(X^{\operatorname{an}},G)$ has the inverse limit topology, it follows that $j$ has dense image.  Let $j^*$ denote the extension of scalars along $j$. As $H_{t,c,\omega}(X,G)$ is noetherian, it suffices to show that $j^*$ preserves short exact sequences of finitely presented modules.  By \cite[Corollary 3.4]{schneider2002algebras}, every finitely presented module over  $\mathcal{H}_{t,c,\omega}(X^{\operatorname{an}},G)$ is co-admissible.  Thus, 
it suffices to show that  the $j_n:H_{t,c,\omega}(X,G)\rightarrow \mathcal{H}_{t,c,\omega}(X_n,G)$ are flat. This was shown in Lemma \ref{Lemma construction of the map to GAGA extension}. For faithful flatness, let $\mathcal{M}$ be a finitely presented $H_{t,c,\omega}(X,G)$-module and assume that $j^*\mathcal{M}=0$. By \cite[Proposition 4.4.18]{p-adicCheralg}, the map $H_{t,c,\omega}(X_n,G)\rightarrow \mathcal{H}_{t,c,\omega}(X_n,G)$ is faithfully flat for each $n\geq 0$. Thus, $j^*\mathcal{M}=0$ if and only if 
$i_{X_n}^*\mathcal{M}=0$ for all $n\geq 0$. However, by $(ii)$ in Lemma \ref{Lemma construction of the map to GAGA extension} we have
$i_{X_n}^*\mathcal{M}=A_n\otimes_A\mathcal{M}=0$. We will now show that this implies $\mathcal{M}=0$. If $\mathcal{M}\neq 0$, then there is a non-zero finitely generated $A$-module $\mathcal{N}\subset \mathcal{M}$. Consider the Fréchet-Stein presentation  $B:=\OX_{X^{\operatorname{an}}}(X^{\operatorname{an}})=\varprojlim_n A_n$. Then we have:
\begin{equation*}
    B\otimes_A\mathcal{N}=\left(\varprojlim A_n\right)\otimes_A\mathcal{N}=\varprojlim \left(A_n\otimes_A\mathcal{N}\right)=0,
\end{equation*}
where the second identity follows from the fact that $B\otimes_A\mathcal{N}$ is a co-admissible $B$-module. As the map $A\rightarrow B$ is faithfully flat, this implies $\mathcal{N}=0$, a contradiction.
\end{proof}
\begin{coro}
$\D_X(X)\rightarrow \wideparen{\D}_{X^{\operatorname{an}}}(X^{\operatorname{an}})$ is faithfully flat with dense image.
\end{coro}
\subsection{Arens-Michael envelopes}\label{Section AM envelopes}
In the previous section, we investigated the algebraic properties of the canonical map:
\begin{equation*}
    H_{t,c,\omega}(X,G)\rightarrow \mathcal{H}_{t,c,\omega}(X^{\operatorname{an}},G).
\end{equation*}
In this section, we will worry about the analytical properties of this extension. In particular, we will study the Arens-Michael envelopes of Cherednik algebras on smooth affine $K$-varieties. A complete treatment of Arens-Michael envelopes can be found in the book by Helemskii \cite{Helemskiibook}. These envelopes where first defined by J. Taylor in \cite{TAYLOR1972183}, but the terminology used there is different.\\
An Arens-Michael $K$-algebra is a locally convex $K$-algebra
topologically isomorphic to a projective limit of $K$-Banach algebras. Given a locally convex $K$-algebra $A$, let $\mathfrak{M}(A)$  be the set of continuous and sub-multiplicative seminorms on $A$. For every $q\in \mathfrak{M}(A)$, the completion $\widehat{A}_q$ of $A$ with respect to $q$ is a Banach $K$-algebra, and we have a morphism of locally convex $K$-algebras with dense image:
\begin{equation*}
    A\rightarrow \widehat{A}:=\varprojlim_{q\in\mathfrak{M}(A)}A_q.
\end{equation*}
The algebra $\widehat{A}$ is called the Arens-Michael envelope of $A$. It satisfies the following universal property: Let $B$ be a Banach $K$-algebra, and $f:A\rightarrow B$ be a continuous map of locally convex $K$-algebras. There is a unique continuous morphism of $K$-algebras $\widehat{f}:\widehat{A}\rightarrow B$ satisfying that the following diagram is commutative:
\begin{equation*}
\begin{tikzcd}
\widehat{A} \arrow[rd, "\widehat{f}"] &   \\
A \arrow[r, "f"] \arrow[u]            & B
\end{tikzcd}
\end{equation*}
In particular, the categories of continuous $K$-linear representations of $A$ and $\widehat{A}$ on Banach spaces are canonically isomorphic. This particular feature was used in \cite{PirStabflat}, and \cite{schmidt2012stableflatnessnonarchimedeanhyperenveloping} to study the representation theory of a finite-dimensional Lie algebra $\mathfrak{g}$ via the Arens-Michael envelope $\widehat{U}(\mathfrak{g})$.\\
Let $V$ be a $K$-vector space. We may regard $V$ as a locally convex $K$-vector space by considering the finest locally convex topology on $V$. This topology is defined by letting a basis of open neighborhoods of zero be given by all $\mathcal{R}$-submodules
$\mathcal{V}\subset V$ such that $V=\mathcal{V}\otimes_{\mathcal{R}}K$. This topology satisfies that any $K$-linear map $V\rightarrow W$ to a locally convex $K$-vector space is continuous. Furthermore, given two $K$-vector spaces $V$ and $W$ with the finest locally convex topology, their projective tensor product $V\otimes_{K,\pi} W$ (cf. \cite[Chaper IV 17.B]{schneider2013nonarchimedean}) also has the finest locally convex topology. Hence, given a $K$-algebra $A$, we can regard $A$ as a locally convex $K$-algebra with respect to this topology. Thus, we may form its Arens-Michael envelope $\widehat{A}$, and we call it the Arens-Michael envelope of $A$.
 \begin{teo}\label{teo GAGA and Arens-Michaels envelopes}
Let $X=\Spec(A)$ be a $G$-variety with an étale map $X\rightarrow \mathbb{A}^r_K$, and let $H_{t,c,\omega,X,G}$ be a sheaf of Cherednik algebras on $X$. The  map of $K$-algebras:
\begin{equation*}
    H_{t,c,\omega}(X,G)\rightarrow \mathcal{H}_{t,c,\omega}(X^{\operatorname{an}},G),
\end{equation*}
realizes
$\mathcal{H}_{t,c,\omega}(X^{\operatorname{an}},G)$ as the Arens-Michael envelope of $H_{t,c,\omega}(X,G)$.
 \end{teo}
 \begin{proof}
Let $t=1$, and  $X^{\operatorname{an}}=\varinjlim_{n\geq 0} X_n$ be a Stein covering by $G$-invariant affinoid subdomains with $X_n=\Sp(A_n)$.  As in the proof of Theorem \ref{teo faithfully flat map to analytification}, we choose families $(m(n))_{n\geq 0}$ and $(\mathscr{A}_{\omega,n})_{n\geq 0}$ such that we have a Fréchet-Stein presentation:
\begin{equation}\label{equation FS presentation in A-M proof}
    \mathcal{H}_{1,c,\omega}(X^{\operatorname{an}},G)=\varprojlim_{n\geq 0}\mathcal{H}_{1,c,\omega}(X_n,G)_{\pi^{m(n)}\mathscr{A}_{\omega,n}}.
\end{equation}
Thus, $\mathcal{H}_{1,c,\omega}(X^{\operatorname{an}},G)$ is Arens-Michael, and $j:H_{1,c,\omega}(X,G)\rightarrow \mathcal{H}_{1,c,\omega}(X^{\operatorname{an}},G)$ is a map of $K$-algebras with dense image. We need to show that this map satisfies the universal property of Arens-Michael envelopes. Let $q$ be a continuous sub-multiplicative norm on $H_{1,c,\omega}(X,G)$, and $B_q$ be the associated completion. We need to show that there is some $n\geq 0$ such that the map  $\psi:H_{1,c,\omega}(X,G)\rightarrow B_q$ factors as the following composition of continuous maps of $K$-algebras:
\begin{equation*}
    H_{1,c,\omega}(X,G)\rightarrow\mathcal{H}_{1,c,\omega}(X_n,G)_{\pi^{m(n)}\mathscr{A}_{\omega,n}}\rightarrow B_q.
\end{equation*}
By construction, we have an identification:
\begin{equation*}
    \mathcal{H}_{1,c,\omega}(X_n,G)_{\pi^{m(n)}\mathscr{A}_{\omega,n}}=\widehat{H}_{1,c,\omega}(X_n,G)_{\pi^{m(n)}\mathscr{A}_{\omega,n}}\otimes_{\mathcal{R}}K.
\end{equation*}
Hence, it suffices to show that for some $n\geq 0$  we have  a map of $\mathcal{R}$-algebras:
\begin{equation*}
   \Tilde{\psi}: H_{1,c,\omega}(X_n,G)_{\pi^{m(n)}\mathscr{A}_{\omega,n}}\rightarrow B^{\circ}_q, \textnormal{ such that  }\Tilde{\psi}\otimes_K \operatorname{Id}=\psi.
\end{equation*}

Let $\{f_1,\cdots,f_s\}$ be a set of elements which generate $A$ as a $K$-algebra. Notice that the map $A\rightarrow A_n$ has dense image. Hence, the set $\{f_1,\cdots,f_s\}$ is a set of topological generators of $A_n$
as a Banach $K$-algebra. As $\{f_1,\cdots,f_s\}$ is a finite set, there is a minimal non-negative integer $s(n)\geq 0$ such that:
\begin{equation*}
    \{\pi^{s(n)}f_1,\cdots,\pi^{s(n)}f_s\}\subset A_n^{\circ}.
\end{equation*}
For simplicity, we may assume that $s(n)=n$, and let $\mathcal{A}_n$ be the image of the map:
\begin{equation*}
    \varphi:\mathcal{R}\langle t_1,\cdots, t_s\rangle\rightarrow A_n^{\circ}, \textnormal{ }t_i\mapsto \pi^nf_i.
\end{equation*}
Then $\mathcal{A}_n$ is an admissible $\mathcal{R}$-algebra, and $\mathfrak{X}_n=\Spf(\mathcal{A}_n)$ is an affine formal model of $X_n$. In particular, as $K$ is discretely valued, $A^{\circ}_n$ is a finite module over $\mathcal{A}_{n}$  (\emph{cf}. \cite[6.4.1 Corollary 6]{BGR}). As the sequences $\{\pi^nf_i\}_{n\geq 0}$ converge to zero for all $1\leq i\leq s$, it follows that there is some non-negative integer $N$ such that for all $n\geq N$ we have $\{\pi^nf_1,\cdots, \pi^nf_s \}\subset B^{\circ}_q$. Hence, as $\mathfrak{X}_n$ is an affine formal model of $X_n$, the map $A\rightarrow B_q$ extends to a continuous morphism of Banach $K$-algebras $\varphi_n:A_n\rightarrow B_q$, which satisfies $\varphi(\mathcal{A}_n)\subset B_q^{\circ}$. As $B^{\circ}_q$ is integrally closed, and $A^{\circ}_n$ is integral over $\mathcal{A}_n$, it follows that $\varphi_n(A_n^{\circ})\subset B^{\circ}_q$. Furthermore, as  $G$ is a finite group, all its elements have finite order. Hence, we have $G\subset B_q^{\circ}$. A fortiori, it follows that $\varphi_n(G\ltimes A_n^{\circ})\subset B_q^{\circ}$ for all $n\geq N$.\\
For each $n\geq 0$, let $\mathscr{T}_n$ be the image of the anchor map $\mathscr{A}_{\omega,n}\rightarrow \mathcal{T}_{X/K}(X_n)$. Let $v_1,\cdots,v_r$ be a set of generators of $\pi^{m(N)}\mathscr{T}_N$ as a finite $A^{\circ}_N$-module, and let $D_{v_1},\cdots, D_{v_r}$ be a choice of Dunkl-Opdam operators in  
$H_{1,c,\omega}(X_N,G)_{\pi^{m(N)}\mathscr{A}_{\omega,N}}$.  Again, the set $\{D_{v_1},\cdots, D_{v_r}\}$ is finite. Hence, there is some  $s\geq 0$  such that $\{\pi^sD_{v_1},\cdots, \pi^sD_{v_r}\}\subset B_q^{\circ}$. By $(ii)$ in \cite[Proposition 3.1.8]{p-adicCheralg}, and the fact that $\varphi_N(G\ltimes A^{\circ}_N)\subset B_q^{\circ}$, it follows that we have a morphism of $\mathcal{R}$-algebras: 
\begin{equation*}
    \Tilde{\psi}:H_{1,c,\omega}(X_N,G)_{\pi^{m(N)+s}\mathscr{A}_{\omega,N}}\rightarrow B^{\circ}_q, \textnormal{ such that }\Tilde{\psi}\otimes_K\operatorname{Id}=\psi.
\end{equation*}
The right hand side is $\pi$-adically complete, so we obtain an extension:
\begin{equation*}
    \mathcal{H}_{1,c,\omega}(X_N,G)_{\pi^{m(N)+s}\mathscr{A}_{\omega,N}}\rightarrow B_q.
\end{equation*}
Thus, for big enough $n\geq 0$, we get a map $\mathcal{H}_{1,c,\omega}(X_n,G)_{\pi^{m(n)}\mathscr{A}_{\omega,n}}\rightarrow B_q$, as wanted.
 \end{proof}
As an upshot of the proof of Theorem \ref{teo GAGA and Arens-Michaels envelopes}, we obtain the following:
\begin{coro}
Let $X$ be an affine finite type $K$-variety and let $X^{\operatorname{an}}$ be its rigid analytification. Then $\OX_{X^an}(X^{\operatorname{an}})$ is the Arens-Michael envelope of $\OX_X(X)$. 
\end{coro}
The corresponding fact for complex varieties was shown in \cite[Section 5]{TAYLOR1972183}.
\section{\texorpdfstring{Category $\wideparen{\OX}$ for $p$-adic rational Cherednik algebras}{}}\label{Section category O for p-adic rat Cher}

\subsection{Rational Cherednik algebras}\label{intro rational Cher algebras}
Let us start by fixing some terminology: 
Let $\mathfrak{h}$ be a finite-dimensional $K$-vector space, and $G\subset \operatorname{Gl}(\mathfrak{h})$ be a finite group. By replacing $K$ by a finite field extension, we may assume that $K[G]$ is split semi-simple over $K$. As $G$ is finite,  it is contained in a maximal compact open subgroup of $\operatorname{Gl}(\mathfrak{h})$. As all such groups are conjugate to $\operatorname{Gl}_n(\mathcal{R})$, we can choose a basis $\underline{t}:=(t_1,\cdots,t_n)$ of $\mathfrak{h}$ such that $G\subset \operatorname{Gl}_n(\mathcal{R})$. Let $\underline{t}^*:=(t_1^*,\cdots,t_n^*)$ be the dual basis of $\mathfrak{h}^*$. In this situation, we get identifications: 
$\mathfrak{h}=\Spec(K[\underline{t}^*])$, and $\mathfrak{h}^*=\Spec(K[\underline{t}])$.\\
Let $S(G):=S(G,\mathfrak{h})$ be the set of reflection hypersurfaces of the action of $G$ on $\mathfrak{h}$. As the action of $G$ on $\mathfrak{h}$ is linear, each element of $G$ fixes at most one hyperplane. Thus,  $S(G)\subset G$, and we call $S(G)$ the set of pseudo-reflections in $G$. For every $g\in S(G)$, we let $\lambda_g$ be its only non-trivial eigenvalue. For the rest of the section, we fix a reflection function $c\in \operatorname{Ref}(\mathfrak{h},G)$. \\
Given $v\in \mathfrak{h}$, $x\in \mathfrak{h}^*$, we denote the evaluation of $x$ at $v$ by $(v,x)$. As each $Y_g$ is a hyperplane, we can choose linear forms $\alpha_g\in \mathfrak{h}^*$ such that $Y$ is the kernel of $\alpha_g:\mathfrak{h}\rightarrow K$. The set of such forms is a one dimensional $K$-vector space inside $\mathfrak{h}^*$. We choose  $\alpha_g^{\vee}\in \mathfrak{h}$ to be the unique eigenvector satisfying $(\alpha_g^{\vee},\alpha_g)=2$. Now that the notation has been fixed, we can finally define rational Cherednik algebras:
\begin{defi}\label{presentation rational cherednik algebras}
 The rational Cherednik algebra $H(\mathfrak{h},G)_{c}$ is the quotient of $G\ltimes T(\mathfrak{h}\oplus \mathfrak{h}^*)$ by the two-sided ideal generated by the following elements:
\begin{enumerate}[label=(\roman*)]
    \item $[v,w]$, where $v,w\in \mathfrak{h}$.
    \item $[x,y]$, where $x,y\in \mathfrak{h}^*$.
    \item $[v,x]-(v,x)+\sum_{g\in S(G)}c(g)(v,\alpha_g)(\alpha_g^{\vee},x)g$, where $x\in\mathfrak{h}^*$ and $v\in \mathfrak{h}$.
\end{enumerate}
\end{defi}
For the rest of the section we fix a rational Cherednik algebra $\textnormal{H}(\mathfrak{h},G)_{c}$ on $\mathfrak{h}$.
\begin{obs}
Let us make a few comments on the definition:
    \begin{enumerate}[label=(\roman*)]
        \item  $H(\mathfrak{h},G)_{c}$ only depends on $c\in \operatorname{Ref}(\mathfrak{h},G)$, and not on the choices of $\alpha_g$.
        \item If $c=0$ or $S(G)$ is empty, then $H(\mathfrak{h},G)_{c}\simeq  \D(\mathfrak{h})$.
    \end{enumerate}
\end{obs}
Hence, rational Cherednik algebras provide generalizations of Weyl algebras which contain the group algebra $K[G]$ as a subalgebra. This is a relevant feature, as it will allow us to construct $H(\mathfrak{h},G)_{c}$-modules out of $K$-linear $G$-representations. Additionally, we point out that this definition is isomorphic to the one given in Section \ref{Section GAGA for cherednik algebras} in terms of Dunkl-Opdam operators.\\
Rational Cherednik algebras share many properties with Weyl algebras. In particular, there is a canonical filtration on $H(\mathfrak{h},G)_{c}$, known as the Dunkl-Opdam filtration. This filtration satisfies a version of the PBW Theorem. In particular, there is a canonical isomorphism of graded $K$-algebras:
\begin{equation*}
    G\ltimes \operatorname{Sym}_K(\mathfrak{h}\oplus\mathfrak{h}^*)\rightarrow \gr H(\mathfrak{h},G)_{c}.
\end{equation*}
As a consequence, there is a decomposition of $K[\underline{t}^*]$-modules:
\begin{equation}\label{equation triangular decomposition rational Cher algebra}
    H_{c}(\mathfrak{h},G)=K[\underline{t}^*]\otimes_K K[G] \otimes_K K[\underline{t}].
\end{equation}
As the reader will undoubtedly already suspect, this decomposition hints at the fact that rational Cherednik algebras admit a triangular decomposition. 
\begin{teo}[{\cite[Section 3.2]{ginzburg2003category}}]\label{teo decomposition of a Cherednik algebra}
The rational Cherednik algebra $H_{c}(\mathfrak{h},G)$ admits a triangular decomposition, with graded subalgebras  $A=K[\underline{t}^*]$, $B=K[\underline{t}]$, $H=K[G]$, and grading element:
\begin{equation*}
    \partial = \sum_{i=1}^nt^*_it_i +\frac{\operatorname{dim}(\mathfrak{h})}{2}- \sum_{g\in S(G)}\frac{2c(g)}{1-\lambda_g}g.
\end{equation*}    
\end{teo}
The element $\partial$ is called the deformed Euler vector. By Definition \ref{defi notions of triang decomp}, we have:
\begin{equation*}
    \partial_0 = \frac{\operatorname{dim}(\mathfrak{h})}{2}- \sum_{g\in S(G)}\frac{2c(g)}{1-\lambda_g}g.
\end{equation*}
Thus, for $W\in\operatorname{Irr}(G)$, the scalar $c(W)\in K$ is the scalar by which $\partial_0$ acts on $W$. 
\begin{defi}
We call the triangular decomposition of $H_c(\mathfrak{h},G)$ described above the canonical decomposition of $H_c(\mathfrak{h},G)$.
\end{defi}
Thus, for any $c\in\operatorname{Ref}(G)$ there is a category $\OX_c$ associated to the triangular decomposition of $H_c(\mathfrak{h},G)$. In particular, $\OX_c$ is a highest weight category, and its irreducible objects are in one to one correspondence with the irreducible $K$-linear $G$-representations. Furthermore, for every $W\in \operatorname{Irr}(G)$, the module $L(W)$
is a finite $K[\underline{t}^*]$-module. In particular, it induces a coherent sheaf on $\mathfrak{h}$. Letting $\mathcal{L}(W)$ denote this coherent sheaf, we obtain a new invariant for $W$. Namely, the support of $\mathcal{L}(W)$ in $\mathfrak{h}$. The study of this invariant in the complex case has been the subject of intensive research in past years. A sample of this may be found in \cite{bezrukavnikov2009parabolic} and \cite{primitive ideals}.

\subsection{\texorpdfstring{$p$-adic rational Cherednik algebras}{}}\label{Section p-adic rational Cher algebra} Keeping the notation and choices of the previous section, let $\mathfrak{h}^{\operatorname{an}}$ be the rigid analytification of $\mathfrak{h}$. In this section we will show that the $p$-adic rational Cherednik algebra:
\begin{equation*}
\mathcal{H}_c(\mathfrak{h}^{\operatorname{an}},G):=\Gamma(\mathfrak{h}^{\operatorname{an}}/G,\mathcal{H}_{c,\mathfrak{h}^{\operatorname{an}},G}),
\end{equation*}
admits a triangular decomposition  in the sense of Fréchet-Stein algebras. By Theorem \ref{teo GAGA and Arens-Michaels envelopes}, the $p$-adic rational Cherednik algebra $\mathcal{H}_c(\mathfrak{h},G)$ satisfies that the map:
\begin{equation*}
  i_c:H_{c}(\mathfrak{h},G)\rightarrow\mathcal{H}_c(\mathfrak{h}^{\operatorname{an}},G),
\end{equation*}
is faithfully flat, has dense image, and realizes $\mathcal{H}_c(\mathfrak{h}^{\operatorname{an}},G)$ as the Arens-Michael envelope of $H_{c}(\mathfrak{h},G)$.  Hence, every  Banach space representation of $H_{c}(\mathfrak{h},G)$ lifts to a continuous Banach space representation of $\mathcal{H}_c(\mathfrak{h}^{\operatorname{an}},G)$. As any such representation factors through one of the Banach algebras in a Fréchet-Stein presentation of $\mathcal{H}_c(\mathfrak{h}^{\operatorname{an}},G)$, it follows that every $K$-linear representation of $H_{c}(\mathfrak{h},G)$ on a Banach space lifts to a representation of a Banach algebra with a triangular decomposition.
\begin{obs}
To simplify notation, we will write $\mathfrak{h}$ for both the algebraic variety and the rigid  space. The context will make clear which of the two we mean.
\end{obs}
As shown in Proposition \ref{prop sections on Stein spaces}, $\mathcal{H}_c(\mathfrak{h},G)$ is  a Fréchet-Stein algebra. In order to show that it admits a triangular decomposition, we will now give a  Fréchet-Stein presentation of $\mathcal{H}_c(\mathfrak{h},G)$:  We start by introducing the following affinoid spaces:
 \begin{equation*}
     \mathbb{B}^n(m)=\Sp(K\langle \pi^m\underline{t}^* \rangle ), \textnormal{ } \Tilde{\mathbb{B}}^n(m)=\Sp(K\langle \pi^m\underline{t} \rangle ),
 \end{equation*}
corresponding to the closed disks of radius $\vert \pi\vert^{-m}$ in $\mathfrak{h}$ and $\mathfrak{h}^*$ respectively. We denote the affinoid algebras associated to these spaces by:
\begin{multline*}
    \mathscr{A}(m)=:K\langle \pi^m\underline{t}^* \rangle=K\langle \pi^mt^*_1,\cdots,\pi^mt^*_n\rangle,\\ \mathscr{B}(m):=K\langle \pi^m\underline{t} \rangle=K\langle \pi^mt_1,\cdots,\pi^mt_n\rangle.
\end{multline*}
As $\mathbb{B}^n(m)$ and $\Tilde{\mathbb{B}}^n(m)$ are $G$-invariant, we have $G$-invariant Stein coverings:
\begin{equation*}
    \mathfrak{h}=\varinjlim_m \mathbb{B}^n(m), \textnormal{ } \mathfrak{h}^{*}=\varinjlim_m \Tilde{\mathbb{B}}^n(m).
\end{equation*}

The vectors $t_1,\cdots,t_n \in\mathfrak{h}$ form a basis of $\mathcal{T}_{\mathfrak{h}/K}(\mathfrak{h})$ as an $\OX_\mathfrak{h}(\mathfrak{h})$-module. Hence, we can follow the procedure in Proposition \ref{prop sections on Stein spaces}, and define the $\mathscr{A}(m)$-module:
\begin{equation*}
    \mathcal{A}_m:=\bigoplus_{i=1}^nt_i\mathcal{R}\langle \pi^m\underline{t}^*\rangle. 
\end{equation*}
We can also choose a strictly increasing sequence $\{r(m)\}_{m\geq 0}\subset \mathbb{Z}^{\geq 0}$ such that $\pi^{r(m)}\mathcal{A}_m$ is a Cherednik lattice of $\mathbb{B}^n(m)$ for all $m\geq 0$. 
\begin{defi}
Let $\mathcal{H}(m)_c$ be the $\mathcal{R}$-algebra generated inside $G\ltimes \D(\mathbb{B}^n(m)^{\operatorname{reg}})$ by $\mathcal{R}\langle \pi^m\underline{t}^*\rangle$, $G$, and the following family of Dunkl-Opdam operators:
\begin{equation*}
    D_{\pi^{r(m)}t_i}= \pi^{r(m)}t_i+ \sum_{g\in S(G)}\frac{2c(g)}{1-\lambda_{g}}\frac{(\alpha_g,\pi^{r(m)}t_i)}{\alpha_g}(g-1), \textnormal{ for } 1\leq i\leq n.
\end{equation*}
\end{defi}
Let $\widehat{\mathcal{H}}(m)_c$ be the $\pi$-adic completion of $\mathcal{H}(m)_c$, and $\widehat{\mathcal{H}}(m)_{c,K}=\widehat{\mathcal{H}}(m)_{c}\otimes_{\mathcal{R}}K$. In this situation, each $\widehat{\mathcal{H}}(m)_{c,K}$ is a two-sided noetherian Banach $K$-algebra, and we have a Fréchet-Stein presentation:
\begin{equation*}
    \mathcal{H}_c(\mathfrak{h},G)=\varprojlim_{m\geq 0}\widehat{\mathcal{H}}(m)_{c,K}.
\end{equation*}
The first step towards showing that $\mathcal{H}_c(\mathfrak{h},G)$ admits a triangular decomposition is showing that each $\widehat{\mathcal{H}}(m)_{c,K}$ admits a triangular decomposition as a Banach algebra. 
\begin{Lemma}
There is a canonical decomposition of Banach $\mathscr{A}(m)$-modules:
\begin{equation*}
    \widehat{\mathcal{H}}(m)_{c,K}=\mathscr{A}(m)\widehat{\otimes}_KK[G]\widehat{\otimes}_K\mathscr{B}(r(m)).
\end{equation*}   
\end{Lemma}
\begin{proof}
By \cite[Corollary 4.4.4]{p-adicCheralg}, we have an isomorphism of $G\ltimes \mathcal{R}\langle\pi^m\underline{t}^* \rangle$-modules: 
\begin{equation*}
    \mathcal{H}(m)_{c}=G\ltimes \mathcal{R}\langle\mathcal{R}\langle\pi^m\underline{t}^* \rangle\otimes_{\mathcal{R}}\mathcal{R}[\pi^{r(m)}\underline{t}].
\end{equation*}
applying $\pi$-adic completion and tensoring by $K$ yields the desired isomorphism.
\end{proof}
\begin{Lemma}
The noetherian Banach $K$-algebra $\widehat{\mathcal{H}}(m)_{c,K}$ admits a triangular decomposition with respect to the following data:
\begin{equation*}
    (H_c(\mathfrak{h},G), K[\underline{t}^*], K[\underline{t}], K[G], \partial).
\end{equation*}
\end{Lemma}
\begin{proof}
The fact that $(K[\underline{t}^*], K[\underline{t}], K[G], \partial)$ is a triangular decomposition of $H_c(\mathfrak{h},G)$ was shown in  Theorem  \ref{teo decomposition of a Cherednik algebra}. By construction, $\widehat{\mathcal{H}}(m)_{c,K}$ is the completion of $H_c(\mathfrak{h},G)$ with respect to a norm. Similarly, 
$\mathscr{A}(m)$ is the completion of $K[\underline{t}^*]$. Furthermore, every element $x\in \mathscr{A}(m)$ is of the form:
\begin{equation*}
    x=\sum_{I\subset \mathbb{Z}^{\geq 0,n}} a_i \underline{t}^{*,I},
\end{equation*}
which shows that $\mathscr{A}(m)$ admits a semi-simple weight space decomposition of finite type with respect to the action of $\operatorname{ad}(\partial)$, and that $\mathscr{A}(m)^{\operatorname{ws}}=K[\underline{t}^*]$. An analogous argument shows that $\mathscr{B}(r(m))$ admits a semi-simple weight space decomposition, and that  $\mathscr{B}(r(m))^{\operatorname{ws}}=K[\underline{t}]$.
\end{proof}
All our results thus far imply the following theorem:
\begin{teo}\label{teo category O p-adic rat Cher}
The $p$-adic rational Cherednik algebra admits a Fréchet-Stein presentation:
\begin{equation*}
    \mathcal{H}_c(\mathfrak{h},G)=\varprojlim_{m\geq 0}\widehat{\mathcal{H}}(m)_{c,K},
\end{equation*}
which admits a triangular decomposition with respect to the following data:
\begin{equation*}
    (H_c(\mathfrak{h},G), K[\underline{t}^*], K[\underline{t}], K[G], \partial).
\end{equation*}
\end{teo}
We may condense the  results of the paper into the following two corollaries:
\begin{coro}
There is an abelian subcategory $\wideparen{\OX}_c\subset \mathcal{C}(\mathcal{H}_c(\mathfrak{h},G))$  satisfying:
\begin{enumerate}[label=(\roman*)]
    \item $\wideparen{\OX}_c$ is closed under closed submodules and finite direct sums.
    \item $\wideparen{\OX}_c$ is a highest weight category. Its irreducible objects are in one-to-one correspondence with the irreducible $K$-linear representations of $G$.
\end{enumerate}
\end{coro}
\begin{coro}
The map $H_c(\mathfrak{h},G)\rightarrow \mathcal{H}_c(\mathfrak{h},G)$ satisfies the following properties:
\begin{enumerate}[label=(\roman*)]
    \item It is faithfully flat and has dense image.
    \item It makes $\mathcal{H}_c(\mathfrak{h},G)$ the Arens-Michael envelope of $H_c(\mathfrak{h},G)$.
    \item It induces an equivalence of highest weight categories:
    \begin{equation*}
        \mathcal{H}_c(\mathfrak{h},G)\otimes_{H_c(\mathfrak{h},G)}-:\OX_c\leftrightarrows \wideparen{\OX}_c: (-)^{\operatorname{ws}}.
    \end{equation*}
\end{enumerate}    
\end{coro}

\end{document}